\documentclass[11pt,leqno]{amsart}
\usepackage[latin1]{inputenc}
\usepackage{dsfont}
\usepackage{amsxtra}
\usepackage{amsmath}
\usepackage{amscd}
\usepackage{amssymb}
\usepackage{amsfonts}
\usepackage[all]{xy}
\usepackage{mathrsfs}
\usepackage{amsthm} 
\usepackage{color}
\usepackage{stmaryrd}
\usepackage{mathabx}
\usepackage{enumitem}
\usepackage[left=2cm, right=2cm, top=2cm]{geometry}

\def\bbA{\mathbb{A}}
\def\bbC{\mathbb{C}}

\def\bbN{\mathbb{N}}

\def\k{\mathbb{Q}}

\def\scrS{\mathscr{S}}

\def\frakg{\mathfrak{g}}

\def\frakl{\mathfrak{l}}

\def\frakC{\mathfrak{C}}
\def\frakD{\mathfrak{D}}

\def\frakh{\mathfrak{H}}

\def\frakN{\mathfrak{N}}

\def\frakp{\mathfrak{P}}

\def\frakS{\mathfrak{S}}

\def\frakU{\mathfrak{U}}

\def\calC{\mathcal{C}}

\def\calE{\mathcal{E}}
\def\calF{\mathcal{F}}
\def\calG{\mathcal{G}}

\def\calO{\mathcal{O}}

\def\calU{\mathcal{U}}

\def\frakg{\mathfrak{g}}
\def\frakh{\mathfrak{h}}

\def\frakl{\mathfrak{l}}

\def\frakn{\mathfrak{n}}
\def\frakp{\mathfrak{p}}

\def\fraku{\mathfrak{u}}

\def\frakgl{\mathfrak{gl}}

\def\1{{\mathbf{1}}}

 \def\AAA{{\mathbb{A}}}
 \def\Ab{{Ab}}
  \def\Ac{\mathcal{A}}
\def\ad{\operatorname{ad}\nolimits}

\def\Aff{{{\mathscr{A\!}f\!f}}}

 \def\an{{\on{an}}}
 \def\Ass{{{\mathcal A}ss}}
 \def\Aut{\operatorname{Aut}\nolimits}

\def\be{\begin{equation}}
\def\Bc{{\mathcal{B}}}

\def\BM{{\on{BM}}}
\def\Bun{{\on{Bun}}}

\def\Cc {{\mathcal {C}}}
\def\CC{{\mathbb{C}}}

\def\Co{{\on{C}}}

\def\Coh{{\on{Coh}}}
\def\coh{{\on{coh}}}

\def\D{\operatorname{D}\nolimits}
\def\DD{{\mathbb{D}}}
\def\del{\partial}

\def\Den{{\mathfrak{D}}}
\def\dg{{\on{dg}}}
\def\dgVect{{\on{dgVect}}}

\def\DK{{\on{DK}}}

\def\dim{\mathrm{dim}}

\def\Ec{\mathcal E}
\def\ee{\end{equation}}

\def\eps{{\varepsilon}}
\def\Eq{{\on{Eq}}}

\def\ev{{\operatorname{ev}\nolimits}}

\def\Ext{\operatorname{Ext}\nolimits}

\def\Fc{{\mathcal F}}
\def\FCoh{{\on{FCoh}}}
\def\FF{{\mathbb {F}}}
\def\FILT{{\on{FILT}}}

\def\Gc{\mathcal{G}}
\def\gen{{\mathfrak {g}}}

\def\gl{{\gen\len}}

\def\Gpd{{\on{Gpd}}}
\def\Gpdg{{\mathfrak{Gpd}}}

\def\h{{\on{h}}}

\def\HH{{\mathbb{H}}}

  \def\hocolim{{  \underrightarrow {\on{holim}} }}
 \def\holim{{  \underleftarrow {\on{holim}} }}
\def\Hom{\operatorname{Hom}\nolimits}

\def\hra{{\, \hookrightarrow}}

\def\Hc{{\mathcal{H}}}

\def\Id{\operatorname{Id}\nolimits}

\def\Ipt{{1\on{pt}}}

\def\k {{\mathbf{k}}}
\def\Kan{{\on{Kan}}}
\def\Kc{\mathcal{K}}
\def\Ker{{\operatorname{Ker}\nolimits}}

\def\Lc{{\mathcal {L}}}

\def\len{{\mathfrak{l}}}

\def\lla{\longleftarrow}

\def\lra{{\longrightarrow}}

\def\Map{{\on{Map}}}

\def\mc{{\mathbf{mc}}}
\def\MC{{\operatorname{MC}\nolimits}}
\def\Mc{{\mathcal{M}}}

\def\Mor{{\on{Mor}}}

\def\Nc{{\mathcal{N}}}
\def\NC{{NC}}
\def\ne{{\mathbf{ne}}}

\def\Ob{\operatorname{Ob}\nolimits}
\def\Oc{{\mathcal O}}

\def\ol{\overline}
\def\on{\operatorname}
\def\oo{{\infty}}
\def\op{{\operatorname{op}\nolimits}}

\def\Pc{\mathcal{P}}
\def\pol{{\on{pol}}}
\def\prim{{\on{prim}}}

\def\pt{{\on{pt}}}
\def\Perv{\operatorname{Perv}\nolimits}

\def\Qc{{\mathcal {Q}}}
\def\qcoh{{\on{qcoh}}}

\def\Qis{{\on{Qis}}}

 \def\QQ{{\mathbb{Q}}}
 \def\Quot{{\Qc uot}}

\def\R{\operatorname{R}\nolimits}

\def\Rc{{\mathcal{R}}}

\def\RHom{\operatorname{RHom}\nolimits}
\def\rk{{\operatorname{rk}\nolimits}}

\def\RR{{\mathbb {R}}}

\def\Sc{\mathcal {S}}
\def\Sch{{\on{Sch}}}

\def\Set{{Set}}
\def\SES{\on{SES}}

\def\Sing{{\on{Sing}}}

\def\Spec{{\operatorname{Spec}\nolimits}}

\def\sset{{\Delta^\circ \Set}}
\def\St{{{St}}}
\def\Stan{{Stan}}
\def\Stang{{\mathfrak{Stan}}}

\def\Stg{{\mathfrak{St}}}
\def\Stg{{\mathfrak St}}

\def\Sym{{\operatorname{Sym}\nolimits}}

\def\tf{{\on{tf}}}

\def\Top{\operatorname{Top}\nolimits}

\def\Tot{{\operatorname{Tot}\nolimits}}

\def\Uc{{\mathcal {U}}}
\def\Uen{{\mathfrak{U}}}

\def\ul{\underline}

\def\Vc{{\mathcal{V}}}
\def\Vect{{\on{Vect}}}

\def\vrk{{\on{vrk}}}

\def\W{{\operatorname{W}\nolimits}}

\def\wt{\widetilde}

\def\X{\operatorname{X}\nolimits}

 \def\ZZ{\mathbb{Z}}

\def\={{\,\simeq \,}}

\newtheorem{theorem}[equation]{\bf{Theorem}}
\newtheorem{thm}[equation]{\bf{Theorem}}
\newtheorem{lemma}[equation]{Lemma}
\newtheorem{lem}[equation]{Lemma}
\newtheorem{cor}[equation]{Corollary}

\newtheorem{proposition}[equation]{Proposition}
\newtheorem{prop}[equation]{Proposition}

\newtheorem{definition}[equation]{Definition}
\newtheorem{Defi}[equation]{Definition}

\theoremstyle{remark} 
\newtheorem{rem}[equation]{Remark}
\newtheorem{rems}[equation]{Remarks}
\newtheorem{remark}[equation]{Remark}
\newtheorem{example}[equation]{Example}
\newtheorem{ex}[equation]{Example}
\newtheorem{exas}[equation]{Examples}

\numberwithin{itemcounter}{subsection}
\numberwithin{equation}{subsection}

\makeatletter
\let\c@equation\c@subsubsection

\let\c@subsubsection\c@subsubsection
\makeatother

\title[The COHA   of a   surface]
{The cohomological Hall algebra  of  a surface and factorization cohomology}

\author{M. Kapranov, E. Vasserot}

\date{}

\begin{document}
\maketitle

\begin{abstract}
For a smooth quasi-projective surface $S$ over $\CC$ we consider the Borel-Moore homology of the
stack of coherent sheaves on $S$ with compact support and make this space into an associative
algebra by a version of the Hall multiplication. This multiplication involves data  (virtual pullbacks)
 governing the derived moduli stack, i.e., the perfect obstruction theory naturally existing on the
 non-derived stack. By restricting to sheaves with support of given dimension,
 we obtain several types of Hecke operators.  In particular, we study  $R(S)$, the Hecke
 algebra of $0$-dimensional
 sheaves. For the case $S=\AAA^2$, we show that $R(S)$ is an enveloping algebra
 and identify it, as a vector space, with the symmetric algebra of an explicit
 graded vector space.
 For a general $S$,  we find the
  graded dimension of
 $R(S)$, using the techniques of factorization cohomology.

 \end{abstract}

\setcounter{tocdepth}{2}

\tableofcontents

\setcounter{section}{-1}

\vfill\eject

\section{Introduction}

\subsection {Motivation}  A large part of the classical theory of automorphic forms 
for $GL_n$ over functional fields can be interpreted in terms of Hall algebras of abelian categories
\cite{kapranov-eis},  \cite{KSV}. Relevant here is
$Coh(C)$, the category of coherent sheaves on a smooth projective curve $C/\FF_q$.  Taking the Hall algebra
of $Bun(C)$, the subcategory of vector bundles, produces (unramified) automorphic forms, while
 $Coh_0(C)$, the category
of torsion sheaves, gives rise to the Hecke algebra.

\smallskip

The classical Hall algebra of a category such as $Coh(C)$ consists of functions on ($\FF_q$-) points of the
moduli stack of objects and so admits various modifications, cf.\cite[Ch.~8]{DyK}. 
Most important is the {\em cohomological Hall algebra}
(COHA) where we take the {\em  cohomology of the stack} instead of  the space of functions on 
the set of its points  \cite{KS}.  This allows us
to work over more general fields such as $\CC$. 

\smallskip

Study of Hall algebras (classical or cohomololgical) of the categories $Coh(S)$ for   varieties $S$ of dimension $d>1$
can be therefore considered as a  higher-dimensional analog of the theory of automorphic forms. In this paper
we consider the case of surfaces ($d=2$)  over $\CC$ and  study their  COHA.  In this case we have a whole
new range of motivations coning from gauge theory, where cohomology of the moduli spaces of instantons is an object
of longstanding interest  \cite{nekrason-inst}, \cite{AGT}, \cite{BFN}. 

\smallskip

\subsection{Description of the results} 
The familiar 2-fold subdivision into \emph{automorphic forms vs. Hecke operators} now becomes 3-fold: we have categories
$Coh_m(S)$, $m=0,1,2$,  of {\em purely $m$-dimensional sheaves}, see \S \ref{subsec:stack-coh}. 
Here, $Coh_2(S)$ consists of vector bundles, while $Coh_0(S)$ is the category of punctual sheaves.  
 An important feature is that the COHA of $Coh_{m-1}(S)$ acts on that of $Coh_m(S)$ by \emph{Hecke operators}. 
 
\smallskip
 
 We denote by $R(S)$ the COHA of the category 
   $Coh_0(S)$. It is  the most immediate analog of
 the unramified Hecke algebra of the classical theory and we relate it to objects studied before. 
 
 \smallskip
 
In the \emph{flat case} $S=\AAA^2$, the algebra $R(\AAA^2)$
  is identified with the direct sum, over $n\geq 0$, of the $GL_n$-equivariant
  Borel-Moore homology of the {\em  commuting varieties}
of $\gl_n$.

\smallskip
 
Our first main result, Theorem \ref{thm:flat}, shows that  
$R(\AAA^2)$ is an enveloping algebra
and is identified, as a graded vector space, with the symmetric algebra of 
an explicit graded vector space $\Theta$. It is convenient
to write $\Theta=H_\bullet^\BM(\AAA^2)\otimes\Theta'$, where the first factor is $1$-dimensional, in homological degree $4$.

\smallskip 
 
 For a general surface $S$, the algebra $R(S)$ 
 is non-commutative.  Our second main result, Theorem \ref{thm:PBW},
 provides a version of Poincar\'e-Birkhoff-Witt theorem for $R(S)$. It exhibits
 a system of generators as well as determines the graded dimension of $R(S)$. More precisely, it establishes
 an isomorphism of graded vector spaces
 \be\label{eq:intro-PBW}
\sigma: \Sym(H_\bullet^\BM(S)\otimes\Theta')   \,\simeq \,  R(S).
 \ee
Like the classical PBW isomorphism for enveloping algebras, $\sigma$ is given by the symmetrized product
map on the space of generators. 


\smallskip

\subsection{Role of factorization algebras}
  Our proof of Theorem  \ref{thm:flat} is based on  the techniques of factorization homology
  \cite{CG}, \cite{ginot}, \cite{GL}, \cite{lurie-HA}. 
    More precisely, we consider the cochain lift $\Rc(S)$ of
 $R(S)$. This can be seen as a homotopy associative   algebra whose cohomology is $R(S)$. 
 For any open set $U\subset S$ we have a similarly defined  algebra $\Rc(U)$.
 Further, one can consider $U$ to be any open set in the complex analytic topology.
 In this case $\Coh_0(U)$ can be considered as an analytic stack and so its Borel-Moore homology and 
 our entire construction of the COHA make sense. 
 
 \smallskip
 
 We prove in Theorem \ref{prop:loc-cst} that
the assignment 
  $U\mapsto \Rc(U)$ is a {\em factorization coalgebra in the category of $E_1$- (i.e., homotopy associative dg-)
  algebras}.  
  This is a reflection of a more fundamental fact:  $U\mapsto \Coh_0(U)$
  is
  a {\em factorization algebra  in the category of analytic stacks}, see Proposition \ref{prop:fact}. 
  These considerations allow us to  lift  $\sigma$ to a morphism of factorization coalgebras
   in the category of  dg-vector spaces
 and deduce the global isomorphism from the local one, i.e., from the case when $S$ is an open ball which
 is equivalent to that of  $S=\AAA^2$. 
 
 \smallskip
 
 In fact, the identification \eqref{eq:intro-PBW} is suggestive of
  {\em non-abelian Poincar\'e duality} (NAPD), compare
  \cite[thm.~5.5.6.6]{lurie-HA}, although it does not seem to be a formal consequence of it.
   NAPD   can be extended to include, for instance, the classical 
    Atiyah-Bott theorem
 on the cohomology of $\Bun_G(\Sigma)$, the moduli stack of  holomorphic $G$-bundles on a compact
 Riemann surface $\Sigma$, see \cite{GL}.  In that latter setting, we have that $H^\bullet(BG)=\Sym(V)$
 (with $V$ being the space of characteristic classes for $G$-bundles), and
 \[
 H^\bullet(\Bun_G(\Sigma)) \,\simeq \, \Sym(H_\bullet(\Sigma)\otimes V). 
 \]

 \subsection {Derived nature of the COHA} \label{subsec:der-nat-coha}
 As a vector space, our COHA is the Borel-Moore
 homology of the Artin stack $\Coh(S)$ (the moduli stack of objects of $Coh(S)$), i.e., it is
  the cohomology of the dualizing complex:
 \[
 H_\bullet^\BM(\Coh(S)) \,=\, H^{-\bullet}(\Coh(S), \omega_{\Coh(S)}). 
 \]
 Since $S$ is a surface, $\Coh(S)$ is singular due to obstructions encoded by $\Ext^2$, so the dualizing complex
 is highly non-trivial. However,  $\Coh(S)$ is in fact a truncation
 of a finer object, the {\em derived moduli stack} $R\Coh(S)$,  smooth in the derived sense,
 see \cite{toen-vezzosi}, \cite{T14}.   While the vector space underlying our COHA depends on $\Coh(S)$ alone,
 the multiplication makes appeal   to the derived structure: we use the refined pullbacks corresponding to the
 perfect obstruction theories on $\Coh(S)$ and  on the related stack of short exact sequences. 
 So our construction has appearance of applying some cohomology theory to the derived stack $R\Coh(S)$ itself
 and using its natural functorialities for morphisms of derived stacks.  
 More recently, this approach has been implemented by M. Porta and F. Sala
 \cite{porta-sala} at the K-theoretical level.

 \subsection{Relation to other work}
 The COHA of a surface that we consider here is a non linear analog of the COHA associated to the preprojective algebra
 of the Jordan quiver considered in \cite{SV13}, see, e.g., \cite{SV18} for the case of arbitrary quivers.
 M. Kontsevich and Y. Soibelman introduced in \cite{KS}  cohomological Hall algebras for 3-dimensional Calabi-Yau
 categories, by taking cohomology of the moduli stack of objects with coefficients in the natural perverse
 sheaf of ``vanishing cycles" with respect to the Chern-Simons functional. Although the details of the approach
 have been worked out only for  quiver-type situations, see, e.g., \cite{davison-comparison} for a comparison with  \cite{SV13},
it seems applicable, in principle,
  to the category of compactly supported
 coherent sheaves on any 3-dimensional Calabi-Yau manifold $M$. In particular, our COHA for a surface $S$
 should be related to the Kontsevich-Soibelman COHA for $M$ 
 the total space of the anticanonical bundle on $S$.  
 
 \smallskip
 
 Instead of Borel-Moore homology of the stack $\Coh(S)$, one can take its Chow groups or its algebraic K-theory,
 in particular, study K-theoretic analogs of the Hecke operators. This approach was developed
 by A. Negut \cite{N17} who studied the K-theoretic effect of explicit Hecke correspondences on the moduli spaces
 and, very recently, by Y. Zhao \cite{zhao} who defined independently the K-theoretic Hall algebra of $0$-dimensional 
 sheaves by a method  similar to ours.
 On the other hand, algebraic K-theory, being a more
 rigid object than homology, does not easily localize on the complex analytic topology  and so determining 
 the size of the resulting objects is  more difficult.

 \smallskip
 
 In the particular case where $S$ is the cotangent bundle to a smooth curve, 
 other versions of the COHA (of 0-dimensional sheaves and of purely 1-dimensional sheaves) of $S$
 appeared recently in \cite{Mi},  \cite{SS}.
 
 \smallskip

 After this paper first appeared on arXiv, there have been some important new developments. Thus, M. Porta and F. Sala \cite{porta-sala}  have defined a categorical
 and a K-theoretical version of the COHA for surfaces, using the derived enhancement
 of the stack of coherent sheaves. Further, A. Khan \cite{khan} introduced a motivic
 framework for Borel-Moore cohomology for Artin stacks which could potentially
 simplify the treatment of some questions considered in this paper.

 \subsection{Structure of the paper} In \S \ref{sec:gen-stacks} we discuss the basic generalities 
 on groupoids and stacks,
 including higher stacks understood as homotopy sheaves of simplicial sets. We pay special attention to
 Dold-Kan and Maurer-Cartan (Deligne) stacks associated to $3$-term complexes and dg-Lie algebras.
 These constructions are used in \S \ref{sec:stack-ext} to describe stacks of extensions (needed for
 defining the Hall multiplication) and filtrations (needed to prove associativity). 
 
\smallskip
 
 In \S \ref{sec:BM} we define and study the Borel-Moore homology of Artin stacks. This concept, which is
 a topological analog of  A. Kresch's  concept of Chow groups for Artin stacks
 \cite{K99}, can be defined easily once we have a good formalism of constructible derived categories and their
 functorialities $f^{-1}, Rf_*, Rf_c, f^!$. While in the ``classical'' approach (sheaves first, complexes later)
 this may present complications, cf. \cite{O07}, \cite{olsson-I} for a discussion, the modern point of
 view of homotopy descent cf. \cite{GR},  allows a straightforward
  definition of the {\em enhanced} derived category of a stack
 as the $\infty$-categorical limit of the corresponding categories for schemes. The desired functorialities
 are also inherited from the case of schemes. We  study virtual pullback in this context. 
 
\smallskip
 
 The COHA is defined in \S \ref{sec:COHA}, first as a vector space, then as an associative
 algebra.
 
\smallskip
 
 In \S \ref{sec:Hecke} we consider subalgebras in the COHA corresponding to sheaves with various
 condition on the dimension of support. These subalgebras play the role of Hecke algebras, since
 they act on other subspaces in COHA (corresponding to sheaves whose dimension of support is
 bigger) by natural ``Hecke operators'' (operators formally dual to those of the Hall multiplication). 
 
\smallskip
 
 In \S  \ref{sec:local-Hec} we study the flat Hecke algebra $R(\AAA^2)$  
 by relating it to the earlier work on commuting varieties in $\gen\len_n$. Here we prove Theorem 
 \ref{thm:flat}. 
 
\smallskip
 
 Finally, in \S \ref{sec:global-Hec} we globalize the consideration of  \S  \ref{sec:local-Hec} by describing the
 global Hecke algebra $R(S)$ as the factorization (co)homology of an appropriate factorization
 (co)algebra.  This leads to the proof of Theorem \ref{thm:PBW} .
 
 \smallskip
 
 The paper has two Appendices. 
  Appendix \ref{sec:appendix},   logically preceding the entire paper, 
   provides a reminder on $\oo$-categories and dg-categories.
  Appendix \ref{sec:hom-can-Euler}
 spells out the homotopy unique nature of  Euler (top Chern) classes and orientation classes
 at the cochain level. It logically depends on  \S \ref{sec:gen-stacks}-\ref{sec:BM}
 (i.e., assumes the formalism of stacks presented in these sections) but precedes
 \S \ref{sec:global-Hec} for which it provides necessary material.

 \subsection{Acknowledgements} We are grateful to B. Feigin, G. Ginot, B. Hennion,  A. Negut, M. Porta, 
 N. Rozenblyum and Y. Soibelman for useful discussions.  We would like to thank  J. Sch\"urrmann for
 pointing out  an inaccuracy in an earlier version, and B. Davison for pointing out to us that the statement about commutativity of $R(\bbA^2)$ in an earlier version is not true. We are also grateful to the referee
 for numerous detailed remarks and suggestions that helped us improve the
 paper.  The research of M.K. was supported by World Premier International Research Center Initiative (WPI Initiative), 
 MEXT, Japan and by the IAS School of Mathematics.
 The research of E.V. was supported by the grant ANR-18-CE40-0024. 
 
 \smallskip
 
 The results of this work have been presented at several conferences: 
 CMND international conference on geometric representation theory and symplectic varieties, 06/2018,
 University of Notre Dame\,;\,
 Vertex algebras, factorization algebras and applications, 07/2018, Kavli IPMU\,;\,
 Vertex algebras and gauge theory, 12/2018, Simons Center\,;\,
 TCRT6, 01/2019, Academia Sinica). 
 We are grateful to the organizers of these events for the invitations to speak.

 \medskip

\section{Generalities on stacks}\label{sec:gen-stacks}

\subsection{Groupoids and simplicial sets}

A {\em groupoid} is a category $G$  in which all morphisms are invertible. 
We write $G= \{G_1 \rightrightarrows G_0\}$ where $G_0 = \Ob(G)$ is the class of 
objects and  $G_1=\Mor(G)$ is the class of morphisms.
For an essentially small groupoid $G$ let $\pi_0(G)$ be the set of 
isomorphisms classes of objects of $G$.
For any object $x\in G_0$ let $\pi_1(G,x) = \Aut_G(x)$ be the automorphism group of $x$. 
All groupoids in the sequel will be assumed essentially small. 

\smallskip

Small groupoids form a 2-category $\Gpdg$. For each groupoids $G_1,$ $G_2$ we have
a groupoid whose objects are functors $G_1\to G_2$ and morphisms are natural 
transformations
of functors. 
We will refer to 1-morphisms of  $\Gpdg$  as simply
{\em morphisms of groupoids}. 
Considered with this notion of morphisms, groupoids form a category which we denote $\Gpd$. 
Let $\Eq\subset\Mor(\Gpd)$ be the class of equivalences of groupoids. 


\begin{prop}\label{prop:pi-gpd-iso}
Let $f: G\to G' $ be a morphism of  groupoids. Suppose that $f$ induces a bijection  of sets $\pi_0(G)\to \pi_0(G')$ 
and, for any
 $x\in\Ob(G)$, an isomorphism of groups $\pi(G, x)\to\pi_1(G', f(x))$.Then $f$ is an equivalence of groupoids.  
\end{prop}

\begin{proof} The conditions just mean that $f$ is essentially surjective and fully faithful hence an equivalence.
\end{proof}

\smallskip

For a category $C$ let $\Delta^\circ C$ be the category of simplicial objects in $C$.
In particular, we will use the category $\sset$ of simplicial sets and $\Delta^\circ\Ab$ of simplicial abelian groups.
 For a simplicial set $X$ let $|X|$ be its geometric
realization.  A morphism $f: X\to X'$ of simplicial sets is called a {\em weak equivalence}, if it induces a homotopy
equivalence $|X|\to |X'|$.  In this case we write $X\sim X'$. 
Let $\W$ be the class 
of weak equivalences. 

\smallskip

We also associate to  any simplicial set $X$ its {\em fundamental groupoid} $\Pi X$. 
Objects of $\Pi X$ are vertices
of $X$, i.e., elements $x\in X_0$, and, for $x,y\in X_0$, 
the set $\Hom_{\Pi X} (x,y)$ consists of homotopy classes of paths
in $|X|$ joining $x$ and $y$. 
Let $\pi_0(X)$ be the set of connected components of $|X|$, and, 
for each $i\geq 1$ and $x\in X_0$
let $\pi_i(X,x)$ be the topological homotopy groups of $|X|$ at $x$. 

\smallskip

Dually, the {\em nerve} $NG$ of a groupoid $G$ is a simplicial set with the set of $m$-simplices being
  \be\label{eq:nerve}
 N_mG \,= \, G_1 \times_{G_0} \times_{G_0} \cdots \times_{G_0} G_1 \quad \text{($m$ times)}.
  \ee
 The topological homotopy groups of $NG$ match those defined above algebraically for $G$:
 $$
 \pi_0(NG) \,=\,\pi_0(G), \quad \pi_1(NG, x) = \pi_i (G,x), \quad \pi_i (NG,  x)=0, \quad i\geq 2. 
 $$
A simplicial set is {\em of groupoid type}, if it is weak equivalent 
to the nerve of some groupoid.  We denote by $\sset^{\leq 1}\subset\sset$ the full subcategory of
simplicial sets of groupoid type.

\smallskip

\begin{prop}\label{prop:gpd} 
\hfill
\begin{itemize}[leftmargin=8mm]
\item[$\mathrm{(a)}$]
A simplicial set $X$ is of groupoid type if and only if  $\pi_i(X, x)=0$ for each $i\geq 2 $, $x\in X_0$. 
Then, we have $X \simeq N \Pi X$.
\item[$\mathrm{(b)}$]
The functors $\Pi$, $N $ yield quasi-inverse equivalences of homotopy categories 
$\sset^{\leq 1}[\W^{-1}] \simeq \Gpd[\Eq^{-1}].$
\qed
\end{itemize}

\end{prop}

\smallskip

Let $\Ac$ be an abelian category. We denote by $\Co(\Ac)$ the category of cochain complexes 
$K= (K^n, d^n: K^{n-1}\to K^n)_{n\in\ZZ}$ over $\Ac$ bounded below, with morphisms
being morphisms of complexes.
For $n\in \ZZ$ we denote by $\Co^{\leq n}(\Ac)$ the category of complexes concentrated in degrees 
$\leq n$. 
For $K\in \Co(\Ac)$ we denote by
\[
\begin{gathered}
K^{\leq n} \,=\, \bigl\{ \cdots \buildrel d^{n-1} \over\lra K^{n-1} \buildrel d^{n}\over\lra K^n \lra 0\lra  \cdots\bigr\}\,\in 
\, \Co^{\leq n}(\Ac),
\\
\tau_{\leq n} K \,=\,  \bigl\{ \cdots \buildrel d^{n-1} \over\lra K^{n-1} \buildrel d^{n}\over\lra \Ker(d^{n+1})\lra 0
\lra  \cdots\bigr\}\,\in \, \Co^{\leq n}(\Ac)
\end{gathered}
\]
its \emph{stupid} and \emph{cohomological} truncation in degrees $\leq n$. 
Note that $\tau_{\leq n}$ sends quasi-isomorphisms
of complexes to quasi-isomorphisms. 

\begin{exas}[Dold-Kan groupoids] \label{ex:compl-stack}
Let $\Ab$ denote the category of abelian groups. 
\hfill
\begin{itemize}[leftmargin=8mm]
\item[$\mathrm{(a)}$]
Given a $3$-term complex over $\Ab$
$$\xymatrix{K=\{K^{-1}\ar[r]^-{d^0}&K^{0}\ar[r]^-{d^1}&K^1\}},$$
 we  have the \emph{action groupoid} 
 $$\varpi K=\Ker(d^1)/\!/K^{-1}:=\{K^{-1}\times\Ker(d^1)\rightrightarrows\Ker(d^1)\}$$ 
 whose
  set of objects is $\Ker(d^1)$ and whose morphisms
$s\to t$ are given by $\{h\in K^{-1}\,;\,s+d^0(h)=t\}.$ Then we have
\[
\pi_0 (\varpi K) = H^0(K), \quad \pi_1(\varpi K, s) = H^{-1}(K),  \quad \forall \, s\in \Ob \,\varpi K. 
\]

\smallskip

\item[$\mathrm{(b)}$] 
The {\em Dold-Kan correspondence}
$\DK: \dg^{\leq 0}\Ab \to\Delta^\circ\Ab$
associates to a $\ZZ_{\leq 0}$-graded complex $K$ 
the simplicial abelian group $\DK(K)$ such that
 \begin{itemize}
 \item[$\bullet$] $\DK(K)_0=K^0$,
 
 \item[$\bullet$] the set of edges joining $s,t\in K^0$ is $\{h\in K^{-1}\,;\,s+d^0(h)=t\}$,
 
 \item[$\bullet$] 2-simplices with given 1-faces are in bijection with certain elements of $K^{-2}$, and so on,   
 see, e.g., \cite[\S 8.4.1]{weibel}. 
 \end{itemize} 
 For each $i\geqslant 0$, we have an isomorphism $\pi_i(\DK(K)) \,\simeq \, H^{-i}(K)$
 which is independent of the base point. 
 In fact, the correspondence preserves the respective standard model structures.
 In particular, for a 3-term complex $K$ as in (a), we have
 \begin{align}\label{PDK}
 \varpi K \, = \Pi \DK(\tau_{\leq 0} K). 
 \end{align}
 \end{itemize}
 \end{exas}

 \begin{exas}[Maurer-Cartan groupoids]\label{ex:mc-grpd}
 We will use a non-abelian generalization of Examples \ref{ex:compl-stack}, due to Deligne,
 see \cite{GM-BAMS},  \cite{GM} and references therein,   Hinich \cite{hinich}  and Getzler \cite{getzler}. 
 \hfill
\begin{itemize}[leftmargin=8mm]
\item[$\mathrm{(a)}$] 
Consider a (possibly infinite dimensional) dg-Lie algebra  $\gen$ over $\bbC$ situated in degrees $[0,2]$:
$$\xymatrix{\frakg=\{\frakg^{0}\ar[r]^-{d^0}&\frakg^{1}\ar[r]^-{d^1}&\frakg^2\}}.$$
Thus $\frakg^0$ is an ordinary complex Lie algebra. We assume that it is nilpotent, so we have
the nilpotent group $G^0 = \exp({\gen^0})$. By definition, $G^0$ consists of formal symbols $e^y, y\in \gen^0$
(so $G^0$ is identified with 
  $\frakg^0$ as a  set), 
with the multiplication given by the Campbell-Hausdorff formula.
The set of Maurer-Cartan elements of $\frakg$ is
\[
\mc(\frakg)\, =\, \biggl\{x\in\frakg^1\,;\,d^1x+ {1\over 2} [x,x]=0 \biggr\}.
\]
The group $G^0$ acts on $\mc(\frakg)$ by the formula
\be\label{eq:gauge}
e^y * x \,=\, e^{\ad}(y) (x)  + { 1-e^{\ad(y)} \over \ad(y)} (d^1(y)), 
\ee
see  \cite[p.~45]{GM}.
We define the \emph{Maurer-Cartan groupoid \footnote{In this paper we use the terms 
``Maurer-Cartan groupoid'' and ``Maurer-Cartan stack'' in order
to avoid clashes with the algebro-geometric notion of Deligne-Mumford stacks.}} (or {\em Deligne groupoid}) of $\frakg$ to be the action groupoid
$$\MC(\frakg)=\mc(\frakg)/\!/G^0:=\{G^0\times\mc(\frakg)\rightrightarrows\mc(\frakg)\}.$$
Note that if the dg-Lie algebra $\frakg$ is abelian, i.e., if it reduces to a 3-term cochain complex, then
$G^0=\frakg^0$ and it acts on $\mc(\frakg)=\Ker(d^1)$ by translation, so we have
$\MC(\frakg)=\varpi(\frakg[1])$ where $\varpi$ is as in Example \ref{ex:compl-stack} (a) .

\smallskip

\item[$\mathrm{(b)}$] More generally, let $\gen$ be any nilpotent dg-Lie algebra over $\CC$. The {\em Maurer-Cartan
simplicial set} $\mc_\bullet(\gen)$ is defined by
\[
\mc_n(\gen) \,=\, \mc(\gen \otimes_\CC \Omega^\bullet_\pol (\Delta^n)), 
\]
where $\Omega^\bullet_\pol (\Delta^n)$ is the commutative dg-algebra of polynomial differential forms on the 
standard $n$-simplex, see \cite{hinich}, \cite{getzler}. Further, in \cite{getzler} it is proved that
if $\gen$ is concentrated in degrees $[0,2]$
then $N_\bullet(\MC(\gen))$, is weak equivalent to $\mc_\bullet(\gen)$. 
  
\end{itemize}
\end{exas}

\begin{prop}\label{prop:mc-qis}
A quasi-isomorphism $\phi: \gen_1\to\gen_2$ of nilpotent dg-Lie algebras induces a weak equivalences of simplicial sets
$\mc_\bullet(\gen_1)\to \mc_\bullet(\gen_2)$. In particular:
 \hfill
\begin{itemize}[leftmargin=8mm]
\item[$\mathrm{(a)}$] 
If $\gen_1$, $\gen_2$ are concentrated in degrees $[0,2]$, then $\phi$ induces an equivalence of groupoids
$\MC(\gen_1) \to \MC(\gen_2)$.

\item[$\mathrm{(b)}$] A quasi-isomorphism $K_1\to K_2$ of cochain complexes as in Example 
$\ref{ex:compl-stack}\mathrm{(a)}$ induces an equivalence of groupoids $\varpi K_1\to \varpi K_2$. 
\end{itemize}
 \qed

\end{prop}

Let now  $p:\frakg\to\frakh$ be a surjective morphism of dg-Lie algebras,
both situated in degrees $[0,2]$. 
Let $\frakn\subset\frakg$ be the kernel of $p$ and assume that there is an embedding $i:\frakh\to\frakg$ with
$p\circ i=1$ such that $\frakg=\frakh\ltimes\frakn$ is the semi-direct product. 

\smallskip

We have a functor of groupoids $p_*:\MC(\frakg)\to\MC(\frakh)$.
Recall that for a functor $\phi:C\to D$ and an object $x\in\Ob(D)$,
the \emph{fiber category} $\phi/x$ consists of pairs $(y,h)$ with $y\in\Ob(C)$ and
$h:\phi(y)\to x$ a morphism in $D$, with the obvious notion of morphisms of such pairs.
If $C$, $D$ are groupoids, so is $\phi/x$.
We apply this when $C=\MC(\frakg)$, $D=\MC(\frakh)$ and $\phi=p_*$.
We get the fiber category $p_*/x$.
On the other hand, the object $x\in\Ob(\MC(\frakh))$ being an element of $\mc(\frakh)$, 
it gives a new differential $d_x=d-\ad(x)$ on $\frakn$, where we abbreviate $x=i(x)$.
Let $\frakn_x$ be the dg-Lie algebra with underlying Lie algebra $\frakn$ and differential $d_x$.

\begin{proposition} \label{prop:local}
The fiber category
$p_*/x$ is equivalent to the groupoid $\MC(\frakn_x)$.
\qed
\end{proposition}

\medskip


\subsection {Stacks and homotopy sheaves}
Let  $\scrS$ be  a Grothendieck site.  We recall that  a {\em stack}   (of essentially small 
groupoids)   on $\scrS$ is  a  presheaf   of groupoids $B: T\mapsto B(T)$, $T\in\Ob(\scrS)$, 
 satisfying the 2-categorical descent condition  extending that for sheaves of sets, see [...] for background. 
Stacks on $\scrS$ form a 2-category $\Stg_\scrS$. 
 We will refer to 1-morphisms   of $\Stg_\scrS$ as {\em morphisms of stacks} and 
will denote by $\St_\scrS$ the category of
stacks  on $\scrS$
with these morphisms. Let $\Eq\subset\Mor(\St_\scrS)$ be the class of equivalences of stacks.

\begin{rem}
For most purposes, the above 1-categorical point of view on stacks will be sufficient. 
However, in various constructions below such as gluing, 
the  full 2-categorical structure on $\Stg_\scrS$ becomes important.  
In particular,  as  with objects of any 2-category, to define a stack ``uniquely'' (e.g., naively,   in a way ``independent''
 on some choices)
means, more formally,  to define it {\em uniquely up to an equivalence which is  defined uniquely up to a unique
  isomorphism}. 

\end{rem}

A  stack of groupoids $B$ gives rise to a sheaf of sets $\ul\pi_0(B)$ on $\scrS$, 
obtained by sheafifying the presheaf $T\mapsto \pi_0(B(T))$. 
Similarly, for any $T\in\Ob(\scrS)$
and any object $x\in B(T)$ we  have a  sheaf of groups $\ul\pi_1(B, x)$ on $T$, i.e., on the site $\scrS/T$,
obtained by sheafifying the presheaf $T'\mapsto \pi_1(B(T'),x|_{T'})$,
where $x|_{T'}$ is the pullback by the  morphism $T'\to T$. 

\begin{prop}\label{prop:pi-stack-iso}
Let $f: B\to B'$ be a morphism in $\St_\scrS$ which
induces an isomorphism of sheaves  of sets 
$\ul\pi_0(B)\to \ul\pi_0(B')$ and an isomorphism of sheaves of groups 
$\ul\pi_1(B, x)\to \ul\pi_1(B', f(x))$ for any
$T\in\Ob(\scrS)$, $x\in\Ob(B(T))$.
Then $f$ is an  equivalence
of stacks. 
\end{prop}

 \begin{proof} 
 Follows from Proposition \ref{prop:pi-gpd-iso} by sheafification. 
 \end{proof}
 
 \smallskip
 
 Let $\sset_\scrS$ be the category of presheaves of simplicial sets on $\scrS$. 
  Recall \cite{toen-vezzosi} that such a presheaf $X$ is called a {\em homotopy sheaf} or an {\em $\infty$-stack},
  if it satisfies descent in the homotopy sense. We denote by $\St^\infty_\scrS$ the category of homotopy sheaves
  of simplicial sets on $\scrS$ and by $W\subset \Mor(\St^\infty_\scrS)$ the class of weak equivalences (defined
  stalk-wise). 
  A homotopy sheaf $X$ gives rise to a sheaf of sets $\ul\pi_0(X)$ on $\scrS$ and, for any $T\in\Ob(\scrS)$ and
  any vertex $x\in X(T)_0$, a sheaf of groups  $\ul\pi_i(X,x)$ on $\scrS/T$.  We have:

  \begin{prop}
  
   Let $f: X\to X'$ be a morphism in  $\St^\infty_\scrS$. Suppose $f$ induces an isomorphism of sheaves of sets
  $\ul\pi_0(X) \to \ul\pi_0(X')$ and, for each $T\in\Ob(\scrS)$ and $x\in X(T)_0$, an isomorphism of sheaves of 
  groups $\ul\pi_i(X,x)\to \ul\pi_i(X', f(x))$.
  Then $f$ is a weak equivalence. 
  \end{prop}
  
 \begin{proof} If $\scrS$ is a point, this is the standard: a map of simplicial sets is a weak equivalence
  iff it induces isomorhism on homotopy groups. The case of general $\scrS$ is obtained from this by sheafification.
  \end{proof}
  
  \smallskip
  
  Any homotopy sheaf $X$ gives a stack of groupoids $\Pi X$ on $\scrS$,  defined by taking $T\mapsto \Pi X(T)$. 
  Any stack of groupoids $B$ on $\scrS$ gives rise to a homotopy sheaf $N(B)$ taking $T$ to the nerve of the
  groupoid $B(T)$. A homotopy sheaf $X$ is called of {\em groupoid type}, if it is weak equivalent to $N(B)$ for
  some stack $B$. 
  We denote by $\St^{\infty, \leq 1}_\scrS\subset\St^\infty_\scrS$ the full category of homotopy sheaves
   of
  groupoid type. 
  
  \begin {prop}\label{prop:simpl-stack}
   \hfill
\begin{itemize}
[leftmargin=8mm]
\item[$\mathrm{(a)}$]
 A homotopy sheaf $X$ is of groupoid type if and only if $\ul\pi_i(X,x)=0$ for each $T\in\Ob(\scrS)$,  $x\in X(T)_0$ 
 and  $i\geq 2$. 
  
\item[$\mathrm{(b)}$]
The functors $\Pi$, $N$ induce mutually quasi-inverse equivalences of homotopy categories
$\St^{\infty, \leq 1}_\scrS [\W^{-1}] \simeq \St_\scrS[\Eq^{-1}].$  \qed
 \end{itemize}

  \end{prop}

\medskip


\subsection{Artin  and f-Artin stacks}\label{sec:stack}
In this paper all schemes, algebras, etc., will be considered over the base field $\CC$ of complex numbers.
Let $\wt \Aff $ be the category of affine schemes over $\CC$ equipped with the \'etale topology. 
We refer to  \cite{LMB}, \cite{olsson-book} for
general background on {\em Artin stacks}, i.e., stacks of groupoids on $\wt\Aff$ with a smooth atlas and a 
representable, quasi-compact, quasi-separated diagonal.

\begin{exas}\label  {ex:grpd-stack}
\hfill
\begin{itemize}
[leftmargin=8mm]
\item[$\mathrm{(a)}$]
Let 
  $G = \{\xymatrix{
  G_1  \ar@<.7ex>[r]^s \ar@<-.7ex>[r]_t & G_0
  }
  \}$ 
be a groupoid 
 the category of schemes of finite type such that  the source and target maps $s,t$ are smooth morphisms. It
  gives rise to an Artin  stack which we denote by $\|G\|$. By definition,  $\|G\|$ is
  the stack associated with the prestack
  \[
  T\,\mapsto \{ \Hom(T, G_1)  \rightrightarrows \Hom(T, G_0)\}.  
  \]
  
\item[$\mathrm{(b)}$] In particular, let $G$ be an affine  algebraic group acting on a scheme $Z$ of finite type. Then we have the
 {\em action groupoid }$\{G\times Z\rightrightarrows Z\}$ in the category of schemes of finite type. The corresponding Artin stack
  is denoted $Z/\!/G$ and is called the {\em quotient stack} of $Z$ by $G$. Explicitly, for $T\in\wt\Aff$ 
  the groupoid
  $(Z/\!/G)(T)$ is identified with the category of pairs $(P, u)$, where $P$ is a $G$-torsor over $T$
  (locally trivial in \'etale topology) and $u: P\to Z$ is a $G$-equivariant map. 
  \end{itemize}
 \end{exas}

\begin{Defi}
An Artin stack $B$ is called:
\hfill
\begin{itemize}
[leftmargin=8mm]
\item[$\mathrm{(a)}$]
 {\em Of finite type}, if it is equivalent to the stack of the form  $\|G\|$ for a groupoid $G$
as in Example $\ref{ex:grpd-stack}\mathrm{(a)}. $

\item[$\mathrm{(b)}$] An {\em f-Artin stack}, if  it is locally of finite type.

\item[$\mathrm{(c)}$] A {\em quotient} (resp. {\em locally quotient)  stack} is it is equivalent
(resp. locally equivalent)  to the stack of the form $Z/\!/G$ where $Z, G$ are in Example 
$\ref{ex:grpd-stack}\mathrm{(b)}. $

\end{itemize} 
\end{Defi}

All the stacks we will use will be f-Artin. 
Let the 2-category $\Stg$ and the category $\St$ be
the full 2-subcategory  in  $\Stg_{\wt \Aff}$ and   the full subcategory in $\St_{\wt\Aff}$ formed by f-Artin
stacks. 

\smallskip

Let $\Aff\subset\wt\Aff$ be the category of affine schemes {\em of finite type} with its \'etale topology.
We note that f-Artin stacks are determined by their restrictions to $\Aff$, and so we can  and will
consider them as stacks of groupoids on $\Aff$. 

\smallskip

   Given an f-Artin stack $B$, let $\Stg_B$ be the 2-category of  f-Artin stacks over $B$, i.e., 
   of f-Artin stacks $X$ together with
   a morphism of stacks $X\to B$. Objects of $\Stg_B$ can, equivalently, be seen as stacks of groupoids
   over the Grothendieck site $\Aff_B$ formed by affine schemes $T$  of finite type
   together with a morphism of stacks 
   $f:T\to B$.
    Thus, an f-Artin
   stack  $X$ over $B$ can be seen as associating to each $T\in\Aff_B$ a groupoid $X(T)$.

 \medskip

  
 \section{Stacks of extensions and filtrations} \label{sec:stack-ext}
  
\subsection{Cone stacks}\label{subsec:cone-stack}
We refer to \cite{O07, olsson-book} for general background on quasi-coherent sheaves on Artin stacks.
 For an f-Artin stack $B$ we denote by $QCoh (B)$, resp. $Coh(B)$ the category of quasi-coherent,
resp. coherent sheaves of $\Oc_B$-modules. By a {\em vector bundle} we mean a locally free sheaf of
finite rank. 

\smallskip

 Let $B$ be an f-Artin  stack 
 and $R=\bigoplus_{i\in\bbN}R^i$ be
  a graded quasi-coherent sheaf of $\calO_B$-algebras 
such that $R^0=\calO_B$, $R^1$ is coherent and $R$ is generated by $R^1$ locally over $B$.
The relative  affine $B$-scheme
 $C=\Spec\,  R$ is called a \emph{cone} over $B$, see, e.g., \cite[\S 1]{BF}.
 
 \smallskip

If $\calE$ is a coherent sheaf over $B$, we get the {\em associated cone}
$C(\calE)=\Spec ( \Sym_{\calO_B}(\calE))$ which is   an affine group scheme over $B$.
 Its value  (the set of points) on 
  $(T\buildrel f\over\to B)\in \Aff_B$  is   $\Hom_{\calO_T}(f^*\Ec ,\calO_T)$. 
 We call such a cone an \emph{abelian cone}.
 
 \smallskip
 
For instance, the {\em  total space} of a vector bundle $\calE$ over $X$ is defined as 
\[
\Tot(\calE)\,=\, C(\calE^\vee)\,=\, \Spec \, \Sym_{\Oc_B}(\Ec^\vee)
\]
 where $\calE^\vee$ is the dual sheaf of $\calO_B$-modules.  
 For any affine $B$-scheme $f: T\to B$ we have
 \be\label{eq:total-bundle}
 \Tot(\Ec)(T)\,=\, H^0(T, f^*\Ec). 
 \ee
 Thus, a section $s\in H^0(B, \Ec)$ is the same as  a morphism  $B\to\Tot(\Ec)$
 of schemes over $B$.

 \smallskip
 
Any cone $C=\Spec\,  R$ is canonically a closed subcone of the abelian cone $\Spec ( \Sym_{\calO_B} (R^1)),$
called the \emph{abelian hull} of $C$.

\begin{ex}
\label{ex:ncone}
  Let $d: \Ec\to\Fc$ be a morphism of vector bundles on $B$. 
We denote by $\ul\Ker(d)\subset \Ec$ the sheaf-theoretic
kernel of $d$. On the other hand, let $\pi: \Tot(\Ec)\to B$ be the projection. 
The morphism $d$ determines a section $s$
of the vector bundle $\pi^*\Fc$ on $\Tot(\Ec)$, and we define the abelian cone $\Ker(d)\subset \Tot(\Ec)$ as the zero 
locus of this section. 
We note that $H^0(B, \ul\Ker(d))\subset H^0(B,\Ec)$ consists  
precisely of those sections $s$ which, considered as morphisms $B\to\Tot(\Ec)$, 
factor through the substack $\Ker(d)$.

\end{ex}

A {\em morphism of abelian cones} over $B$ is, by definition a morphism of group schemes over $B$.
Given a morphism of abelian cones $E\to F$, we have an action of the affine group scheme $E$ over $B$
on $F$.
Hence, we can form the quotient Artin stack $F/\!/E$.  Stacks of this form are called \emph{abelian cone stacks}.


\subsection  {Total spaces of perfect complexes}\label{subsec:tot-compl}

Let $B$ be an f-Artin  stack.
 We denote $C_\qcoh(B)$ the category formed by complexes of $\Oc_B$-modules with quasi-coherent cohomology.
Let $\on{qis}$ be the class of quasi-isomorphisms in  $C_\qcoh(B)$
and $D_\qcoh(B) = C_\qcoh(b)[\on{qis}^{-1}]$ be the corresponding derived category.
For any integers $p \leq q$ let  $C^{[p,q]}_\qcoh(B)\subset C_\qcoh(B)$ be the full subcategory formed
 by  complexes
situated in degrees from $p$ to $q$.

\begin{Defi}\label{def:perfect}
 Let $\Cc\in C_\qcoh (B)$ and $p\leq q$ be integers. 
 \hfill
\begin{itemize}[leftmargin=8mm]
\item[$\mathrm{(a)}$]
$\Cc$ is
 {\em strictly $[p,q]$-perfect}, if $\Cc$ is quasi-isomorphic to a complex  of vector bundles 
 \[
 \bigl\{\Cc^p \buildrel d^{p+1} \over\to \Cc^{p+1}\buildrel d^{p+2}\over\to \cdots \buildrel d^q\over\to \Cc^q\bigr\}
 \]
 situated in degrees from $p$ to $q$. This complex is called a {\em presentation} of $\Cc$. 
 
\item[$\mathrm{(b)}$]
$\Cc$ is $[p,q]$-{\em perfect}, if, locally on $B$, it is strictly $[p,q]$-perfect and, moreover, 
the set of open substacks
 $U\subset B$ such that $\Cc|_U$ is strictly $[p,q]$-perfect, is  filtering
 with respect to the partial order by inclusion.
 \end{itemize} 
 \end{Defi}
 
 For a $[p,q]$-perfect complex $\Cc$  and an open $U\subset B$ as above we will refer to a quasi-isomorphism
 $\Cc|_U\to \Cc_U$, with $\Cc_U$ strictly $[p,q]$-perfect, as a {\em presentation} of $\Cc$ over $U$.
 
 \smallskip

 A $[p,q]$-perfect complex $\Cc$ has a {\em virtual rank} $\vrk(\Cc)$ which is a $\ZZ$-valued locally constant 
 function on $B$, i.e., a function constant on each connected component of $B$. 
 It is defined in terms of a presentation of $\Cc$ as
 $\vrk(\Cc) \,=\,\sum _{i=p}^q (-1)^i \rk(\Cc^i).$
 
 \smallskip

We will be interested in making sense of total spaces of perfect complexes using \eqref{eq:total-bundle}
as a motivation, cf.  \cite[\S 3.3]{T14} . 

\begin{Defi}
 \hfill
\begin{itemize}
[leftmargin=8mm]
\item[$\mathrm{(a)}$]
Let $\Cc \in C^{\leqslant 0} _\qcoh(B)$. We define the simplicial presheaf $\Tot^\infty(\Cc)$ on $\Aff_B$ by
\[
\Tot^\infty (\Cc)(T) \,=\, \DK(H^0(T, f^*\Cc)), \quad( T\buildrel f\over\to B)\,\in \,\Aff_B. 
\]
\item[$\mathrm{(b)}$]
Let $\Cc\in C^{[-1,0]}_\qcoh(B)$. We define the pre-stack of groupoids $\Tot(\Cc)$ 
on $\Aff_B$ by
\[
\Tot(\Cc)(T) \,=\, \varpi (H^0(T, f^*\Cc)), \quad( T\buildrel f\over\to B)\,\in \,\Aff_B. 
\]
\end{itemize}
\end{Defi}

We call $\Tot(\Cc) $ the \emph{total space} of $\Cc$.

\begin{prop}\label{prop:Tot-simpl}
 \hfill
\begin{itemize}
[leftmargin=8mm]
\item[$\mathrm{(a)}$]
Let $\Cc \in C^{\leqslant 0} _\qcoh(B)$.
The simplicial presheaf $\Tot^\infty(\Cc)$ is a homotopy sheaf. 
For any $x\in  \Tot^\infty(\Cc)(T)_0$ we have (independently on the choice of base points)
  \[
 \ul\pi_i(\Tot^\infty(\Cc)) \,= \,  \ul H^{-i}(\Cc), \quad i\geq 0. 
 \]
A  morphism $\phi: \Cc_1\to \Cc_2$ in $C^{\leqslant 0}_\qcoh(B)$ induces a morphism of homotopy sheaves 
 $\phi_\flat: \Tot^\infty(\Cc_1)\to\Tot^\infty(\Cc_2)$ which is an equivalence, if $\phi$ is a quasi-isomorphism.

\item[$\mathrm{(b)}$]
Let $\Cc\in C^{[-1,0]}_\qcoh(B)$. The pre-stack $\Tot(\Cc)$ on $\Aff_B$ is a stack. 
The homotopy sheaf $\Tot^\infty(\Cc)$
is of groupoid type and  $\Pi\Tot^\infty(\Cc)=\Tot(\Cc)$. In particular, 
the total space is functorial and
takes quasi-isomorphisms $\phi$ to isomorphisms $\phi_\flat$. 
\end{itemize}
\end{prop}

\begin{proof} Part (a) follows from the fact that $\Cc$ is a sheaf and from the properties of  the Dold-Kan correspondence. 
Part (b) follows  by
 Proposition \ref{prop:simpl-stack}. 
 \end{proof}

\smallskip 

Recall that a stack morphism $f$ is called an   {\em l.c.i.}, i.e., a \emph{locally complete intersection morphism},
if it factorizes as $f=p\circ i$ where
$p$ is a smooth map and $i$ is a regular immersion. 

\begin{prop}\label{prop:tot(-1,0)}
\hfill
\begin{itemize}[leftmargin=8mm]
\item[$\mathrm{(a)}$]
Let $\Cc\in C^{[-1,0]}_\qcoh(B)$ be strictly $[-1,0]$-perfect. 
 Then we have a canonical equivalence of stacks of groupoids 
 $u: \Tot(\Cc) \lra \Tot(\Cc^0)/\!/ \Tot(\Cc^{-1})$ on $\Aff_B$.

\item[$\mathrm{(b)}$]
Let $\Cc\in C^{[-1,0]}_\qcoh(B)$ be $[-1,0]$-perfect. Then $\Tot(\Cc)$ is an Artin stack over $B$. 

\item[$\mathrm{(c)}$] For any morphism $\phi$ of $[-1,0]$-perfect complexes, 
 the induced morphism $\phi_\flat$ of stacks is an l.c.i. 
\end{itemize}
\end{prop}

\begin{proof} Part (a)
   is similar to the proof of \cite[lem~0.1]{He04}.
    That is, look at any $(T\buildrel f\over\to B)\in \Aff_B$. 
     By definition, the groupoid 
   $\Tot(\Cc)(T)$  is the category whose objects are elements of $x$  of $H^0(T, f^*\Cc^0)$ and a morphism 
   $x\to x'$ is an element of $H^0(T, f^*\Cc^{-1})$ mapping by $d^0$ to $x'-x$.
 At the same time, the groupoid 
$(\Tot(\Cc^1)/\!/\Tot(\Cc^0))(T)$
is the category of pairs consisting of an $f^*\Cc^{-1}$-torsor $P$
over $T$ and an $f^*\calC^{-1}$-equivariant morphism $P\to\Cc^0$ of sheaves over $T$.
We see that the former category is the full subcategory of the second  consisting of data with
the torsor $P$ being the standard  trivial one, $P=f^* \Cc^{-1}$. This defines a fully faithful functor $u_T$,
and such functors for all $T$ give the sought-for morphism of stacks $u$. 
Now, since $T$ is affine,  $H^1(T, f^*\Cc^{-1})=0$ and so any torsor $P$ above is trivial. This means that the
functor $u$ is (locally) essentially surjective hence an equivalence of stacks.  
This proves (a). 
Parts (b) and (c)  follow  from (a). \end{proof}

\begin{ex}\label{ex:perf-com}
Now, let    $\Cc$ be a strictly $[-1,1]$-perfect complex
\be\label{eq:presentation}
\xymatrix{\calC=\{\calC^{-1}\ar[r]^-{d^0}&\calC^{0}\ar[r]^-{d^1}&\calC^1\}.}
\ee
 The stupid truncation $\Cc^{\leq 0} = \{ \Cc^{-1}\to\Cc^0\}$
 is strictly $[-1,0]$-perfect. 
 We denote by 
 \[
 \pi: \Tot(\Cc^0)\to B ,\quad  \ol\pi: \Tot(\Cc^{\leq 0})\ = \Tot(\Cc^0)/\!/\Tot(\Cc^{-1}) \lra  B
 \]
 the projections.  We recall from Example \ref{ex:ncone}(c) the abelian cone $\Ker(d^1)\subset\Tot(\Cc^0)$
 given as the zero locus of the section $s$ of $\pi^*\Cc^1$ induced by $d^1$.
 \end{ex}
 
 \begin{prop}\label{prop:perf-com}
  \hfill
\begin{itemize}
[leftmargin=8mm]
\item[$\mathrm{(a)}$]
If $\Cc$ is  strictly $[-1,1]$-perfect, then we have a canonical equivalence of
stacks 
$ \Ker(d^1)/\!/\Cc^{-1} \lra \Tot(\tau_{\leq 0}\Cc),$
 i.e., the section $s$ descends to a section $\ol s$ of $\ol\pi^*\Cc^1$, and
 $\Tot(\tau_{\leq 0} \Cc)$ is the zero locus of  $\ol s$. 
 
 \item[$\mathrm{(b)}$]
 If $\Cc$ is $[-1,1]$-perfect, then $\Tot(\tau_{\leq 0} \Cc)$ is an Artin stack over $B$. 
 \end{itemize} 
 \end{prop}

 \begin{proof} Part (a) is completely analogous to the proof of Proposition
 \ref{prop:tot(-1,0)}(a), with $\Cc^0$ replaced by $\ul\Ker(d^1)$. Part (b) follows from (a). 
 \end{proof}

\smallskip

We call $\Tot(\tau_{\leq 0}\Cc) $ the \emph{truncated total space} of $\Cc$.

\begin{proposition}\label{prop:C}
Let $\Cc$ be a  $[-1,1]$-perfect complex and $(T\buildrel f\over\to B)\in\Aff_B$. 
\begin{itemize}
[leftmargin=8mm]
 \item[$\mathrm{(a)}$] 
For all $s\in\Ob(\Tot(\tau_{\leq 0}\Cc)(T))$ we have 
$$\ul \pi_0(\Tot(\tau_{\leq 0}\Cc)) \,\simeq \, \ul H^0(\Cc),\quad
\ul \pi_1(\Tot(\tau_{\leq 0}\Cc), s)\,\simeq \, \ul H^{-1}(f^*\Cc).$$

\item[$\mathrm{(b)}$]
The truncated total space of   $[-1,1]$-perfect complexes is functorial and
takes quasi-isomorphisms $\phi$ to isomorphisms $\phi_\flat$.  \qed

\end{itemize}
\end{proposition}

\begin{proof} 
  Part (a) is a consequence of Proposition \ref{prop:perf-com}.
 Part (b) follows from (c).
 More precisely, a morphism  (resp. quasi-isomorphism) $\phi:\Cc_1\to\Cc_2$ of $[-1,1]$-perfect complexes yields
 a morphism (resp. quasi-isomorphism) $ \tau_{\leq 0} \Cc_1\to \tau_{\leq 0} \Cc_2$ and
 the statement follows from Proposition \ref{prop:Tot-simpl}(b). 
 \end{proof}


 \subsection{Stacks of extensions} 

We now consider the following general situation. 
Let $B$ be an f-Artin stack and $p:Y\to B$ be a scheme  of finite type over $B$.
Let $\calE$, $\calF$ be coherent sheaves over $Y$ which are flat over $B$. 
We can form the object $\Cc\in D^b_\qcoh(B)$ given by
$$\Cc= Rp_* \, R\underline\Hom_{\calO_Y}(\calF,\calE)[1].$$
Let $\SES$ be the stack over $B$
 classifying  short exact sequences 
$0\to \calE\lra \calG\lra \calF\to 0$
  of coherent sheaves over $Y$. 
  That is, for any $B$-scheme $T\in\Aff_B$ the objects of the groupoid $\SES(T)$ are 
short exact sequences
\be\label{eq:exten}
0\to \calE|_T\to\calG\to \calF|_T\to 0
\ee
 of coherent sheaves  of $\calO_{Y\times_B T}$-modules, and the morphisms are the isomorphisms
of such sequences identical on the boundary terms. We then have
\be\label{eq:ext-on-T}
\pi_0 (\SES(T)) \,=\,  \Ext^1_{\calO_{Y\times_B T}}(\calF|_T,\calE|_T), \quad 
\pi_1(\SES(T),\Gc) \,=\,  \Ext^0_{\calO_{Y\times_B T}}(\calF|_T,\calE|_T),
\ee
for any object $\Gc$ of $\SES(T)$. This implies identifications of sheaves of sets on $\Aff_B$, 
and of sheaves of groups on $\Aff_T$: 
 \be
\begin{gathered}
\ul\pi_0(\SES) \,=\, \ul H^0 ( \Cc), 
\quad 
 \ul \pi_1(\SES\times_BT, \Gc) \,=\,
 \ul H^{-1} (\Cc|_T). 
 \end{gathered} 
\ee
These identifications, together with those of Proposition \ref{prop:B} (b), suggest the following.

\begin{proposition} \label{prop:A}
Assume that the complex $\Cc$ is $[-1,1]$-perfect.
Then, we have an equivalence 
$\Tot (\tau_{\leq 0}\Cc) = \SES$
of cone stacks over $B$. 
 \end{proposition}

 \begin{proof} As pointed out, the $\ul\pi_0$ and $\ul\pi_1$ of the two stacks 
 $\Tot (\tau_{\leq 0}\Cc)$ and $\SES$ are isomorphic.
 So it remains to construct a morphism of stacks inducing these identifications. 
 For this, we first make some general discussion. 
 
 \smallskip

  We recall \cite{bondal-kapranov}, \cite{keller-dg}, \cite{toen-morita}
 that for any Artin stack $Z$ the category $D_\qcoh^b(Z)$ has a {\em dg-thickening}, i.e., there is a pre-triangulated 
  dg-category $C_\qcoh(Z)$ with the same objects and spaces of morphisms being upgraded to
 complexes $\RHom_{C_\qcoh(Z)}(\Kc, \Lc)$ of $\CC$-vector spaces such that
 \[
 \Hom_{\Oc_Z}(\Kc, \Lc) \,=\, H^0 \RHom_{C_\qcoh(Z)}(\Kc, \Lc). 
 \]
The complex $\RHom$ above can be explicitly found as 
\be\label{eq:RHom-expl}
\RHom_{C_\qcoh(Z)}(\Kc, \Lc) \,=\, \Hom^\bullet_{\Oc_Z}(I(\Kc), I(\Lc)),
\ee
 where $I(\Kc)$ is a fixed injective resolution of $\Kc$ for each $\Kc$. 
 
 \smallskip
 
 We now specialize to the case
 \[
 Z = Y\times_BT, \quad \Kc= \Fc|_T, \,\, \Lc = \Ec|_T[1],
 \]
 where $T\in\Aff_B$ is an affine $B$-scheme. The complex of $\CC$-vector spaces
 \[
 \tau_{\leq 0} \RHom_{C_\qcoh(Z)}(\Fc|_T, \Ec|_T[1])
 \] 
 has cohomology only in degrees $0$ and $-1$, given by the Ext groups in \eqref{eq:ext-on-T}. 
We consider the simplicial set 
 $$X(T) \,=\, \DK \bigl(  \tau_{\leq 0} \RHom_{C_\qcoh(Z)}(\Fc|_T, \Ec|_T[1]) \bigr),$$
which is of groupoid type by Proposition \ref{prop:gpd}(a).
Its vertices are morphisms of complexes $I(\Fc|_T)\to I(\Ec|_T[1])$. 
The cone of such a morphism is a complex of sheaves which has only one cohomology
sheaf, in degree $-1$, and this sheaf $\Gc$ fits into a short exact sequence as in \eqref{eq:exten}. 
In this way we get 
a morphism of groupoids
$$h(T): \Pi X(T)\to \SES(T).$$ At the same time, by \eqref{PDK}, 
the groupoid $\Pi X(T)$ is equivalent to the groupoid 
$\Gamma H^0(T, \Cc|_T)$ in Example \ref{ex:compl-stack}(a), hence to
$ \Tot(\tau_{\leq 0}\Cc)(T)$ by Proposition \ref{prop:C}(a). 
Combining these constructions for all $T\in\Aff_B$, we get a homotopy sheaf $X$ of simplicial sets on 
  $\Aff_B$ 
  of groupoid type, together with an equivalence and a morphism of stacks
  $$\Tot(\tau_{\leq 0} \Cc) \simeq \Pi X \buildrel h\over\lra \SES.$$
The morphism $h$ induces the required identification on 
  $\ul\pi_0$ and $\ul\pi_1$, so it is an equivalence of stacks. 
  Proposition \ref{prop:A} is proved. \end{proof}

 \medskip
 
 
 \subsection{Maurer-Cartan stacks}\label{sec:MCS}
 We now describe a non-abelian generalization of the construction of \S \ref {subsec:tot-compl}.  
 Let $B$ be an f-Artin  stack
 and $(\Gc, d, [-,-])$ be an $\Oc_B$-dg-Lie algebra with quasi-coherent
 cohomology. In other words, $\Gc$ is a Lie algebra object in the symmetric monoidal category 
 $(C_\qcoh(B), \otimes_B)$. We will assume that $\Gc$ is nilpotent. 
 We define the {\em Maurer-Cartan $\infty$-stack} of $\Gc$ to be the simplicial presheaf $\mc_\bullet(\Gc)$
 on $\Aff_B$ defined by
 \[
 \mc_\bullet(\Gc) (T) \,=\, \mc_\bullet (H^0(T, f^*\Gc)). 
 \]
 Here $(T\buildrel f\over\to B)$ is an object of $\Aff_B$, and we apply the functor $\mc_\bullet$
 to the dg-Lie algebra $H^0(T, f^*\Gc)$ over $\CC$. 
 
 \begin{prop}
  \hfill
\begin{itemize}
[leftmargin=8mm]
\item[$\mathrm{(a)}$]
The simplicial presheaf $\mc_\bullet(\Gc)$ is a homotopy sheaf.
 
\item[$\mathrm{(b)}$] A morphism (resp. quasi-isomorphism) $\phi: \Gc_1\to \Gc_2$ of nilpotent $\Oc_B$-dg-Lie 
algebras induces
 a morphism (resp. weak equivalence) of homotopy sheaves $\phi_\flat: \mc_\bullet(\Gc_1) \to \mc_\bullet(\Gc_2)$.
 \end{itemize}

 \end{prop} 
 
 \begin{proof} Part (b) follows from Proposition \ref{prop:mc-qis}   by sheafification. 
 \end{proof}
 
 \smallskip
 
Assume that  the dg-Lie algebra $\Gc$ is situated in degrees $[0,2]$, i.e., 
 \begin{align}\label{SPdgLA}
 \xymatrix{\Gc=\{\Gc^{0}\ar[r]^-{d^0}&\Gc^{1}\ar[r]^-{d^1}&\Gc^2\}.}
 \end{align}
Then we define the stack $\MC(\Gc)$ of groupoids on $\Aff_B$ by
 \[
 \MC(\Gc)(T) \,=\, \MC (H^0(T, \Gc|_T))
 \]
 We call $\MC(\Gc)$ the {\em Maurer-Cartan stack} associated to
 a 3-term $\Oc_B$-dg-Lie algebra $\Gc$. 
 
 \begin{prop}
 If $\Gc$ is situated in degrees $[0,2]$, then the simplicial sheaf $\mc_\bullet(\Gc)$ is of groupoid type
 and $\Pi\mc_\bullet(\Gc)=\MC(\Gc)$. \qed
 \end{prop}

 Let  $\Gc$ be any $\Oc_B$-dg-Lie algebra with quasi-coherent cohomology. 
 As for complexes, we call $\Gc$ {\em strictly $[0,2]$-perfect}, if it is quasi-isomorphic, 
 as an $\Oc_B$-dg-Lie algebra,
 to a $3$-term dg-Lie algebra
 \eqref{SPdgLA}
 with each $\Gc^i$ being a vector bundle on $B$. We say that $\Gc$ is {\em $[0,2]$-perfect}, if,
 locally on $B$, it is strictly $[0,2]$-perfect and, moreover, the set of open substacks
 $U\subset B$ such that $\Gc|_U$ is strictly $[0,2]$-perfect, is  filtering
 with respect to the partial order by inclusion.

 \smallskip
 
 We now assume that  $\Gc$ be a strictly $[0,2]$-perfect dg-Lie algebra as in \eqref{SPdgLA}. 
 Then, we have the closed substack
$\mc(\Gc)\subset\Tot(\Gc^1)$
``given by the equation $d^1x+{1\over 2}  [x,x]=0$'', with 
two equivalent definitions :

\begin{itemize}
\item[($\mathrm{mc1}$)]  
For any affine $B$-scheme $T\buildrel f\over\to B$ we have a dg-Lie algebra $H^0(T, \Gc|_T)$, and we define 
\[
\mc(\Gc)(T)=\mc (H^0(T,\Gc|_T)).
\]

\item[($\mathrm{mc2}$)]  
The stack $\mc(\Gc)$ is the zero locus of the section $s_\Gc$ of $\pi^*\Gc^2$ given by
the \emph{curvature}
\be\label{eq:curv}
\Gc^1\to\Gc^2,\quad x\mapsto d^1x+ {1\over 2} [x,x].
\ee

\end{itemize} 

Since the Lie algebra $\Gc^0$ is nilpotent, we have a sheaf of groups
$G^0 = \exp(\Gc^0)$ on $B$ by Malcev theory, which acts on the stack $\mc(\Gc)$ as in \eqref{eq:gauge}, 
and we can consider the quotient stack
$\mc(\Gc)/\!/G^0.$
 Consider also the quotient stack
$$\Tot(\Gc^{\leqslant 1})=\Tot(\Gc^1)/\!/G^0$$ 
and denote its projection to $B$ by $\ol \pi$.
 
 \begin{prop}
  \hfill
\begin{itemize}
[leftmargin=8mm]
\item[$\mathrm{(a)}$]
 Let $\Gc$ be a strictly $[0,2]$-perfect  dg-Lie algebra as in \eqref{SPdgLA}.  
 
\begin{itemize}
 
\item[$\mathrm{(a1)}$] We have an equivalence of stacks $u: \MC(\Gc)\to \mc(\Gc)/\!/G^0$,
so $\MC(\Gc)$ is an Artin stack. 
 
\item[$\mathrm{(a2)}$]
The section $s_\Gc$ of the bundle $\pi^*\Gc^2$ on $\Tot(\Gc^1)$ descends to a
section $\ol s_\Gc$ of the bundle $\ol \pi^*\Gc^2$ on $\Tot(\Gc^{\leqslant 1})$, and the substack
$\MC(\Gc)\subset \Tot(\Gc^{\leqslant 1})$ 
is the zero locus of  $\ol s_\Gc$.

\end{itemize}
 
\item[$\mathrm{(b)}$]
If $\Gc$ is a $[0,2]$-perfect $\Oc_B$-dg-Lie algebra, then the simplicial sheaf $\mc_\bullet(\Gc)$ is
 of groupoid type.  The stack of groupoids $\MC(\Gc):=\Pi\mc_\bullet(\Gc)$ is an Artin stack over $B$. 
 \end{itemize}
  \end{prop}
 
 \begin{proof} Part (a1) is proved similarly to Proposition \ref{prop:tot(-1,0)}(a), using the fact
 that, $G^0$ being a unipotent sheaf of groups, any $f^*G^0$-torsor over any  $T\in\Aff_B$
 is trivial. Part (a2) follows from (a) and from the equivalence of the two definitions (mc1) and (mc2)
 of the stack $\mc(\Gc)$. Part (b) follows because being of groupoid type and being an Artin stack over $B$
 are properties local on  $B$. \end{proof}

\begin{example}\label{ex:abelian}
If the dg-Lie algebra $\Gc$ is abelian, i.e., it reduces to a  $[0,2]$-perfect complex on $B$, then 
$\MC(\Gc)=\Tot(\tau_{\leqslant 0}(\Gc[1]))$.
\end{example}

\smallskip

Let us now  globalize the considerations of Proposition \ref{prop:local}  as follows.
Let $p:\Gc=\Hc\ltimes\Nc\to\Hc$ be a split extension of strictly $[0,2]$-perfect dg-Lie algebras on $B$.
The $B$-scheme $\pi_\Hc:\mc(\Hc)\to B$ carries a strictly $[0,2]$-perfect dg-Lie algebra $\tilde\Nc$
which is equal to $\pi_\Hc^*\Nc$ as a sheaf graded of $\Oc_{\mc(\Hc)}$-Lie algebras, 
with the differential $d_x$ at a point $x\in\mc(\Hc)$ defined as above.
The action of the sheaf of groups $H^0$ on $\mc(\Hc)$
extends to a compatible action on $\tilde\Nc$, so that $\tilde\Nc$ descends to a
strictly $[0,2]$-perfect dg-Lie algebras on the stack $\MC(\Hc)$.
We denote this descended dg-Lie algebra by the same symbol $\tilde\Nc$.
Note that $\MC(\tilde\Nc)$ is a stack over $\MC(\Hc)$, hence over $B$.
Now, we have the following global analogue of Proposition \ref{prop:local}.

\smallskip

\begin{proposition}\label{prop:global}
The stacks 
$\MC(\Gc)$ and $\MC(\tilde\Nc)$
over $B$ are isomorphic.
\end{proposition}

\begin{proof}
For each affine $B$-scheme $T\in\Aff_B$,  we have a split exact sequence of dg-Lie algebras
$$\xymatrix{
0\ar[r]& H^0(T,\Nc|_T)\ar[r]& H^0(T,\Gc|_T)\ar[r]^p& H^0(T,\Hc|_T)\ar[r]& 0}$$
which gives rise to a functor $p_*:\MC(H^0(T,\Gc|_T))\to \MC(H^0(T,\Hc|_T))$
with the fiber category over an object $x$ equivalent to $\MC(H^0(T,\Hc|_T)_x)$.
This yields the following isomorphism of groupoids over $\MC(H^0(T,\Hc|_T))$
$$\MC(H^0(T,\Gc|_T))=\MC(H^0(T,\tilde\Nc|_T)).$$
\end{proof}

 \medskip
 

 \subsection{Stacks of filtrations}\label{sec:FILT}
 Let $B$ be an f-Artin stack  and $p:Y\to B$ be a scheme over $B$, locally of finite type.
 Let $\Ec_{01}$, $\Ec_{12}$, $\Ec_{23}$ be coherent sheaves over $Y$ which are flat over $B$.
 We define $\FILT$ to be the stack over $B$ classifying filtered coherent sheaves
 $\Ec_{01}\subset\Ec_{02}\subset\Ec_{03}$ over $Y$, together with identifications
 $\Ec_{0j}/\Ec_{0i}\simeq\Ec_{ij}$ for $ij=12,23$.
 We have a sheaf of associative dg-algebras over $B$ defined by
 \begin{align}\label{Gc}
 \Gc=\bigoplus_{ij<kl} Rp_*\underline\RHom(\Ec_{kl},\Ec_{ij}),\quad 01<12<23.
 \end{align}
 We'll consider $\Gc$ as a sheaf of dg-Lie algebras using the supercommutator.
 Then, we have the following generalization of Proposition \ref{prop:A}.
 \smallskip
 
\begin{proposition}\label{prop:B}
Assume that $\Gc$ is a strictly $[0,2]$-perfect dg-Lie algebra on $B$.
Then, we have an equivalence $\MC(\Gc)=\FILT$ of stacks over $B$.
\end{proposition}

\begin{proof}
Let $\SES_{012}$ be the stack over $B$
classifying  short exact sequences 
\be\label{eq:ext-02}
\Ec_{012} \,=\,\bigl\{ 0\to \Ec_{01}\lra \Ec_{02}\lra \Ec_{12}\to 0\bigr\}
\ee
of coherent sheaves over $Y$. 
Then $\FILT$ is the stack over $\SES_{012}$
classifying  short exact sequences 
\be\label{eq:ext-03}
\Ec_{0123} \,=\, \bigl\{ 0\to \Ec_{02}\lra \Ec_{03}\lra \Ec_{23}\to 0\bigr\},
\ee
and
$\Gc=\Hc\ltimes\Nc$
where
\begin{align*}
\Nc=Rp_*\underline\Hom(\Ec_{23},\Ec_{01}\oplus \Ec_{12}),\quad
\Hc=Rp_*\underline\Hom(\Ec_{12},\Ec_{01}).
\end{align*}
Since the dg-Lie algebra $\Hc$ is abelian, 
by Example \ref{ex:abelian} and Proposition \ref{prop:A} the stacks
$\MC(\Hc)$, $SES_{012}$ are equivalent,
and $\Nc$ gives an abelian strictly $[0,2]$-perfect dg-Lie algebra $\tilde\Nc$ over $\SES_{012}$.
Further, by Proposition \ref{prop:global}, we have $\MC(\Gc)=\MC(\tilde\Nc)$ as stacks over $\SES_{012}$.
So we are reduced to prove that $\MC(\tilde\Nc)$ is the stack over $\SES_{012}$
classifying  short exact sequences \eqref{eq:ext-03}. 

\smallskip

Let $T\buildrel f\over\to B$ be an affine $B$-scheme. Suppose the object $\Ec_{012}$  of $\SES_{012}(T)$ 
is the cone of a morphism
$u_{012}$ in $\RHom^1_{Y\times_B T} (f^* \Ec_{12}, f^* \Ec_{01}).$
Thus, given injective resolutions of $f^*\Ec_{ij}$ for each $i,$ $j$,
the complex $\Ec_{02} $ is quasi-isomorphic to the complex 
$C(u_{012})=I_{12}\oplus I_{01}$ where the differential is the sum of the differentials
of $I_{12}$, $I_{01}$ and the composition with $u_{012}$, viewed as a morphism of complexes of sheaves
$I_{12}\to I_{01}[1]$. 

\smallskip

 Next, we have $\tilde\Nc=\pi_\Hc^*\Nc$ as a graded sheaf, and the differential $d_{012}$ of $\wt\Nc$ at the point
 $\Ec_{012}$  is given by
 \[
 d_{012}(u) = d(u) -  \ad(u_{012})(u), \quad\forall u\in\Hom_{Y\times_BT}(f^*\Ec_{23},f^*\Ec_{01}\oplus f^*\Ec_{12}),
 \]
see Proposition \ref {prop:local} and the discussion before it. 
In our case $\ad(u_{012})(u)$ reduces  to the composition $u_{012}u$. 
Thus, the condition for $u$  to satisfy the equation $d_{012}(u)=0$
 is equivalent to saying that it lifts to a morphism of complexes
$f^*\Ec_{23}\to C(u_{012})$, i.e.,  to a dotted arrow $u_{0123}$ in the diagram. 
\[
\xymatrix{
&\Ec_{02}\ar[dr]
&&
\\
f^*\Ec_{01} \ar[ur]
&&  \ar[ll]_{+1}^{u_{012}} f^* \Ec_{12}&\ar[l]^{u} f^*\Ec_{23}
\ar@{-->}[ull]_{ u_{0123}}^{+1}\ar@/^2pc/[lll]^-u
}
\]

The cone of such an arrow defines $\Ec_{03}$ with a short exact sequence
\eqref{eq:ext-03}.   
We have thus constructed a morphism $\MC(\tilde\Nc)\to \FILT$ of stacks over $\SES_{012}$,
and it is easy to check that this morphism is an equivalence.  \end{proof}

\medskip
  
 
 \section{Borel-Moore homology of stacks and virtual pullbacks}\label{sec:BM}

 \subsection{BM homology  and operations for schemes}\label{subsec:BM-spaces}
   We fix a field $\k$ of characteristic $0$ which will serve as the field of coefficients for (co)homology. 
  The cases $\k=\QQ$ or $\k=\QQ_l$  will be the most important.
   For basics on simplicial categories, $\infty$-categories and dg-categories, see \S\ref{sec:appendix} and the references there.
By  $\dgVect=\dgVect_\k$ we denote
  the dg-category of cochain complexes over $\k$. 
 We recall the standard formalism of constructible derived categories of complexes of $\k$-vector spaces
 and their functorialities
 \cite{kashiwara-schapira}, together with its $\infty$-categorical enhancement. 
 
\smallskip

 Let $\Sch$ denote the category of schemes of finite type over $\CC$. For a scheme
   $T\in\Sch$  we denote by $\Co(T)$ the category
 of  constructible complexes of sheaves of $\k$-vector spaces on
 $T(\CC)$. 
 Let $\D(T)=\Co(T)[\Qis^{-1}]$
 be the constructible derived category, i.e., the localization of $\Co(T)$ by the class of
 quasi-isomorphisms.  We denote by $\D(T)_\dg$ and $\D(T)_\infty$ the dg- and $\oo$-categorical enhancements of 
 $\D(T)$ defined as 
  in \S \ref{subsec:enh-der}.  
 If $\k=\QQ_l$,  we can use the \'etale  $l$-adic version of the constructible derived category, see \cite{O07}, 
 \cite{O15}. 
 It admits similar enhancements.

 \smallskip
 
 These categories carry the Verdier duality  functor which we denote by $\DD$. 
  For a morphism $f: S\to T$  in $\Sch$ we have the usual functorialities
 \[
 \xymatrix{
 \D(S) \ar@<.7ex>[r]^{Rf_*, f_!} & \D(T) \ar@<.7ex>[l]^{f^{-1}, f^!}
 }
 \]
with their standard adjunctions,  see \cite{kashiwara-schapira} for the case of classical topology or \cite{O07}, \cite{O15}
 for the case of \'etale topology.
 They extend to the above enhancements and we will be using these extensions.
 
\smallskip

 We denote by $\omega_T = p^!\,  \k$, $p: T\to\pt$,   the dualizing complex of $T$. 
 The  {\em Borel-Moore homology} of $T$ and its complex of Borel-Moore chains are defined by 
   \be\label{eq:HBM}
 H_\bullet^\BM(T) \,=\, H^{-\bullet} (T, \omega_T),\quad
 C_\bullet^\BM(T)=R\Gamma(T, \omega_T).
 \ee
 The Poincar\'e-Verdier duality implies that
 \be\label{eq:PVD}
 H_\bullet^\BM(T) \,=\, H^\bullet_c(T)^*. 
 \ee
    
 \smallskip
 
  A morphism $f: S\to T$ in $\Sch$ is called {\em strongly orientable of relative dimension $m\in\ZZ$,} if there is an isomorphism
  $\ul\k_S\to f^! \ul\k_T[m] $ in $\D(S)$. A choice of such an isomorphism is called a
  {\em strong orientation} of $f$. 
For not necessarily connected $S$ we can speak of relative dimension being a locally constant function
  on $S$, with the obvious modifications of the above. 
  
  \smallskip
  
  Recall that $H_\bullet^\BM$  is covariantly functorial
 with respect to proper morphisms. 
 By \eqref{eq:HBM}, an oriented morphism $f: S\to T$ of relative dimension $m$ gives rise to a pullback map
 $f^*: H_\bullet^\BM(T)\to H_{\bullet+m}^\BM(S)$, and such maps are compatible with
 compositions of oriented morphisms. 
  
  \begin{exas}\label{exas-orient}\hfill
  \begin{itemize}[leftmargin=8mm]
  \item[$\mathrm{(a)}$]
  A smooth morphism $f$ of dimension $d$ is strongly oriented of relative dimension $2d$.
  
  \item[$\mathrm{(b)}$] An l.c.i.~(locally complete intersection) morphism is
  a morphism $f: S\to T$ represented as a composition $f=p\circ i$ where $p$ is smooth and
  $i$ is a regular embedding. Thus an l.c.i. morphism $f$ has a well defined 
  {\em dimension}  $d$,
  which is a locally constant $\ZZ$-valued  function on $S$.
  If the embedding $i$ is strongly oriented, then $f$ is also strongly oriented of relative dimension $2d$,
  hence gives rise to a pullback morphism $f^*$. Note that the map $f^*$ still make sense for any l.c.i.~morphism,
  see, e.g., \cite[\S 2.17]{O15}.
  \end{itemize}


  \end{exas}

 \smallskip
 
 \begin{ex}\label{ex:s!-spaces}
 Let $\Ec$ be a rank $r$  vector bundle on $T$. 
 We recall that the $r$th Chern class $c_r(\Ec)\in H^{2r}(T,\k)$
 is the obstruction to the existence of a  continuous  section of $\Ec$ which does not vanish anywhere. 
 Let $s$ be any  section of $\Ec$. 
 We denote the zero locus of $s$ with its embedding into $T$ by $T_s\buildrel i_s\over\to T$.
 In this situation we have the {\em refined $r$th Chern class}
 \[
 c_r(\Ec,s) \,\in\, H^{2r}_{T_s}(T, \k) \,=\, H^{2r}(T_s, i_s^! \ul\k_T)
 \]
 whose image in $H^{2r}(T,\k)$ is   $c_r(\Ec)$, yieldding
 a \emph{virtual pullback} map $ s^!: H_\bullet^\BM(T) \lra H_{\bullet-2r}^\BM(T_s).$
More precisely, following  
   \cite[\S 7.3]{fulton-macpherson},  we introduce
 the {\em bivariant cohomology} of any morphism $f:S\to T$ to be
 \[
 H^\bullet(S\buildrel f\over\to T) \,=\, H^\bullet(S, f^! \ul\k_T). 
 \]
 Recall that
 \begin{itemize}[leftmargin=8mm]
\item[$\mathrm{(a)}$]
 We have $H^\bullet(S\buildrel\Id\over\to S) = H^\bullet(S, \k)$ while 
 $H^\bullet(S\to\pt) = H_{-\bullet}^\BM(S)$. 
 
\item[$\mathrm{(b)}$] 
For a composable pair of maps $S\buildrel f\over\to T\buildrel g\over\to U$ we have the product
$$H^\bullet(S\buildrel f\over\to T) \otimes H^\bullet ( T\buildrel g\over\to U) \lra 
H^\bullet(S\buildrel gf\over \to U ).$$
 So, taking $U=\pt$, each $h\in H^d(S\buildrel f\over\to T)$ gives
 rise to a map $u_h: H_\bullet^\BM(T)\to H_{\bullet-d}^\BM(S)$.
 \end{itemize}
We deduce that 
 $c_r(\Ec,s) \,\in\,  H^{2r}(T_s\buildrel i_s\over \to T)$ defines a map
 $H_\bullet^\BM(T) \lra H_{\bullet-2r}^\BM(T_s).$
 
 \smallskip
 
 The construction of $c_r(\Ec,s)$ is as follows. We consider  the embedding $T \buildrel 0\over \to 
 \Tot(\Ec)$ 
 as the zero section.  It is strongly oriented of relative dimension $2r$, 
 see  \cite[prop.~4.1.3, 7.3.2]{fulton-macpherson},
 hence we get a class $\eta\in H^{2r}_T(\Tot(\Ec))$.
  Now $T_s$ is the intersection of $T$ with $\Gamma_s$,
 the graph of $s$
 inside $\Tot(\Ec)$, and  $c_r(\Ec,s)$ is the  image of  $\eta$ under the restriction map
 \[
 H^{2r}_T(\Tot(\Ec), \ul\k) \lra H^{2r}_{T\cap \Gamma_s} (\Gamma_s, \ul \k)   \,=\, H^{2r}_{T_s}(T,\ul\k). 
 \]
 See also  \cite[\S 2.17]{O15} for a different approach. 
\end{ex}  

\begin{prop}
Let $\Ec$ be a vector bundle on $T$ of rank $r$ and let  $p: \Tot(\Ec)\to B$ be the projection.
The pullback $p^*: H_\bullet^\BM(T)\to H_{\bullet+r}^\BM(\Tot(\Ec))$
is an isomorphism. \qed
\end{prop}

\begin{rem}
For $T\in \Sch$ let $A_m(T)$ be the Chow group of $m$-dimensional cycles in $T$.
We have the canonical {\em class map} $\on{cl}: A_m(T)\to H_{2m}^\BM(T)$. All the above constructions
(proper pushforwards, l.c.i. pullbacks, virtual pullbacks) have natural analogs for the Chow groups,
see \cite{F},
which are compatible, via $\on{cl}$,  with the sheaf-theoretical constructions described above.  
\end{rem}

\medskip

  
  \subsection{BM homology and operations for stacks} 
  
  The formalism of constructible derived categories and their functorialities extends to f-Artin stacks. For the case $\k=\QQ_l$ and
  \'etale topology this is done in \cite{O07, O15}. Another approach using $\oo$-categorical limits, which we outline below,
   is applicable for the complex analytic topology, any $\k$, as well
  as for  the case of analytic stacks  in \S \ref{subsec:anal-stack}.
   It is an adaptation of    the
  approach used in \cite{GR}, \S 3.1.1 for ind-coherent sheaves,  to
  the  constructible case. 
     All stacks in this sections will be f-Artin.

\smallskip
   
  Let $B$ be a stack.  By 
  $\Sch_B$ we denote the category formed by schemes  $T$ of finite type over $\CC$ together with a
  morphism of stacks $T\to B$. 
  We define
  \be\label{eq: D-holim}
 \D(B)_\infty \,=\, \varprojlim_{\{T\to B\} } \D(T)_\infty, \quad \D(B)_\dg \,=\, \varprojlim_{\{T\to B\} } \D(T)_\dg,
  \ee
  the $\infty$-categorical  projective limit,  resp. dg-categorical (homotopy) projective limit
  over the category  $\Sch_B$, with respect to the pullback functors. 
  Note that  $\D(B)_\infty$, resp.    $\D(B)_\dg$ also carries the Verdier duality $\DD$ induced by such dualities
  on the $\D(T)_\infty$, resp.  $\D(T)_\dg$ above.

 \smallskip
 
 We compare this with the following. 
  Let $Z$ be a scheme of finite type over $\CC$ with an action of an affine algebraic group $G$.
 Then we have action groupoid $\{G\times Z\rightrightarrows Z\}$ in the category of schemes, so its
 nerve $N_\bullet\{G\times Z\rightrightarrows Z\}$ is a simplicial scheme defined as in
 \eqref{eq:nerve}. 
 The {\em Bernstein-Lunts
 equivariant derived constructible $\infty$-category} of $Z$ is 
 \[
 \D(Z,G)_\infty \,=\,\varprojlim_{[n]\in\Delta^\circ} \D(N_n\{G\times Z\rightrightarrows Z\})_\infty.
 \]
 It is an $\infty$-categorical version of the definition from \cite{bernstein-lunts}. 
 Just as in \cite{bernstein-lunts},
 given  a $G(\CC)$-equivariant constructible complex  $\Fc^\bullet$ on $Z(\CC)$,
 then 
 \[
 \Ext^\bullet_{\D(Z,G)_\infty} (\ul \k_Z, \Fc^\bullet) \,=\, H^\bullet_{G(\CC)}(Z(\CC), \Fc^\bullet)
 \]
 is the topological equivariant (hyper)cohomology. 
 
 \smallskip

 \begin{prop}
 The $\infty$-category $\D(Z,G)_\infty$ is identified with $\D(Z/\!/G)_\infty$. 
 \end{prop}
 
\begin{proof} Each $N_n\{G\times Z\rightrightarrows Z\}$ is an affine scheme over $Z$, therefore
 over $Z/\!/G$. In fact
 \[
 N_n\{G\times Z\rightrightarrows Z\} \,=\, Z\times_{Z/\!/G} \cdots \times_{Z/\!/G} Z \quad  (n \text{ times}). 
 \]
 So $N_\bullet\{G\times Z\rightrightarrows Z\}$ is the nerve of  the (smooth)
  morphism $Z\to Z/\!/G$,  which we can see as 
 a $1$-element covering
 of $Z/\!/G$ in the smooth topology. Our statement therefore means that  $\D(-)_\infty$
 satisfies ($\infty$-categorical) descent with respect to this covering. A more general statement
 if true: $\D(-)_\infty$ as a functor from stacks to $\infty$-categories satisfies descent
(for any covering) in the smooth topology. This statement is a formal consequence of the corresponding,
obvious, statement for
shemes: $\D(-)_\infty$ as a functor from $\Sch$ to  $\infty$-categories satisfies descent
(for any covering) in the smooth topology.
\end{proof}

\smallskip

Given a morphism of stacks $f: B\to C$, the composition with $f$ defines a functor
$f_\square: \Sch_B \to \Sch_C$, hence a functor which we denote
\[
f^{-1}: \D(C)_\infty \,=\, \varprojlim_{(U\to C)} \D(U) \,\, \lra \varprojlim _{(T\to B\buildrel f\over\to C)} \D(T)
= \D(B)_\infty. 
\]
The right adjoint functor to $f^{-1}$ is denoted by  $Rf_*: \D(C)_\infty \to \D(B)_\infty$.   
\smallskip

We further define the functors
\[
f^!  = \DD \circ f^{-1}\circ \DD: \D(C)_\infty \lra \D(B)_\infty, \quad 
Rf_! = \DD \circ Rf_*\circ\DD: \D(B)_\infty \lra \D(C)_\infty. 
\]
 In particular, we have the {\em dualizing complex} $\omega_B = \DD (\ul\k_B) = p^!(\k)$,
where $p: B\to\pt$,  cf. \cite{olsson-I}. 
Note that, for each affine algebraic group $G$ over $\CC$, 
then $\omega_{BG} \simeq \ul\k_{BG}[-2\dim(G)]$, while
for each smooth complex variety $S$ we have $\omega_S \simeq\ul\k_S[2\dim(S)]$.

\smallskip

 We define the   
 {\em Borel-Moore homology}, resp. {\em cohomology with compact support} of an  (f-Artin) stack $B$ as  
 \be
 H_\bullet^\BM(B) = H^{-\bullet}(B, \omega_B), \quad H^\bullet_c(B, \ul\k_B) = H^\bullet(Rp_! \ul\k_B). 
 \ee
 The Poincar\'e-Verdier duality extends from schemes of finite type to f-Artin stacks  and implies that 
$ H_\bullet^\BM(B)\,=\, H^\bullet_c(B, \ul\k_B)^*. $
By gluing the corresponding properties of schemes, we get that
 $H_\bullet^\BM$ is covariantly functorial for proper morphisms 
 and has pullbacks with respect to   l.c.i. morphisms.

\smallskip
 
 \begin{rem}
 The BM homology for stacks is the topological analog of
  the Chow groups for stacks
 as defined by Kresch \cite {K99}. 
 \end{rem} 
 
 We also note the following, cf.  \cite[thm.~2.1.12]{K99}.
 
 \begin{prop} \label{prop:BM-Tot-2term}
 Let $\Cc^\bullet = \{\Cc^{-1}\to\Cc^0\}$ be a two-term strictly perfect complex on $B$ of virtual
 rank $r$, with the total space $\Tot(\Cc^\bullet) = \Cc^0/\!/ \Cc^{-1}\buildrel \pi\over\to B$.
 Then $\pi$ is a smooth morphism, hence it is strongly
 oriented of relative dimension $2r$, and $\pi^*: H_\bullet^\BM(B)\to H_\bullet^\BM(\Tot(\Cc))$
 is an isomorphism if $B$ admits a stratification by global quotients
 (\cite[def.~3.5.3]{K99}), in particular if $B$ is
 locally quotient. 
 \qed
 \end{prop}

 \medskip


\subsection{Virtual pullback for a perfect complex}\label{subsec:ref-pullback}
Let $B$ be a stack and
$\Ec$ be a vector bundle of rank $r$ over $B$. 
Let $s\in H^0(B, \Ec)$ be a section of $\Ec$
and $$i:B_s = \{s=0\}\hookrightarrow B$$ be the inclusion of the zero locus of $s$, 
which is a closed substack. The section $s$ gives a regular embedding in the total space of $\calE$,
which we denote also $s: B\to\Tot(\Ec)$.  The construction of Example \ref{ex:s!-spaces}
extends (by gluing) from schemes to stacks and gives the 
  {\em refined pullback morphism},
or \emph{refined Gysin morphism}
\begin{align}\label{pullback}
s^!: H_\bullet^\BM(B) \lra H_{\bullet-2r}^\BM(B_s),
\end{align}
making the following diagram commute
\[
\xymatrix{
H_\bullet^\BM(B)
\ar[d]_{s_*}
 \ar[r]^{s^!} & H_{\bullet -2r}^\BM(B_s)
 \ar[d]^{i_*}
\\
H_\bullet^\BM (\Tot(\Ec)) & \ar[l]^-{\pi^*}_-\sim H_{\bullet-2r}^\BM(B)
}
\]

\begin{rem} The map $s^!$ is  the BM-homology analog of the refined pullback on Chow groups for
Artin stacks which is a particular case of Construction 3.6 of  \cite{M12}, or of \cite[\S 3.1]{K99}
which uses deformation to the normal cone. 
 
\end{rem}

\smallskip

Now, let $\Cc$ be a strictly $[-1,1]$-perfect complex on $B$ and
$$\pi: \Tot(\Cc^{\leq 0})  \to B,\quad q: \Tot(\tau_{\leq 0} \Cc)  \to B$$ 
be the obvious projections.  
The differential $d^1$ of $\Cc$ 
gives a section $s_\Cc$ of the vector bundle $ \pi^* \Cc^1$ on $\Tot(\Cc^{\leq 0})$ whose zero
locus is the cone stack $\Tot(\tau_{\leq 0} \Cc)$, yielding the diagram
\[
 \xymatrix{
 &\pi^*\Cc^1&\ar[l]_-{s_\Cc}\Tot(\Cc^{\leqslant 0})\\
B  & \Tot(\calC^{\leq 0})  \ar[l]_-{\pi}\ar[u]^-0&\, \Tot(\tau_{\leq 0} \calC) \ar@{_{(}->}[l]_-{i}\ar@{_{(}->}[u]_-i
}
\]
such that $q=\pi\circ i$.
By Proposition \ref{prop:BM-Tot-2term}, see also 
\cite[thm.~2.1.12]{K99},    
the pullback along $\pi$ defines a morphism
$$\pi^*: H_\bullet^\BM(B) \buildrel\sim\over \lra H_{\bullet +2\vrk(\Cc^{\leq 0})}^\BM (\Tot(\Cc^{\leq 0})),$$
which is an isomorphism if $B$ admits a  stratification by global quotients.  
 Further, we have the refined pullback map on Borel-Moore homology
$$s_\Cc^!: H_{\bullet +2\vrk(\Cc^{\leqslant 0})}^\BM (\Tot(\Cc^{\leq 0}) )
\to H_{\bullet + 2\vrk(\Cc)}^\BM (\Tot(\tau_{\leq 0} \Cc)). $$
We define the {\em virtual pullback} associated with $\Cc$ to be the composite map
$$q^!_\Cc= s_\Cc^! \circ\pi^*: H_\bullet^\BM(B) \to H_{\bullet +2\vrk(\Cc)}^\BM (\Tot(\tau_{\leq 0} \Cc)).$$

\smallskip

By Proposition \ref{prop:C},
the stack $\Tot(\tau_{\leq 0}\Cc)$ depends only on the isomorphism class of the complex 
$\Cc$ in  $D^b_\coh(B)$ and not on the choice of the presentation \eqref{eq:presentation}.

\begin{prop}\label{prop:glue}
Let $\Cc$ be a strictly $[-1,1]$-perfect complex on $B$.
The virtual pullback $q^!_\Cc$ 
depends only on the isomorphism class of the strictly $[-1,1]$-perfect complex $\Cc$ in $D^b_\coh(B)$. 
\end{prop}

\begin{proof}
Fix two presentations $\calC_1$, $\calC_2$ of the complex $\calC$ as in \eqref{eq:presentation}, with 
\begin{align*}
\xymatrix{
\calC_k=\{\calC_{k}^{-1}\ar[r]^-{d^0_{k}}&\calC_{k}^{0}\ar[r]^-{d^1_{k}}&\calC_{k}^1}\},\quad k=1,2\end{align*}
and fix a quasi-isomorphism $\phi:\calC_{1}\to\calC_{2}$.
By functoriality of the total space and the truncated total space, we have 
the commutative diagram
$$\xymatrix{
\Tot(\tau_{\leqslant 0}\calC_1)\ar@{=}[r]^-{\phi_\flat}\ar@{^{(}->}[d]_-{i_1}&
\Tot(\tau_{\leqslant 0}\calC_2)\ar@{^{(}->}[d]^-{i_2}\\
\Tot(\calC_1^{\leqslant 0})\ar[r]^-{\phi_\flat}\ar[d]_-{\pi_1}&\Tot(\calC_2^{\leqslant 0})\ar[dl]^-{\pi_2}\\
B.
}$$
We claim that the following triangle commutes
$$\xymatrix{
H_\bullet^\BM(B) \ar[r]^-{q^!_{\Cc_1}}\ar[rd]_-{q^!_{\Cc_2}}
& H_{\bullet +2\vrk(\Cc_1)}^\BM (\Tot(\tau_{\leq 0} \Cc_1))\ar@{=}[d]^-{(\phi_\flat)_*}\\
&H_{\bullet +2\vrk(\Cc_2)}^\BM (\Tot(\tau_{\leq 0} \Cc_2)).}$$
To prove this, we must prove that we have 
$$s_{\Cc_1}^!\circ\pi_1^*=\phi_\flat^* \circ s_{\Cc_2}^!\circ\pi_2^*.$$

By Proposition \ref{prop:tot(-1,0)},  the map
$\phi_\flat:\Tot(\calC_1^{\leqslant 0})\to\Tot(\calC_2^{\leqslant 0})$ is an l.c.i.
Hence there is a Gysin map $(\phi_\flat)^*$
 and we have expressions through the local Chern classes associated to the sections $s_{C_i}$
 of $\pi_i^*\Cc_i^1$, i=1,2: 
\begin{align*}
s_{\Cc_1}^!\circ\pi_1^*&=
c_{\rk(\Cc_1^1)}(\pi_1^*\Cc_1^1, s_{\Cc_1})
\circ \phi_\flat^*\circ\pi_2^*,\\
\phi_\flat^* \circ s_{\Cc_2}^!\circ\pi_2^*&=\phi_\flat^*\circ
c_{\rk(\Cc_2^1)}(\pi_2^*\Cc_2^1, s_{\Cc_2})
\circ\pi_2^*.
\end{align*}
The proposition is a consequence of the following version of the excess intersection formula.

\smallskip

\begin{lemma}
Let $f:B_1\to B_2$ be a morphism of stacks which is an l.c.i\! of relative dimension $r_2-r_1$.
Let $\calE_1$, $\calE_2$ be vector bundles on $B_1$, $B_2$ of ranks $r_1$, $r_2$ and sections
$s_1$, $s_2$ of $\calE_1$, $\calE_2$.
Let $h:\calE_1\to f^*\calE_2$ be a vector bundle homomorphism 
such that $h\circ s_1=s_2\circ f$, which yields a fiber diagram
\begin{align*}
\begin{split}
\xymatrix{
(B_2)_{s_2}\ar@{^{(}->}[r]^-{i_2}& B_2\ar[r]^-{s_2}&\Tot(\Ec_2)\\
(B_1)_{s_1}\ar[u]^g\ar@{^{(}->}[r]^-{i_1}& B_1\ar[r]^-{s_1}\ar[u]^f&\Tot(\Ec_1)\ar[u]^h
}
\end{split}
\end{align*}
where $g$ is an isomorphism.
Then, we have a commutative square
$$\xymatrix{
H^\BM_\bullet(B_1)\ar[r]^-{
c_{r_1}(\Ec_1, s_1)
}&H^\BM_{\bullet-2r_1}((B_1)_{s_1})\\
H^\BM_{\bullet-2r_1+2r_2}(B_2)\ar[r]^-{
c_{r_2}(\Ec_2, s_2)
}\ar[u]^-{f^*}&H^\BM_{\bullet-2r_1}((B_2)_{s_2}).\ar[u]^{g^*}
}$$
\qed
\end{lemma}

\end{proof}

\smallskip

Finally, let now $B$ be an Artin stack   and
let $\Cc$ be any $[-1,1]$-perfect complex on $B$.
Let $\frakU$ be a filtering open cover of $B$ consisting of
open substacks  $U$  such that $\Cc|_U$ is strictly $[-1,1]$-perfect. 
We have
\begin{align}\label{limit}
H_\bullet^\BM(B) \,=\,\varprojlim_{U\in\,\frakU} H_\bullet^\BM(U),\quad 
H_\bullet^\BM(\Tot(\tau_{\leq 0}\Cc)) \,=\,\varprojlim_{U\in\,\frakU} 
H_\bullet^\BM(\Tot(\tau_{\leq 0} \Cc|_U)). 
\end{align}

\smallskip

\begin{definition}
A \emph{coherent perfect system} on a $[-1,1]$-perfect complex $\Cc$ on $B$ is a collection of quasi-isomorphisms
$\phi_U:\Cc|_U\to\Cc_U$ and $\phi_{V\subset U}:\Cc_U|_V\to\Cc_V$ for each
$U,V\in\frakU$ with $V\subset U$ such that $\Cc_U$ is a strictly $[-1,1]$-perfect complex on $U$ 
with a presentation as in \eqref{eq:presentation}, and $\phi_V=\phi_{V\subset U}\circ\phi_U|_V$.
\end{definition}

\smallskip

Given a coherent perfect system on $\Cc$, we define the virtual pullback 
$$q^!_\Cc: H_\bullet^\BM(B) \lra H_{\bullet +2\vrk(\Cc)}^\BM (\Tot(\tau_{\leq 0} \Cc))$$ 
as the map
\begin{align}\label{vpb2}
q^!_\Cc \, =\,\varprojlim_{U\in\,\frakU} ((\phi_U)^*\circ q^!_{\Cc_U}). 
\end{align}

\smallskip

\begin{rem}\label{rem:dgstack}
If $\Cc$ is a strictly $[-1,1]$-perfect complex on the stack $B$, then its total space has a
{\em dg-stack} structure given by
\begin{align}\label{dg-tot}
\Tot(\Cc) \,=\, \Big(  \Tot(\Cc^{\leq 0})\,,\, \big( \Sym (\pi^*  (\Cc^1)^\vee [1])\,,\,  \del_s \big)
\Big), 
\end{align}
that is,  the stack $ \Tot(\Cc^{\leq 0})$ equipped with the sheaf of commutative dg-algebras which is the Koszul
complex of the section $s$ above. 
This dg-stack gives rise to a \emph{derived stack} in the sense of 
\cite{T14}.
The derived stack $\Tot(\Cc)$ depends, up to a natural equivalence, only on the isomorphism class of
the complex
$\Cc$ in $D^b_\coh(B)$. 
We expect a direct conceptual interpretation of the virtual pullback $q^!_\Cc$ in
terms of the derived stack $\Tot(\Cc)$. However, this would require a well behaved 
Borel-Moore homology theory
for derived stacks and we do not know how to do it. 
\end{rem}

 \medskip
 

\subsection{Virtual pullback for Maurer-Cartan stacks}

Let $B$ be an Artin stack of finite type and $\Gc$ be a strictly $[0,2]$-perfect dg-Lie algebra over $B$ 
as in \eqref{SPdgLA}.
We define now a  virtual pullback 
$$q^!_\Gc: H_\bullet^\BM(B) \to H_{\bullet +2\vrk(\Gc)}^\BM (\MC(\Gc))$$
using the diagram 
$$\xymatrix{
B  & \Tot(\Gc^{\leqslant 1})  \ar[l]_-{\pi}&\, \MC(\Gc) \ar@{_{(}->}[l]\ar@/^1pc/[ll]^-q\ar@{_{(}->}[l]_-{i}.}$$
In order to define the map $q^!_\Gc= s_\Gc^! \circ\pi^*$ as in \S\ref{subsec:ref-pullback},
we must check that the pullback morphism
$$\pi^*:H_\bullet^\BM(B)\to H^\BM_{\bullet+2\vrk(\Gc^{\leqslant 1})}( \Tot(\Gc^{\leqslant 1}))$$
and the refined  pullback
$$s_\Gc^!:H^\BM_{\bullet+2\vrk(\Gc^{\leqslant 1})}(\Tot(\Gc^{\leqslant 1}))\to 
H_{\bullet +2\vrk(\Gc)}^\BM (\MC(\Gc))$$
are well-defined.
The refined pullback is defined as in the previous sections, using the fact that 
$\MC(\Gc)$ is the zero locus of the section
$s$ of the bundle $\pi^*\Gc^2$ on $\Tot(\Gc^{\leqslant 1})$ associated with the curvature \eqref{eq:curv}.
The pullback map $\pi^*$ is well-defined, because $\pi$ is a vector bundle stack,
hence is smooth although non representable.

\smallskip

Next, we study the behavior of the virtual pullback under extensions of dg-Lie algebras.
Let $\Gc=\Hc\ltimes\Nc$ and $\tilde\Nc=\pi_\Hc^*\Nc$ be as in \S\ref{sec:MCS}.
Note that Proposition \ref{prop:global}  allows to write the commutative diagram
$$\xymatrix{ 
B&\Tot(\Hc^{\leqslant 1})\ar[l]_-{\pi_\Hc}&\ar@{_{(}^->}[l]_-{i_\Hc}\MC(\Hc)\\
&\ar[lu]^-{\pi_\Gc}\Tot(\Gc^{\leqslant 1})&\ar[u]_-{\pi_{\tilde\Nc}}\Tot(\tilde\Nc^{\leqslant 1})\\
&&\ar@{_{(}^->}[lu]^-{i_\Gc}\MC(\Gc)\ar@{_{(}^->}[u]_-{i_{\tilde\Nc}}.
}$$
The virtual pullback maps $q_\Gc^!$, $q_{\tilde\Nc}^!$ and $q_\Hc^!$ are defined as above.

\begin{proposition}\label{prop:REL}
We have the equality 
$q_\Gc^!=q_{\tilde\Nc}^!\circ q_\Hc^!.$
\end{proposition}
 
 \begin{proof}
Let 
$s_\Gc$, $s_{\tilde\Nc}$, $s_\Hc$ be the sections of the bundles
$\pi^*_\Gc\Gc^2$, $\pi^*_{\tilde\Nc}\tilde\Nc^2$, $\pi^*_\Hc\Hc^2$
associated with the curvature maps of $\Gc$, $\tilde\Nc$, $\Hc$ respectively.
We must prove that
$$s_\Gc^!\circ\pi^*_\Gc=s_{\tilde\Nc}^!\circ\pi^*_{\tilde\Nc}\circ s_\Hc^!\circ\pi_\Hc^*.$$
First, observe that we have the diagram whose square is a fiber square
$$\xymatrix{ 
B&\Tot(\Hc^{\leqslant 1})\ar[l]_-{\pi_\Hc}&\ar@{_{(}^->}[l]_-{i_\Hc}\MC(\Hc)\\
&\ar[lu]^-{\pi_\Gc}\ar[u]^-{p_\flat}\Tot(\Gc^{\leqslant 1})&
\ar[u]_-{\pi_{\tilde\Nc}}\ar@{_{(}^->}[l]_-{j_\flat}\Tot(\tilde\Nc^{\leqslant 1})\\
&&\ar@{_{(}^->}[lu]^-{i_\Gc}\MC(\Gc)\ar@{_{(}^->}[u]_-{i_{\tilde\Nc}},
}$$
and the maps $p_\flat$, $j_\flat$ are given by the functoriality of the total space of a $[-1,0]$-complex.
Note further that we have vector bundle homomorphisms
$$\pi^*_\Gc\Gc^2\to(p_\flat)^*\pi^*_\Hc(\Hc^2),\quad \pi^*_{\tilde\Nc}\tilde\Nc^2\to(j_\flat)^*\pi^*_\Gc\Gc^2.$$
These vector bundle homomorphisms being compatible with the sections $s_\Gc$, $s_{\tilde\Nc}$ and $s_\Hc$,
the claim follows from the functoriality of the refined pullback with  respect to  pullback by smooth maps.
\end{proof}

 \medskip
 

\section{The COHA of a surface} \label{sec:COHA}
  
  \subsection{The COHA as a vector space}\label{subsec:stack-coh}
  
  Let $S$ be a smooth connected quasi-projective surface over $\CC$. 
  Let $\Coh(S)$ be the stack of coherent sheaves on 
  $S$
  with proper support.
 It  is not smooth because the deformation theory can be 
   obstructed due to $\Ext^2$. 
  
  \begin{prop}\label{prop:coh-artin}
   $\Coh(S)$ is a  locally quotient f-Artin stack. 
  \end{prop}

  \begin{proof}  This is standard, see
    \cite[thm.~4.6.2.1]{LMB}. Here are the details for future use in Prop. \ref{eq:prop:C}. 
  Let ${\overline S}$ be a smooth projective variety containing $S$ an an open set.
  Then $\Coh(S)$ is an open substack in $\Coh({\ol S})$. So it is enough to assume that $S$ is projective which we will. 
Let $\Oc(1)$ be the ample line bundle  on $S$ induced by a projective embedding. 
The stack $\Coh(S)$ splits into disjoint union
  \[
  \Coh(S) \,=\,\bigsqcup_{h\in \k[[t]]} \Coh^{(h)}(S), 
  \]
  where $\Coh^{(h)}(S)$ consists of sheaves $\Fc$ with Hilbert polynomial $h$, i.e., of $\Fc$ such that
  \[
  \dim \, H^0(S,\Fc(n)) \,=\, h(n), \quad n\gg 0.
   \]
 For any $N\in\bbN$, let $\Coh^{(h,N)}(S)\subset \Coh^{(h)}(S)$ be the open substack formed by
 $\Fc$ such that for each $n\geq N$ two conditions hold:
\begin{itemize}[leftmargin=8mm]
  \item[$\mathrm{(a)}$] $H^i(S,\Fc(n)) = 0, \,  i>0$, 
 \item[$\mathrm{(b)}$] the canonical map $H^0(S, \Fc(n))\otimes \Oc(-n) \to \Fc$ is surjective. 
 \end {itemize}
Now, for any  coherent sheaf $\Ec$ on a scheme $B$, let $Quot_\Ec$ be the scheme
such that, for any $B$-scheme $T\to B$, the set of $T$-points 
$Quot_\Ec(T)$ is the set of surjective sheaf homomorphisms $\Ec|_T\to\Fc$ where $\Fc$ is flat over $T$, modulo 
the equivalence relation
$$(q:\Ec|_T\to\Fc)\sim(q':\Ec|_T\to\Fc')\iff\Ker(q)=\Ker(q').$$
Let $Quot^{(h,N)}(S)$ be the open subscheme  of $Quot_{\Oc(-N)^{\oplus h(N)}}$ formed by equivalence classes 
of  surjections $\phi:\Oc(-N)^{\oplus h(N)} \to\Fc$ with $\Fc\in \Coh^{(h,N)}(S)$ such that $\phi(N)$
induces an isomorphism $H^0(S, \Oc)^{\oplus h(N)} \to H^0(S,\Fc(N))$.
Then, the stack $\Coh^{(h,N)}(S)$ is isomorphic to the quotient stack of $Quot^{(h,N)}(S)$
by the obvious action of the group
$GL_{h(N)}$.
It is a stack of finite type and, as $N\to\infty$, the  substacks
$\Coh^{(h, N)}(S)$ form an open exhaustion of $\Coh^{(h)}(S)$.   
\end{proof}
  
  \smallskip
  
  \medskip

\subsection{The induction diagram}
  Let $\SES$ be the Artin stack classifying short exact sequences 
\be\label{eq:SES}
0\to\Ec \lra \Gc \lra \calF\to 0
\ee
 of coherent sheaves with proper support over $S$. Morphisms in $\SES$ 
  are isomorphisms of such sequences. We then have the {\em induction diagram} 
 \begin{equation}\label{eq:ind-diagram}
\xymatrix{
\Coh(S) \times \Coh(S) & \SES  \ar[l]_-{q} \ar[r]^-{p} & \Coh(S),
}
\end{equation}
where the map $p$ projects a   sequence \eqref{eq:SES} to $\Gc$,
while $q$ projects it to $(\Ec, \Fc)$. 

\begin{prop}\label{eq:ind-diag-proper}
The morphism $p$ is schematic (representable) and proper.
\end{prop}

\begin{proof} For any coherent sheaf $\Gc$ on $S$ with
proper support, the Grothendieck Quot scheme $Quot_\Gc$ is proper. 
\end{proof}

\medskip

\subsection{The derived induction diagram}
We have the projections
 \[
\Coh(S)\times\Coh(S) \buildrel p_{12}\over\lla \Coh(S)\times\Coh(S)\times S \buildrel p_{13}, \,p_{23} \over \lra \Coh(S)\times S.  
\]
Consider the tautological coherent sheaf
$\calU$ over $\Coh(S)\times S$ 
and the complex of coherent sheaves
   over $\Coh(S)\times \Coh(S)$ given by
 \be\label{eq:PC}
 \Cc \,=\, R(p_{12})_*\ul\RHom(p_{23}^*\Uc, p_{13}^* \Uc)[1]. 
 \ee
Its fiber at a point $(\Ec,\Fc)$ is  the complex of
 vector spaces $\RHom_S(\Fc, \Ec) [1]$. 
Given a substack $X\subset Coh(S)$, let 
$\Uc_X = \Uc|_{X\times S}$ and $ \Cc_X = \Cc|_{X\times X}$
be the restrictions of $\Uc$ and $\Cc$. 


   \begin{prop}\label{eq:prop:C}
      \hfill
\begin{itemize}[leftmargin=8mm]
  \item[$\mathrm{(a)}$]
   The complex $\Cc$ is $[-1,1]$-perfect and admits a perfect coherent system.
  \item[$\mathrm{(b)}$] 
  The complex $\Cc_X$ is strictly $[-1,1]$-perfect if $X=\Coh_0(S)$.
  \end{itemize}
   \end{prop}
   
   \begin{proof} As in the proof of Proposition \ref{prop:coh-artin}, the statements
   reduce to the case when $S$ is projective which we assume. We also keep the notation
   from that proof. 
   Fix two polynomials $h, h'\in\k[t]$ and 
   let $\Ec\in\Coh^{(h)}(S)$, 
   $\Fc\in \Coh^{(h')}(S)$ be two fixed coherent sheaves on $S$ with Hilbert polynomials $h, h'$. 
   Since $S$ is smooth of dimension $2$, we can fix a locally free resolution 
    $\Pc^\bullet = \{ \Pc^{-2}\to \Pc^{-1}\to \Pc^0\}$ of $\Fc$.  If we know that
   the $\Pc^i$ are ``sufficiently negative" with respect to $\Ec$, i.e., 
   for each $i\in [-2,0]$ and $j>0$ the space $\Ext^j_S(\Pc^i, \Ec)= H^j(S, (\Pc^i)^\vee\otimes \Ec)$ vanishes, then
   the complex  of vector spaces $\RHom_S(\Fc, \Ec)[1]$ is represented by the complex
   \be\label{eq:Hom-complex}
   \Hom_S(\Pc^0, \Ec) \to \Hom_S(\Pc^{-1}, \Ec) \to \Hom_S(\Pc^{-2}, \Ec)
   \ee
   situated in degrees $[-1,1]$. In order to achieve this, we define, in a standard way,
   \[
   \Pc^0= H^0(S, \Fc(N_0)) \otimes\Oc(-N_0) \buildrel \ev_0 \over\lra \Fc, \quad N_0 \ll 0, 
   \]
   with $\ev_0$ being the evaluation map. Then we set 
   $\Kc_0 = \Ker(\ev_0)\buildrel \eps_1\over\hookrightarrow \Pc^0$ and 
   \[
   \Pc^{-1} = H^0(S, \Kc_0(N_1)) \otimes \Kc(-N_1) \buildrel \ev_1\over\lra \Kc_0, \quad N_1 \ll N_0,
   \]
   and $\Pc^{-2}=\Ker(\ev_1)\buildrel \eps_2\over\hookrightarrow \Pc^{-1}$. Then by Hilbert's syzygy theorem, $\Pc^{-2}$ is locally free, and
   \[
  \xymatrix{
  \{ \Pc^{-2} \ar[rr]^{ d^{-2}=\eps_2} &&\Pc^{-1}\ar[rr] ^{d^{-1} = \eps_1\circ\ev_1} && \Pc^0 \} \ar[r] ^{\ev_0}  & \Fc
  }
   \]
  is a locally free resolution of $\Fc$.   
  Further, if $N_1\ll N_0 \ll 0$ are sufficiently negative with respect to $\Ec$ and $\Fc$, 
  then the dimensions (denote then $r_{-1}, r_0, r_1$)
  of the term of the complex \eqref{eq:Hom-complex} are determined by $h, h'$ and $N_0, N_1$. For
  fixed $N_1\ll N_0 \ll 0$ the locus of $(\Ec,\Fc)$  for which it is true, form an open substack $U_{N_1, N_0, h,h'}$
   in $\Coh^{(h)}(S)\times\Coh^{(h')}(S)$. On $U_{N_1, N_0, h,h'}$, the complex $\Cc$ is then represented by a
    complex of vector bundles
   whose ranks are $r_{-1}, r_0, r_1$, so it is strictly perfect. 
   Further, as $N_1, N_0\to-\infty$, the substacks $U_{N_1, N_0, h,h'}$ form an open exhaustion of 
     $\Coh^{(h)}(S)\times\Coh^{(h')}(S)$. This proves (a). 
     
     \smallskip
     
     To see (b), we notice that for $0$-dimensional $\Ec$ and $\Fc$  with given $h$ and $h'$, i.e., with
     given dimensions of $H^0(S,\Ec)$ and $H^0(S, \Fc)$,  one can choose $N_0, N_1$ in a universal way. 
     \end{proof}

\smallskip

Let now $X\subset\Coh(S)$ be a substack whose points
 are closed under extensions in $\Coh(S)$.
Let $\SES_X\subset\SES$ be the substack which classifies all short exact sequences 
of coherent sheaves over $S$ which belong to $X$.
We abbreviate $\Uc=\Uc_X$, $\Cc=\Cc_X$ and $\SES=\SES_X$. 
Assume further  that the complex $\calC$ over $X\times X$ is strictly $[-1,1]$-perfect. 
Fix a presentation of $\calC$ as in Example \ref{ex:perf-com}.

\begin{proposition}\label{eq:prop-tau-ses}
The stack $\Tot(\tau_{\leq 0}\Cc)$ is isomorphic to $\SES$. 
\end{proposition} 
 
 \begin{proof} Apply Proposition \ref{prop:A} with $Y=X\times X\times S$ and $\Fc=p_{23}^*\calU$,
$\Ec=p_{13}^*\calU$.
\end{proof}

 \smallskip

Thus, for all $X$ as above we have the following diagram of f-Artin stacks
\begin{equation}\label{eq:virdiagram}
\xymatrix{
X \times X & \Tot(\calC^{\leq 0}) \ar[l]_-{\pi}&\, \SES\ar@{_{(}->}[l]_-{i} \ar[r]^-{p} & X
}
\end{equation}
with $q=\pi\circ i$, which can be viewed as a refinement of the induction diagram \eqref{eq:ind-diagram}.
We call this diagram the \emph{derived induction diagram}.

 \medskip

\subsection{The COHA as an algebra}\label{subsec:HOHA}

 We apply the analysis of \S \ref{subsec:ref-pullback} to all diagrams \eqref{eq:virdiagram} as 
 $X$ runs over the set of open substacks of finite type of $\Coh(S)$ such that
 the complex $\calC$  in \eqref{eq:PC} is strictly $[-1,1]$-perfect over $X\times X$.
 Note that the stack $\Coh(S)$ is covered by all such $X$'s by the proof of 
Proposition \ref{eq:prop:C}.
Since the map $p$ is representable and proper, the pushforward $p_*$ in Borel-Moore homology is well-defined.
 Hence, we have the maps 
\[
H_\bullet^\BM(X\times X)  \buildrel q^!_\Cc\over \lra H_{\bullet+2\vrk(\Cc)}^\BM  ( \SES) 
\buildrel p_*\over\lra H_{\bullet+2\vrk(\Cc)}^\BM (X),
\]
which, by \eqref{limit}, give rise to the maps
\[
H_\bullet^\BM(\Coh(S)\times \Coh(S))  \buildrel q^!_\Cc\over \lra H_{\bullet+2\vrk(\Cc)}^\BM  ( \SES) 
\buildrel p_*\over\lra H_{\bullet+2\vrk(\Cc)}^\BM (\Coh(S)).
\]
Composing the maps $q^!_\Cc$, $p_*$ and the exterior product
$$H_\bullet^\BM(X) \otimes H_\bullet^\BM(X)\to H_\bullet^\BM(X\times X),$$
we get the map
\begin{align}\label{eq:virm}
m~:H_\bullet^\BM(X) \otimes H_\bullet^\BM(X) \lra H_{\bullet+2\vrk(\Cc)}^\BM(X),
\end{align}
and, by \eqref{limit}, the map
\begin{align*}
m~:H_\bullet^\BM(\Coh(S)) \otimes H_\bullet^\BM(\Coh(S)) \lra H_{\bullet+2\vrk(\Cc)}^\BM(\Coh(S)).
\end{align*}
The first main result of this paper is the following theorem.
It is proved in the next section.

\smallskip

\begin{theorem}\label{thm:ALG}
The map $m$ equips $H_\bullet^\BM(X)$ and $H_\bullet^\BM(\Coh(S))$
with an associative $\k$-algebra structure. 
\qed
\end{theorem}

\smallskip

\subsection{Proof of associativity}\label{sec:associativity}
We must prove the associativity of the map $m$. It is enough to do it for $H_\bullet^\BM(X)$.
To do that, we consider the Artin stack $\FILT$ classifying flags of coherent sheaves
$\Ec_{01}\subset \Ec_{02}\subset \Ec_{03}$ over $S$ such that the sheaves
$\Ec_{01},$ $ \Ec_{12},$ $\Ec_{23}$ defined by $\Ec_{ij}=\Ec_{0j}/\Ec_{0i}$
belong to the substack $X\subset \Coh(S)$. 
For any $i<j$ we introduce a copy $X_{ij}$ of the stack $X$
parametrizing sheaves $\Ec_{ij}$.
For any $i<j<k$ we introduce a copy $\SES_{ijk}$ of the stack $\SES$
parametrizing short exact sequences
$$0\to\Ec_{ij} \lra \Ec_{ik} \lra \Ec_{jk}\to 0.$$
Then, we have the fiber diagrams of stacks
\begin{align}\label{diag1}
\begin{split}
\xymatrix{
\FILT\ar[r]^-{x}\ar[d]_-{y}&\SES_{023}\ar[d]_{q}\ar[r]^-p&X_{03}\\
\SES_{012}\times X_{23}\ar[r]^-{p\times 1}\ar[d]_-{q\times 1}&X_{02}\times X_{23}\\
X_{01}\times X_{12}\times X_{23}
}
\end{split}
\end{align}
and
\begin{align}\label{diag2}
\begin{split}
\xymatrix{
\FILT\ar[r]^-{v}\ar[d]_-{w}&\SES_{013}\ar[d]_{q}\ar[r]^-p&X_{03}\\
X_{01}\times\SES_{123}\ar[r]^-{1\times p}\ar[d]_-{1\times q}&X_{01}\times X_{13}\\
X_{01}\times X_{12}\times X_{23}
}
\end{split}
\end{align}
given by 
\begin{align*}
x(\Ec_{01}\subset \Ec_{02}\subset \Ec_{03})&=(\Ec_{02}\subset\Ec_{03}),\quad
y(\Ec_{01}\subset \Ec_{02}\subset \Ec_{03})=(\Ec_{01}\subset\Ec_{02}\,,\,\Ec_{23}),\\
v(\Ec_{01}\subset \Ec_{02}\subset \Ec_{03})&=(\Ec_{01}\subset\Ec_{03}),\quad
w(\Ec_{01}\subset \Ec_{02}\subset \Ec_{03})=(\Ec_{01}\,,\,\Ec_{12}\subset\Ec_{13}).
\end{align*}
We must prove that we have
$$p_*\circ q_\Cc^!\circ(p_*\times 1)\circ(q_\Cc^!\times 1)=p_*\circ q_\Cc^!\circ(1\times p_*)\circ(1\times q_\Cc^!).$$
Note that the morphisms $x$, $z$ are both proper and representable and that we have 
the following equalities of stack homomorphisms 
$$(q\times 1)\circ y=(1\times q)\circ w,\quad p\circ v=p\circ x.$$
We claim that there are virtual pullback homomorphisms $y_\Cc^!$ and $w_\Cc^!$ such that 
\begin{align}\label{3REL}
\begin{split}
x_*\circ y_\Cc^!&=q_\Cc^!\circ(p_*\times 1),\\
v_*\circ w_\Cc^!&=q_\Cc^!\circ(1\times p_*),\\
y_\Cc^!\circ(q_\Cc^!\times 1)&=w_\Cc^!\times(1\times q_\Cc^!).
\end{split}
\end{align}
The complex $\Cc_{023}=(p\times 1)^*\calC$ on $\SES_{012}\times X_{23}$ and
the complex $\Cc_{013}=(1\times p)^*\calC$ on $X_{01}\times \SES_{123}$ are both
strictly $[-1,1]$-perfect. 
Since the squares in the diagrams \eqref{diag1}, \eqref{diag2} 
are Cartesian, by Proposition \ref{prop:A} we have stack isomorphisms
\begin{align*}
\Tot(\tau_{\leqslant 0}\calC_{023})=SES_{012}\times_{X_{02}} SES_{023}=\FILT,\\
\Tot(\tau_{\leqslant 0}\calC_{013})=SES_{123}\times_{X_{13}} SES_{013}=\FILT.
\end{align*}
Therefore, we have virtual pullback maps
\begin{align*}
y^!_\Cc=y^!_{\Cc_{023}}&: H_\bullet^\BM(SES_{012}\times X_{23}) \to H_{\bullet +2\vrk(\Cc)}^\BM (\FILT),\\
w^!_\Cc=w^!_{\Cc_{013}}&: H_\bullet^\BM(X_{01}\times SES_{123}) \to H_{\bullet +2\vrk(\Cc)}^\BM (\FILT)
\end{align*}
associated with the complexes $\Cc_{023}$ and $\Cc_{013}$.
Then, the first two equations in \eqref{3REL} follow from the following base change property of virtual pullbacks.

\begin{lemma}
Let $B,$ $B'$ be Artin stacks of finite type,
$\Cc$ be a strictly $[-1,1]$-perfect complex on $B$,
and  $f:B'\to B$ be a representable and proper morphism of stacks.
Then, the complex $\Cc':=f^*\Cc$ on $B'$ is strictly $[-1,1]$-perfect and gives rise to the following 
Cartesian square
$$\xymatrix{
\Tot(\tau_{\leqslant 0}\calC')\ar[r]^-g\ar[d]_-{q'}&
\Tot(\tau_{\leqslant 0}\calC)\ar[d]^-{q}\\
B'\ar[r]^-f&B.
}$$
Further, we have the following equality of maps 
$$g_*\circ {q'}^!_{\Cc'}=q_\Cc^!\circ f_*:H^\BM_\bullet(B')\to H^\BM_{\bullet+2\vrk(\Cc)}(\Tot(\tau_{\leqslant 0}\Cc)).
$$
\qed
\end{lemma}

\smallskip

Now, we concentrate on the third  equation in \eqref{3REL}. 
To do this, we first apply Proposition \ref{prop:B} to the stack homomorphism
$$p:Y=X_{01}\times X_{12}\times X_{23}\times S\to B=X_{01}\times X_{12}\times X_{23}$$ 
and to the coherent sheaves
$\Ec_{ij}=p_{ij}^*\calU$ with $ij=01,12,23$ 
given by the pullback of the tautological sheaf $\Uc$ by the obvious projections $Y\to  X\times S$.
The sheaf $\Gc$ of associative dg-algebras in \eqref{Gc} is a strictly $[0,2]$-perfect dg-Lie algebra on $B$.
So, Proposition \ref{prop:B} yields an equivalence of stacks over $B$
$$\MC(\Gc)\simeq  \FILT.$$
More precisely, we realize $\Gc$ as a semi-direct product in two ways
$\Gc=\Hc\ltimes\Nc=\Hc'\ltimes\Nc'$
where
\begin{align*}
\Nc=Rp_*\underline\Hom(\Ec_{23},\Ec_{01}\oplus \Ec_{12}),\quad
\Hc=Rp_*\underline\Hom(\Ec_{12},\Ec_{01}),\\
\Nc'=Rp_*\underline\Hom(\Ec_{12}\oplus \Ec_{23},\Ec_{01}),\quad
\Hc'=Rp_*\underline\Hom(\Ec_{23},\Ec_{12}).
\end{align*}
Then, the proof of Proposition \ref{prop:B} yields the following isomorphism of stacks
\begin{align*}
\MC(\Hc)&=\SES_{012}\times X_{23},\\
 \MC(\Hc')&=X_{01}\times \SES_{123},\\
\MC(\Gc)&=\MC(\tilde\Nc)=\SES_{012}\times_{X_{02}} \SES_{023}=\FILT,\\
\MC(\Gc)&=\MC(\tilde\Nc')=\SES_{123}\times_{X_{13}} \SES_{013}=\FILT.
\end{align*}
In particular, we can identify the diagram
\begin{align*}
\xymatrix{
&\pi^*\Cc^1_{023}&\Tot(\calC^{\leq 0}_{023})\ar[l]_-{s}\\
\SES_{012} \times X_{23} & \Tot(\calC^{\leq 0}_{023}) \ar[l]_-{\pi}\ar@{_{(}->}[u]&\, 
\FILT\ar@{_{(}->}[l]_-i\ar@/^1pc/[ll]^-y  \ar@{_{(}->}[u],
}
\end{align*}
with the diagram
\begin{align*}
\xymatrix{
&\pi^*_{\tilde\Nc}\tilde\Nc^2&\Tot(\tilde\Nc^1)/\!/\tilde N^0\ar[l]_-{s_{\tilde\Nc}}\\
\MC(\Hc) & \Tot(\tilde\Nc^1)/\!/\tilde N^0 \ar[l]_-{\pi_{\tilde\Nc}}\ar@{_{(}->}[u]&\, 
\MC(\Gc).\ar@{_{(}->}[l]_-{i_{\tilde\Nc}}\ar@/^1pc/[ll]^-{q_{\tilde\Nc}}  \ar@{_{(}->}[u],
}
\end{align*}
We deduce that $y_\Cc^!=q^!_{\tilde\Nc}$.
Similarly, we get 
$$q_\Cc^!\times 1=q_\Hc^!,\quad
w_\Cc^!=q^!_{\tilde\Nc'},\quad
1\times q_\Cc^!=q_{\Hc'}^!.$$
So the third  equation in \eqref{3REL} follows from Proposition \ref{prop:REL}.
This finishes the proof of Theorem \ref{thm:ALG}.


\medskip

\subsection{Chow groups and K-theory versions of COHA}
 
Given an f-Artin stack $B$, we denote by $A_\bullet(B)$ its rational Kresch-Chow groups, as in 
\cite{K99}. By $K(B)$ we denote the Grothendieck group of the category of coherent sheaves on $B$. 
The construction in  \S \ref{subsec:ref-pullback}  makes sense as well for $A_\bullet$ and K-theory, 
yielding virtual pullback morphisms
\begin{align*}
&q^!_\Cc:A_\bullet(\Coh(S)\times \Coh(S)) \to A_{\bullet+\vrk(\Cc)}  ( \SES),\\
&q^!_\Cc:K(\Coh(S)\times \Coh(S)) \to K ( \SES),
\end{align*}
 associated with
 the complex $\calC$  in \eqref{eq:PC}.
 Composing them with the pushforward 
$p_*:A_\bullet(\SES)\to A_\bullet(\Coh(S))$  and
$p_*:K(\SES)\to K(\Coh(S))$ by the map $p$ in
\eqref{eq:ind-diagram}, we get an associative ring structure
on $A_\bullet(\Coh(S))$ and on $K(\Coh(S))$. 

\smallskip

 A definition of the K-theoretic COHA of finite length coherent
 sheaves over $S$ was independently proposed along these lines in the recent paper of Zhao \cite{zhao}.
 
\medskip

\vfill\eject
 
 \section{Hecke operators}
\label{sec:Hecke} 

\subsection{Hecke patterns and Hecke diagrams}
We continue to assume that $S$ is a smooth quasi-projective surface over $\CC$. Recall that $\Coh(S)$ is the stack of
coherent sheaves on $S$ with proper support. 

\begin{Defi}
A {\em Hecke pattern} for $S$ is a pair  $(X,Y)$ of substacks in $\Coh(S)$ with the following properties:
\hfill
\begin{itemize}
  \item[$\mathrm{(H1)}$]
$X$ is open and $Y$ is  closed. 
 \item[$\mathrm{(H2)}$] For any short exact sequence
 \be\label{eq:ses-xxy}
0\to \Ec\lra \Gc\lra\Fc\to 0
 \ee
  with $\Gc\in X$ and $\Fc\in Y$, we have $\Ec\in X$.
  \item[$\mathrm{(H3)}$] $Y$ is closed under extensions, i.e., if  in the sequence \eqref {eq:ses-xxy} we have $\Ec, \Fc\in Y$, then $\Gc\in Y$. 
\end{itemize}

 \end{Defi}

 To a Hecke pattern $(X,Y)$ we associate a version of the induction diagram \eqref{eq:ind-diagram} which we call the
 {\em Hecke diagram}
 \be
 X\times Y \buildrel  q \over\lla \SES_{XXY} \buildrel p\over\lra X.
 \ee
Here $\SES_{XXY}$ is the moduli stack of short exact sequences \eqref{eq:ses-xxy} with $\Ec, \Gc\in X$ and $\Fc\in Y$, and
the projections
$q: \SES_{XXY}\to X\times Y ,$ $p: \SES_{XXY}\to Y$
 associate to a sequence  \eqref{eq:ses-xxy} the pair of  sheaves $(\Ec, \Fc)$ and to the sheaf $\Gc$ respectively. 
 We note the following analog of Propositions \ref{eq:ind-diag-proper} and \ref{eq:prop-tau-ses}.
 
 \begin{prop}
   \hfill
\begin{itemize}[leftmargin=8mm]
  \item[$\mathrm{(a)}$]
The morphism $p$ is schematic and proper.

  \item[$\mathrm{(b)}$]
 The morphism $q$ identifies $\SES_{XXY}$ with an open substack in $\Tot(\tau_{\leq 0}\Cc_{XY})$, where
 $\Cc_{XY}$ is the $[0,2]$-perfect complex on $X\times Y$  defined as in \eqref {eq:PC}.
 \end{itemize}
 \end{prop}
 
 \begin{proof} The fiber of $p$ over $\Gc$ consists of subsheaves $\Ec\subset\Gc$ such that $\Ec\in X$ and
 $\Gc/\Ec\in Y$.  Because of the property  (H2) we can say that it consists of $\Ec\subset\Gc$ such that $\Gc/\Ec\in Y$.
 Since $Y$ is closed in $\Coh(S)$, our fiber  is a closed part of the Quot scheme of $\Gc$ hence proper. 
 Parts (a) is proved. To prove (b), note that, 
 similarly to  Proposition   \ref{eq:prop-tau-ses}, the full  $\Tot(\tau_{\leq 0}\Cc_{XY})$ is the stack $\SES_{X?Y}$
 formed by short exact sequences \eqref{eq:ses-xxy} with $\Ec\in X$, $\Fc\in Y$ but $\Gc$ being an arbitrary coherent sheaf.
 Now, $\SES_{XXY}$ in the intersection of $\SES_{X?Y}$ with the preimage of $X\subset\Coh(S)$ under the projection 
 to the middle term. Since $X$ is open in $\Coh(S)$, we see that $\SES_{XXY}$ is open in $\Tot(\tau_{\leq 0}\Cc_{XY})$.\end{proof}

 
 \subsection{The derived Hecke action}\label{sec:heckeaction}
 Let $(X,Y)$ be a Hecke pattern for $S$. Denote   $\Hc_X = H_\bullet^\BM(X)$ and $\Hc_Y = H_\bullet^\BM(Y)$. 
 From the property (H3) we see,  as  in  Theorem \ref{thm:ALG}, that the derived induction diagram \eqref{eq:virdiagram} for $Y$ makes
$\Hc_Y$ into an associative algebra.  Further, 
similarly to \eqref{eq:virdiagram}, we have the diagram of f-Artin stacks which
 we call the {\em derived Hecke diagram}:
 \be
 \label{eq:virdiagram}
\xymatrix{
X \times Y & \Tot(\calC_{XY}^{\leq 0}) \ar[l]_-{\pi}&\, \SES_{XXY}\ar@{_{(}->}[l]_-{i} \ar[r]^-{p} & X
}
 \ee
 Here $i$ identifies $\SES_{XXY}$ with an open subset of the zero locus of a section of  the  vector bundle 
 $\pi^* \Cc^1_{XY}$ and so gives rise to the virtual pullback $i^!$. So as in \S \ref{subsec:HOHA}, 
 we define the map
 \[
 \nu:  \Hc_X\otimes \Hc_Y = H_\bullet^\BM(X)\otimes H^\BM_\bullet(Y) \lra H_{\bullet + 2 \on{vrk} \Cc_{XY}}^\BM(X) = \Hc_X.
 \]
 
 \begin{thm}
 The map $\nu$ makes $\Hc_X$ into a right module over the algebra $\Hc_Y$. 
 \end{thm}
 
 \begin{proof}  Completely similar to that of Theorem \ref{thm:ALG}. It is  based on considering $\FILT_{XYY}$, the stack of flags of coherent sheaves
$\Ec_{01}\subset \Ec_{02}\subset \Ec_{03}$ with
$\Ec_{01} \in X$ and $\Ec_{02}/\Ec_{01}$, $\Ec_{03}/\Ec_{02}\in Y.$
\end{proof}
 
 
 \subsection{Examples of Hecke patterns} 
The general phenomenon is that sheaves with support of lower dimension act, by Hecke operators,
on sheaves with support of higher dimension. We consider several refinements of the condition
on dimension of support. 
 
  \begin{Defi}
  Let   $0\leq m\leq 2$. 
   \hfill
\begin{itemize}[leftmargin=8mm]
  \item[$\mathrm{(a)}$]
  A coherent sheaf $\Fc$ on $S$ with proper support is called {\em $m$-dimensional},
  if $\dim\,  \on{Supp}(\Fc) \leq m$. We denote by $\Coh_{\leq m}= \Coh_{\leq m}(S)
  \subset\Coh$ the substack formed by $m$-dimensional sheaves. 
  
  \item[$\mathrm{(b)}$] We say that $\Fc$  is 
   {\em purely $m$-dimensional}, if   any non-zero $\Oc_S$-submodule $\Fc'\subset\Fc$ is
   $m$-dimensional.
  
  \item[$\mathrm{(c)}$] We further say that $\Fc$ is 
   {\em homologically $m$-dimensional}, if it is $m$-dimensional and for any $\CC$-point $x\in S$ we have
  $ \Ext^j_{\Oc_S}(\Oc_x, \Fc) =0$ for $0\leq j< m$.  We denote by $\Coh_m = \Coh_m(S) \subset\Coh$ the substack formed by 
  $m$-dimensional sheaves. 
  \end{itemize}
   \end{Defi}
   
   \begin{prop}
    \hfill
\begin{itemize}[leftmargin=8mm]
  \item[$\mathrm{(a)}$]
   For $m=0$, the conditions ``$\,0$-dimensional'', ``purely $0$-dimensional''  and ``homologically $0$-dimensional'' sheaves are the same.
   
   \item[$\mathrm{(b)}$] For $m=1$,  the conditions ``purely $1$-dimensional"  and     ``homologically $1$-dimensional'' 
   are the same. 
   
   \item[$\mathrm{(c)}$]For $m=2$, the condition ``purely 2-dimensional'' is  the same as ``torsion-free''   while ``homologically 2-dimensional''
 is  the same as  ``vector bundle''. 
 \end{itemize}
   \end{prop}
   
  \begin{proof} Parts (a) and (b) are obvious, as is the first statement in (c).  Let us show the second statement.
   Notice that condition of being homologically $2$-dimensional, i.e., $\Ext^j(\Oc_x, \Fc)=0$ for $j<2$ and all $x$,
   is nothing but the maximal Cohen-Macaulay condition. Since $S$ is assumed to be smooth, any 
   maximal Cohen-Macaulay sheaf
   is locally free. 
   \end{proof}

\smallskip
   
   We denote by $\Coh_m(S)$ the moduli stack of homologically 2-dimensional sheaves with proper support, and
   by
   $\Coh_\tf(S)$ denote the moduli stack of torsion-free (i.e., purely 2-dimensional) sheaves.
   
   \begin{prop}\label{prop:hecke-pairs}
   The following pairs of substacks are Hecke patterns: $(\Coh_1(S), \Coh_0(S)$, $(\Coh_2(S), \Coh_1(S))$, 
    $(\Coh_\tf(S), \Coh_0(S))$ and $(\Coh_\tf(S), \Coh_1(S))$. 
   \end{prop}
   
 To prove the proposition, we note that $\Coh_1(S)$ and $\Coh_0(S)$ are both open and closed in $\Coh(S)$. 
   Further, $\Coh_2(S)$, the stack of vector bundles, is open, as is $\Coh_\tf(S)$. Further, all these stacks are closed under extensions. 
   So it remains to prove the following.
   
   \begin{lem}
   Suppose we have a short exact sequence as in \eqref{eq:ses-xxy}. 
    \hfill
\begin{itemize}[leftmargin=8mm]
\item[$\mathrm{(a)}$]
 If $\Gc\in\Coh_m(S)$ and $\Fc\in\Coh_{m-1}(S)$, then $\Ec\in\Coh_m(S)$. 
 \item[$\mathrm{(b)}$]  
 if $\Gc\in\Coh_\tf(S)$, then $\Ec\in\Coh_\tf(S)$.
 \end{itemize}
   \end{lem}
   
   \begin{proof} (a) Since $\Ec\subset\Gc$, it is clear that $\dim\on{Supp}(\Ec)\leq m$. 
The vanishing of $\Ext^j(\Oc_x, \Ec')$
   for $j<m$ follows at once from the long exact sequence of $\Ext^\bullet(\Oc_x, -)$  induced by the short exact
   sequence above.  Part (b) is obvious: any subsheaf of a torsion free sheaf is torsion free. 
   \end{proof}
 This ends the proof of  Proposition \ref{prop:hecke-pairs}. 
 
 \begin{rem}
  The non-trivial part of the  proposition says that   homologically (or, what is the same, purely) 1-dimensional sheaves
 govern Hecke modifications of vector bundles on a surface. 
 \end{rem}


\subsection{Stable sheaves and Hilbert schemes}

Let $S$ be a smooth connected projective surface and $m=0,1$.
We can apply the construction  in \S \ref{subsec:HOHA} to the 
substack of $m$-dimensional sheaves
$X=\Coh_{\leqslant m}(S)$ of $\Coh(S)$.
We have the derived induction diagram \eqref{eq:virdiagram}, 
hence the formula \eqref{eq:virm} yields an associative multiplication on
$H_\bullet^\BM(\Coh_{\leqslant m}(S))$.

\smallskip

Now, let $P(\Ec):m\mapsto\chi(\Ec(m))$ be the Hilbert polynomial of a coherent sheaf $\Ec$ on $S$,
and $p(\Ec)=P(\Ec)/$(leading coefficient) be the reduced Hilbert polynomial.
The sheaf $\Ec$ is \emph{stable} if it is pure and $p(\Fc) < p(\Ec)$ for any proper 
subsheaf $\Fc\subset\Ec$.
Let $M_S(r,d,n)$ be the moduli space of rank $r$ semi-stable sheaves with first Chern number $d$
and second Chern number $n$.
See \cite{HL} for a general background on these moduli spaces.

\smallskip

\begin{theorem}\label{thm:action}
\hfill
\begin{itemize}[leftmargin=8mm]
\item[$\mathrm{(a)}$]
The direct image by the closed embeddings
$\Coh_{0}(S)\subset\Coh_{\leqslant 1}(S)\subset\Coh(S)$ gives algebra homomorphisms
$H_\bullet^\BM(\Coh_{0}(S))\to H_\bullet^\BM(\Coh_{\leqslant 1}(S))\to H_\bullet^\BM(\Coh(S)).$
\item[$\mathrm{(b)}$]
The algebra $H_\bullet^\BM(\Coh_{\leqslant 1}(S))^\op$ acts on 
$\bigoplus_{d,n}H_\bullet^\BM(M_S(1,d,n))$.
\item[$\mathrm{(c)}$]
The algebra $H_\bullet^\BM(\Coh_{0}(S))^\op$ acts on 
$\bigoplus_{n}H_\bullet^\BM(M_S(1,d,n))$ for each $d$.
\end{itemize}
\end{theorem}  

\begin{proof}
Part (a) follows from base change.
Parts (b), (c) are proved as in \S \ref{sec:heckeaction}.
Let us give more details on (b),  part (c) is proved in a similar way.

\smallskip

First, let us consider the following more general setting : let $X=\Coh(S)$ and $Y\subset\Coh(S)$
the substack consisting of torsion free sheaves.
Note that the substack $Y\subset X$ is both open and stable by subobjects.
We claim that the algebra $H_\bullet^\BM(X)^\op$ acts on $H^\BM_\bullet(Y)$. 
To prove this, we consider the restrictions of $\Tot(\calC^{\leq 0})$ and $\SES$ to the stack $Y\times X$ given by
$$\Tot(\calC^{\leq 0})|_{Y\times X}=\pi^{-1}(Y\times X), \quad\SES|_{Y\times X}=q^{-1}(Y\times X).$$
Then, the derived induction diagram \eqref{eq:virdiagram} gives rise to the following commutative diagram
\begin{equation*}
\xymatrix{
\Coh(S) \times \Coh(S) & \Tot(\calC^{\leq 0}) \ar[l]_-{\pi}&\, \SES\ar@{_{(}->}[l]_-{i} \ar@{=}[r]&\SES\ar[r]^-{p} & \Coh(S)\\
Y \times X \ar@{_{(}->}[u]& \Tot(\calC^{\leq 0})|_{Y\times X} \ar[l]_-{\bar\pi}\ar@{_{(}->}[u]&\, \SES|_{Y\times X} 
\ar@{_{(}->}[l]_-{\bar\imath}\ar@{_{(}->}[u]& \ar@{_{(}->}[l]_-{j}\overline{SES}\ar[r]^-{\bar p} \ar@{_{(}->}[u]& 
Y\ar@{_{(}->}[u]
}
\end{equation*}
where $\overline{SES}=p^{-1}(Y)$ and $j$ is the obvious open immersion of stacks
$j:\overline{SES}\subset \SES|_{Y\times X}.$
Let $\bar s_\Cc$ be the restriction of the section 
$s_\Cc$ of $\pi^*\Cc^1$ to $Y\times X$.
We define a map
\begin{align}\label{eq:virmbis}
\bar m~:H_\bullet^\BM(Y) \otimes H_\bullet^\BM(X) \lra H_{\bullet+2\vrk(\Cc)}^\BM(Y)
\end{align}
as the composition of the exterior product and the composed map 
$\bar p_*\circ \bar j^*\circ \bar s_\Cc^!\circ\bar\pi^*$.
We claim that the map $\bar m$ above defines an action of 
the algebra $H_\bullet^\BM(X)^\op$ on $H_\bullet^\BM(Y)$.
Then, the diagrams \eqref{diag1}, \eqref{diag2} yield the following fiber diagrams of stacks
\begin{align}\label{diag3}
\begin{split}
\xymatrix{
\overline{FILT}\ar[r]^-{x}\ar[d]_-{y}&\overline{\SES}\ar[d]_{q}\ar[r]^-p&Y\\
\overline{\SES}\times X\ar[r]^-{p\times 1}\ar[d]_-{q\times 1}&Y\times X\\
Y\times X\times X
}
\end{split}
\end{align}
and
\begin{align}\label{diag4}
\begin{split}
\xymatrix{
\overline{FILT}\ar[r]^-{v}\ar[d]_-{w}&\overline{\SES}\ar[d]_{q}\ar[r]^-p&Y\\
Y\times\SES\ar[r]^-{1\times p}\ar[d]_-{1\times q}&Y\times X\\
Y\times X\times X,
}
\end{split}
\end{align}
where $\overline{FILT}\subset FILT$ is the open substack
classifying flags of coherent sheaves
$\Ec_{01}\subset \Ec_{02}\subset \Ec_{03}$ over $S$ such that 
$\Ec_{01}, \Ec_{02}, \Ec_{03}\in Y$. 
Then, the claim is proved as in \S\ref{sec:associativity}, replacing the diagrams
\eqref{diag1}, \eqref{diag2} by \eqref{diag3}, \eqref{diag4}.

\smallskip

Now,
a rank 1 coherent sheaf is stable if and only if it is torsion free. 
Thus, setting $X=\Coh_{\leqslant 1}(S)$ and $Y\subset\Coh(S)$ to be the susbtack
consisting of rank 1 torsion free sheaves, the argument above proves the part (b).

\end{proof}

\smallskip

\begin{remark}
\hfill
\begin{itemize}[leftmargin=8mm]
\item[$\mathrm{(a)}$]
The moduli space
$M_S(1,\Oc_S,n)$ of rank one sheaves with trivial determinant and second Chern number $n$
is canonically isomorphic to the Hilbert scheme $\text{Hilb}^n(S)$.
If $S$ is a $K3$ surface, then $\text{Hilb}^n(S)$ is further isomorphic to $M_S(1,0,n)$.
\item[$\mathrm{(b)}$]
The rings  $A_\bullet(\Coh_{\leqslant 1}(S))^\op$ 
and $K(\Coh_{\leqslant 1}(S))^\op$ act on 
$$\bigoplus_{d,n}A_\bullet(M_S(1,d,n)),\quad \bigoplus_{d,n}K(M_S(1,d,n))$$
respectively, as in Theorem \ref{thm:action}.
The proofs are analogous to the proof in Borel-Moore homology.
\end{itemize}
\end{remark}

\medskip


\section{The flat COHA}\label{sec:local-Hec}

 \subsection{$R(\AAA^2)$ and  commuting varieties} 
In this section we assume $S=\AAA^2$ and denote
\[
R(\AAA^2) \,=\, H_\bullet^\BM(\Coh_0(\AAA^2))
\]
the  COHA of $0$-dimensional coherent sheaves on $\AAA^2$.  We note that
\[
\Coh_0(\AAA^2) \,=\, \bigsqcup_{n\geq 0} \Coh_0^{(n)}(\AAA^2),
\]
where $\Coh_0^n(\AAA^2)$ is the stack of $0$-dimensional sheaves $\Fc$ such that the
{\em length of $\Fc$}, i.e., 
$\dim \, H^0(\Fc)$, is equal to $n$. 
We further recall that 
\[
\Coh_0^{(n)}(\AAA^2) \,\simeq \, C_n/\!/ GL_n,
\]
where  $C_n$ is  the $n\times n$ {\em commuting variety} 
\[
C_n\,=\,\bigl\{ (A,B)\in \gl_n(\CC)\times\gl_n(\CC)\, ;\, [A,B]=0\bigr\}, 
\]
  acted upon by $GL_n$ (simultaneous conjugation). Indeed, a  $0$-dimensional coherent sheaf $\Fc$ on
  $\AAA^2$ of length $n$ is the same as a $\CC[x,y]$-module $H^0(\Fc)$ which has dimension $n$ over $\CC$,
  i.e., can be represented by the space $\CC^n$ with two commuting operators $A,B$, the actions of $x$ and $y$. 
  We recall.

\begin{prop}\label{prop:NCn-irr}
$C_n$ is an irreducible variety of dimension $n^2+n$.
Therefore $\Coh_0^{(n)}(\AAA^2)$  is an irreducible stack of
dimension $n$.  \qed
\end{prop}

Accordingly, we have a direct sum decomposition
\[
R(\AAA^2) \,=\,\bigoplus_{n\geq 0} R^{n}(\AAA^2), \quad R^{n}(\AAA^2) \,=\, H_\bullet^\BM(\Coh_0^{(n)}(\AAA^2))\,=\,
H_\bullet^{\BM}(C_n/\!/ GL_n),
\]
where on the right we have the equivariant Borel-Moore homology of the topological space $C_n$. 
The algebra $R(\AAA^2)$ 
  has a $\ZZ^2$ grading (compatible with multiplication), consisting of
(in this order):
\begin{itemize}[leftmargin=8mm]
\item[$\mathrm{(a)}$]   the {\em length degree}, by the  decomposition into the $ \Hc_{\{x\}}^{(n)}$,

\item[$\mathrm{(b)}$]  the  {\em homological degree}, where we put $H_i^\BM$ in degree $i$. 
\end{itemize}
Define the $\ZZ^2$-graded vector space
\be\label{eq:space-V}
\Theta \,=\,  q^{-1}t\cdot \k[q, t], \quad \deg(q) = (0,-2), \,\,\,\deg(t) = (1,0). 
\ee
The following is well known, see, e.g., \cite[\S 5.3]{SV18} and the references there, 
and goes back to the Feit-Fine formula for the number of points in the commuting varieties
over finite fields \cite[(2)]{feit-fine} and  the purity of the Borel-Moore homology
of the commuting stack $C_n/\!/ GL_n$ proved in \cite{D}.

\begin{prop}\label{prop:R-dims}
  As a $\ZZ^2$-graded vector space, $R(\AAA^2) \simeq \Sym(\Theta)$. 
  \qed
 \end{prop}

\smallskip
 
The goal  of this section is to prove the following.
 
 \begin{thm}\label{thm:flat}\hfill
 \begin{itemize}[leftmargin=8mm]
\item[$\mathrm{(a)}$] $\Theta$ has a natural structure of a graded Lie algebra.
\item[$\mathrm{(b)}$] $R(\AAA^2)$ is isomorphic to $U(\Theta)$ as a graded algebra.
\item[$\mathrm{(c)}$] The symmetrized product map yields a graded vector space isomorphism
$ \Sym(\Theta)\simeq R(\AAA^2)$. 
 \end{itemize}
 \end{thm}

 Before to do this, let us observe the following.
 
 \begin{prop}
The algebra $R(\AAA^2)$ is the same as the COHA considered in \cite[\S 4.4]{SV13} in the case of the Jordan quiver.
 \end{prop}
 
 \begin{proof}
 To prove this, we abbreviate
 $X_n = C_n/\!/GL_n$, $S=\AAA^2$,  and note that the tautological sheaf 
$\calU$ over $X_n\times S$ 
is identified with the $GL_n$-equivariant
sheaf over $C_n\times S$ given by
$\calU=\bbC^n\otimes\calO_{C_n},$
with the $\calO_{C_n}$-linear action of $\calO_S=\bbC[x,y]$ 
such that $x,$ $ y$ act as $A\otimes 1$, $B\otimes 1$ respectively on the fiber $\calU|_{(A,B)}$.
Let $\frakp$ be the Lie algebra consisting of $(n,m)$-uper triangular matrices in $\frakgl_{n+m}$,
and let $\fraku$, $\frakl$ be its nilpotent radical and its standard Levi subalgebras.
Let $P$, $U$ and $L$ be the corresponding linear groups.
Write $X_{n,m}=X_n\times X_m$ and $C_{n,m}=C_n\times C_m$. 
We identify $C_{n,m}$ with the commuting variety of the Lie algebra $\frakl$ and
$X_{n,m}$ with the moduli stack $C_{n,m}/\!/L.$
We have $\fraku=\Hom_\bbC(\bbC^n,\bbC^m)$, and
the perfect [-1,1]-complex $\calC$ over $X_{n,m}$
in \eqref{eq:PC} is identifed with the $L$-equivariant
Koszul complex of vector bundles over $C_{n,m}$
 given by 
$$\xymatrix{\fraku\otimes\calO_{C_{n,m}}\ar[r]^{d^0}& \fraku^2\otimes\calO_{C_{n,m}}\ar[r]^{d^1}& 
\fraku\otimes\calO_{C_{n,m}}},$$
where the differentials over the $\bbC$-point $(A,B)$ in $C_{n,m}$ are given respectively by
$$d^0(u)= ([A,u]\,,\,[B,u]),\quad d^1(v,w)= [A,w]-[B,v]=[A\oplus v, B\oplus w],$$
and the direct sum holds for the canonical isomorphism $\frakl\times\fraku\to\frakp$.
The total space $\Tot(\calC)$ of this complex, defined in \eqref{dg-tot},
 is a smooth derived stack over $X_{n,m}$ such that :

\begin{itemize}[leftmargin=8mm]
\item[$\mathrm{(a)}$] 
The underlying Artin stack is the vector bundle stack $\Cc^0/\!/\Cc^{-1}$ over $X_{n,m}$ such that
$$\Cc^{-1}=(C_{n,m}\times\fraku)/\!/L,\quad \Cc^0=(C_{n,m}\times\fraku^2)/\!/L.$$ 
It is isomorphic to the
following quotient relatively to the diagonal $P$-action
$$\Tot(\calC^{\leqslant 0})=(C_{n,m}\times\fraku^2)/\!/P.$$
\item[$\mathrm{(b)}$] The structural sheaf of derived algebras is the 
free $P$-equivariant graded-commutative $\calO_{C_{n,m}\times\fraku^2}$-algebra
generated  by the elements of $\fraku^\vee$ in degree -1. The differential is given by the transpose of the Lie bracket
$\fraku\times\fraku\to\fraku$.
\end{itemize}
Therefore, the derived induction diagram \eqref{eq:virdiagram} is
\begin{align}\label{DI}
\xymatrix{
C_{n,m}/\!/L &(C_{n,m}\times\fraku^2)/\!/P \ar[l]_-{\pi}&
\widetilde C_{n,m}/\!/P\ar@{_{(}->}[l]_-{i} \ar[r]^-{p} & C_{n+m}/\!/GL_{n+m},
}
\end{align}
where $\widetilde C_{n,m}$ is the commuting variety of the Lie algebra $\frakp$.
We can now compare the product
$$m:H_\bullet^\BM(X_n) \otimes H_\bullet^\BM(X_m) \to H_\bullet^\BM(X_{n+m})$$
in \eqref{eq:virm}
with the multiplication in \cite[\S 4.4]{SV13}. 
We have the fiber diagram of stacks
$$\xymatrix{
(C_{n,m}\times\fraku)/\!/P&
\ar[l]_-f(C_{n,m}\times\fraku^3)/\!/P&
\ar[l]_{s}(C_{n,m}\times\fraku^2)/\!/P\\
C_{n,m}/\!/P\ar@{_{(}->}[u]^-{1\times 0}&\ar[l]_-{\pi}\ar@{_{(}->}[u]^-{1\times 0}
(C_{n,m}\times\fraku^2)/\!/P&
\ar@{_{(}->}[l]_-{i}\widetilde C_{n,m}/\!/P.\ar@{_{(}->}[u]_-{i}}$$
where 1 is the identity, 0 is the zero section,
$f$ is the projection to the third component of $\fraku^3$ 
(which is a local complete intersection morphism) and
$s=1\times d^1$. 
Hence, the composed map $g=f\circ s$ is the Lie bracket 
$(A,B;v,w)\mapsto[A\oplus v,B\oplus w]$ and
the composition rule of refined pullback morphisms implies that
\begin{align*}
g^!(x)=s^! f^!(x)
=s^!\pi^* (x)
\end{align*}
in $H_\bullet^\BM(\widetilde C_{n,m}/\!/P)$ 
 for any class $x\in H_\bullet^\BM(X_n\times X_m)$.
We deduce that the multiplication map $m$
is the same as the multiplication considered in \cite[\S 4.4]{SV13}.
\end{proof}
 
 
 \subsection{$R(\AAA^2)$ as a Hopf algebra}\label{subsec:R-bialg}
 
 As a first step in the proof of Theorem \ref{thm:flat}, we introduce on $R(\AAA^2)$ a compatible comultiplication.
 
\smallskip
 
 Let $U\subset \CC^2$ be any open set in the complex analytic topology. We denote by $Coh_0(U)$ the category
 of $0$-dimensional coherent analytic sheaves on $U$. The corresponding moduli stack $\Coh_0(U)$
 can be understood as a complex analytic stack in the sense of \cite{porta}, i.e., as a stack of groupoids
 on the site of Stein complex analytic spaces. It can also be understood in a more elementary way, as follows. 
 
\smallskip
 
 Let $C_n(U)\subset C_n$ be the open subset (in the complex analytic topology) formed by pairs $(A,B)$
 of commuting matrices for which the joint spectrum (the support of the corresponding coherent sheaf
 on $\CC^2$) is contained in $U$. It is, therefore,  a complex analytic space. Then we can define
 \[
 \Coh_0^{(n)}(U) \,=\, C_m(U)/\!/ GL_n(\CC),
 \]
as  the quotient analytic stack, and put
 \[
 \Coh_0(U)\,=\, \bigsqcup_{n\geq 0} \Coh_0^{(n)}(U).
 \]
 Using this understanding, we define directly 
 \[
 R(U) \,=\, H_\bullet^\BM (\Coh_0(U)) \, = \, \bigoplus_{n\geq 0} H_{\bullet}^{\BM} (C_n(U)/\!/GL_n(\CC))
\, = \, \bigoplus_{n\geq 0} R^{n}(U).
 \]
 The same considerations as in \S \ref{sec:COHA} make $R(U)$ into a graded associative algebra. 
 
\smallskip
 
 If $U'\subset U\subset \CC^2$ are two open sets, then $C_n(U')\hookrightarrow C_n(U)$ is an open
 embedding, and we have maps of $\ZZ$-graded, resp. $\ZZ^2$-graded vector spaces
 \[
 \begin{gathered}
 \rho_{U, U'}^{n}: H_\bullet^{\BM}(C_n(U)/\!/GL_n(\CC)) \lra H_\bullet^{\BM}(C_n(U')/\!/GL_n(\CC)),
 \\
 \rho_{U, U'} \, = \bigoplus_{n\geq 0} \rho_{U,U'}^{n}: R(U)\lra R(U'). 
 \end{gathered}
 \]
 
 \begin{prop}\label{prop:R(U)-properties}
  \hfill
\begin{itemize}[leftmargin=8mm]
\item[$\mathrm{(a)}$]
 $\rho_{U, U'}$ is an algebra homomorphism.
 
\item[$\mathrm{(b)}$] If the embedding $U'\hookrightarrow U$ is a homotopy equivalence, then $\rho_{U, U'}$ is an
 isomorphism. 
 
\item[$\mathrm{(c)}$] If $U$ is a disjoint union of open subsets $U_1, \cdots, U_m$, then
 \[
 R(U) \,\simeq \, R(U_1) \otimes \cdots \otimes R(U_n).
 \]
 \end{itemize}
 
 \end{prop}

\begin{proof} Part (a) is clear from definitions. To show (b), we note that $C_n(U)$ and $C_n(U')$ are naturally
 stratified (by singularities), and, under our assumption, the embedding $C_n(U') \hookrightarrow C_n(U)$
 is a homotopy equivalence relative to the stratifications, i.e., it induces homotopy equivalences on
 all the strata. By d\'evissage (spectral sequence argument) this implies that the map
 \[
 H_\bullet^{\BM, GL_n(\CC)} (C_n(U)) \,=\, H^{-\bullet}_{GL_n(\CC)} (C_n(U), \omega_{C_n(U)}) 
 \lra  H^{-\bullet}_{GL_n(\CC)} (C_n(U'), \omega_{C_n(U')}) \,=\,  H_\bullet^{\BM, GL_n(\CC)} (C_n(U'))
 \]
 is an isomorphism. 
 
\smallskip
 
 We abbreviate $GL_{n_1,\dots, n_m}=GL_{n_1}\times\dots\times GL_{n_m}$.
 Then, part (c) follows from the $GL_n(\CC)$-equivariant identifications
 \[
 C_n(U)\,=\, \bigsqcup_{n_1+\cdots + n_m=n} 
 \Big(GL_n(\CC)\times_{GL_{n_1,\dots, n_m}(\CC)}C_{n_1}(U_1) \times \cdots \times C_{n_m}(U_m)\Big),
 \]
 which reflect the fact that a length $n$ sheaf $\Fc$ on $U$ consists of sheaves $\Fc_i$ on $U_i$
 of lengths $n_i$ summing up to $n$. 
 \end{proof}
 
 \begin{cor}
 If an open $U\subset\CC^2$ is homeomorphic to a $4$-ball, then $\rho_{\CC^2, U}: R(\CC^2) \to R(U)$
 is an isomorphism. \qed
 \end{cor} 
 
 Let us now choose, once and for all, two disjoint round balls $U_1, U_2\subset \CC^2$. 
 Define a morphism of $\ZZ^2$-graded vector spaces $\Delta: R(\CC^2) \to R(\CC^2) \otimes R(\CC^2) $ as the composition
 \[
 R(\CC^2) \buildrel \rho_{\CC^2,\, U_1\sqcup U_2}\over\lra R(U_1 \sqcup U_2) \simeq R(U_1) \otimes
 R(U_2)\buildrel \rho^{-1}_{\CC^2, U_1} \otimes \rho^{-1}_{\CC^2, U_2} \over\lra
 R(\CC^2)\otimes R(\CC^2). 
  \]
 
 \begin{prop}
 \hfill
\begin{itemize}[leftmargin=8mm]
\item[$\mathrm{(a)}$]
$\Delta$ does not depend on the choice of the balls $U_1, U_2$ provided they are disjoint.
 
\item[$\mathrm{(b)}$]
 $\Delta$ makes $R(\CC^2) $ into a cocommutative Hopf algebra. 
 \end{itemize}
 \end{prop}
 
\begin{proof} Any two admissible choices of $U_1, U_2$ are connected by a path of
admissible choices, and $\Delta$ does not change along this path.
This proves (a).
To prove (b), note that all the maps in the above chain are compatible with the Hall multiplication, so 
$\Delta$ is a homomorphism of algebras. Its cocommutativity follows from (a) by interchanging $U_1$ and
$U_2$, i.e., by connecting $(U_1, U_2)$ and $(U_2, U_1)$ by a path of admissible choices. 
Coassociativity is proved similarly by considering triples of  disjoint balls.
This proves that $R(\CC^2) $ into a cocommutative  bialgebra. 

\smallskip

It remains to prove that $R(\CC^2) $ has an antipode. This is a standard argument using  co-nilpotency,
see, e.g., \cite[\S 1.2]{loday}.
That is, define 
\[
\ol\Delta: R(\CC^2) \lra R(\CC^2) \otimes R(\CC^2) , \quad \ol\Delta(x) = \Delta(x) - (x\otimes 1 + 1\otimes x),
\]
and let $\ol\Delta^n: R(\CC^2) \lra R(\CC^2)^{\otimes n}$ be the $n$-fold iteration of $\ol\Delta$. 
Then $R(\CC^2)$ is {\em co-nilpotent}, that is, for any $x\in R(\CC^2)$ there is $n$ such that
$\ol\Delta^m(x)=0$ for $m\geq n$. Therefore the antipode $\alpha:  R(\CC^2)\to R(\CC^2)$ is given by the 
following geometric series,
terminating upon evaluation on any $x\in  R(\CC^2)$:
\[
\alpha\,=\,\sum_{n=1}^\infty (-1)^n m_n\circ\ol\Delta^n,
\]
where $m_n:  R(\CC^2)^{\otimes n}\to  R(\CC^2)$ is the $n$-fold multiplication.

 \end{proof}
 
 \smallskip
 
Let $R(\CC^2) _\prim \,=\,\bigl\{ a\in R(\CC^2)  \, ; \, \Delta(a) = a\otimes 1 + 1\otimes a\bigr\}$
be the Lie algebra of primitive elements with the bracket $[a,b]=ab-ba$.

\begin{cor}\label{cor:prim}
 \hfill
\begin{itemize}[leftmargin=8mm]
\item[$\mathrm{(a)}$]
$R(\CC^2) $ is isomorphic, as a  Hopf algebra, to the universal enveloping algebra of $\R(\CC^2) _\prim$.

\item[$\mathrm{(b)}$]  $R(\CC^2) _\prim \simeq \Theta$ as a $\ZZ^2$-graded vector space. 
\end{itemize}
\end{cor} 

\begin{proof} Part (a) follows from the Milnor-Moore theorem. 
Part (b) follows from the Poincar\'e-Birkhoff-Witt theorem and from Proposition 
 \ref{prop:R-dims}. \end{proof}

 
 \subsection{Explicit primitive elements in $R(\AAA^2)$}\label{subsec:exp-prim}
 
  For any open $U\subset\CC^2$ let $\Coh^{(n)}_\Ipt(U)\subset \Coh_0^{(n)}(U)$ be the closed
  analytic substack formed by {\em 1-point} coherent sheaves, i.e., sheaves whose support consists
  of precisely one point. In other words,
  \[
  \Coh_\Ipt^{(n)}(\CC^2) \,=\, C_{n, \Ipt}(U)/\!/ GL_n(\CC),
  \]
 where $C_{n, \Ipt}(U) \subset C_n(U)$ is the closed analytic subspace formed by pairs $(A,B)$
 of commuting matrices whose joint spectrum reduces to one point in $\CC^2$ (but can be degenerate).
 Still more explicitly,
 \[
 C_{n, \Ipt}(U) \,=\, U\times \NC_n, 
 \]
 where $\NC_n$ is the $n$ by $n$ {\em nilpotent commuting variety}
 \[
 \NC_n \,=\, \bigl\{ (A,B)\in \gl_n(\CC)\times\gl_n(\CC)\,;\, [A,B]=A^n=B^n=0\bigr\}. 
 \]
 In particular, we have the closed subvariety
 \be\label{eq:c-n-1pt}
 C_{n, \Ipt} \,=\, C_{n, \Ipt}(\CC^2) \,=\, \CC^2 \times \NC_n \,\subset\, C_n, 
 \ee
 invariant under
  $GL_n(\CC)$. We recall.
  
  \begin{prop}[\cite{Ba}]\label{prop:NC-dim}
  $\NC_n$ is an irreducible algebraic variety of dimension $n^2-1$. \qed
  \end{prop}
  
  The proposition implies that $C_{n, \Ipt}$ is an irreducible variety of dimension $n^2+1$. So 
  $\Coh_{\Ipt}^{(n)}(\CC^2)$, its
  image in $\Coh_0^{(n)}(\CC^2)$, is an irreducible stack of dimension $1$, and it has the equivariant
  fundamental class
  \[
  \theta_n \,=\, [C_{n, \Ipt}] \,\in \, H_2^{\BM}(C_n/\!/ GL_n). 
  \]
 Further, let $\Ec_n$ be the trivial vector bundle of rank $n$ on the $GL_n$-variety $C_n$, equipped with the vectorial 
 representation of $GL_n$. We call $\Ec_n$ the \emph{tautological sheaf}.
 Being an equivariant vector bundle,
 it has the equivariant Chern characters
 \[
 ch_i(\Ec_n)\,\in \, H^{2i}(C_n/\!/ GL_n),\quad i\geqslant 0,
 \]
 and, for $i\geqslant 0$, $n\geqslant 1$, we define
 \be\label{eq:theta-n-i}
 \theta_{n, i} \,=\, ch_i(\Ec_n)\cap\theta_n  \,\in\, H_{2-2i}^{\BM}(C_n/\!/ GL_n) \,=\, R^{n, 2-2i}(\CC^2). 
 \ee
 Comparing the $\ZZ^2$-grading of $\Theta$, we see that the map
 \be\label{eq:alpha}
 \alpha: \Theta\lra R(\CC^2), \quad t^nq^{i-1} \mapsto \theta_{n,i}, 
 \ee
 is a morphism of $\ZZ^2$-graded vector spaces.
 
 \begin{prop}
   \hfill
\begin{itemize}[leftmargin=8mm]
\item[$\mathrm{(a)}$]
 $\alpha$ is injective, i.e., each $\theta_{n,i}$ is non-zero.
 \item[$\mathrm{(b)}$]
 $\theta_{n,i}$ is primitive.
 \end{itemize}
 \end{prop}
 
 \begin{proof} 
The claim (a) follows from  \cite[thm.~C]{D} and the explicit computations in \cite[\S 5]{D} in the case of the Jordan quiver.
More precisely, let $Q_g$ be the quiver with one vertex and $g$ loops.
For each integer $n\geqslant 0$, 
let $\Mc(Q_g)_n$ be the coarse moduli space of semisimple $n$-dimensional 
representations  of $\CC Q_g$, i.e., the categorical quotient of $(\frakg\frakl_n)^g$ by the adjoint action of $GL_n$.
We'll abbreviate $\Mc(Q_g)=\bigsqcup_{n\geqslant 0}\Mc(Q_g)_n.$
The direct sum of representations yields a finite morphism $\Mc(Q_g)\times\Mc(Q_g)\to\Mc(Q_g)$,
hence a symmetric monoidal structure on the category $\Perv(\Mc(Q_g))$ of perverse sheaves on $\Mc(Q_g)$, which allows
to consider the $n$-th symmetric power $\Sym^n(\Ec)$ for any object $\Ec$ in $\Perv(\Mc(Q_g))$.
Let $\Sym(\Ec)=\bigoplus_{n\geqslant 0}\Sym^n(\Ec)$. 
Set $g=3$ and fix an embedding $Q_2\subset Q_3$. By \cite{D}, we have
\begin{align}\label{diag11}\begin{split}
\bigoplus_{n\geqslant 0}H^\bullet_c( C_n/\!/GL_n)&=H_c^\bullet\big(\Mc(Q_3)\,,\,
\Sym\big( \Bc\Pc\Sc\otimes H_c^\bullet(B\CC^\times)\big) \big)\\
\bigoplus_{n\geqslant 0}H^\bullet_c( C_{n,\Ipt}/\!/GL_n)&=H_c^\bullet\big(\Mc(Q_3)_{\Ipt}\,,\,
\Sym\big( \Bc\Pc\Sc\otimes H_c^\bullet(B\CC^\times)\big)\big)
 \end{split}
\end{align}
where $\Bc\Pc\Sc=\bigoplus_{n>0} \Bc\Pc\Sc_n$ 
and $\Bc\Pc\Sc_n$ is, up to some shift, the constant sheaf supported on the small diagonal
$\CC^3\subset \Mc(Q_3)_n$.
For each $n$, the closed subset $\Mc(Q_3)_{n,\Ipt}\subset\Mc(Q_3)_n$ is the coarse moduli space of 
semisimple representations of $\CC Q_3$
for which the underlying $\CC Q_2$-module has a punctual support in $\CC^2$.
We have $$\CC^3\subset \Mc(Q_3)_{n,\Ipt}\subset\Mc(Q_3)_n.$$
Taking the direct summand in \eqref{diag11}
$$\Bc\Pc\Sc_n\otimes H_c^\bullet(B\CC^\times)\subset\Sym\big( \Bc\Pc\Sc\otimes H_c^\bullet(B\CC^\times)\big),$$
we get the commutative diagram
$$\xymatrix{
H_c^\bullet\big(\Mc(Q_3)_n\,,\,\Bc\Pc\Sc_n\otimes H_c^\bullet(B\CC^\times)\big)^*\ar@{^{(}->}[r]&H_\bullet^\BM(C_n/\!/ GL_n)\\
H_c^\bullet\big(\Mc(Q_3)_{n,\Ipt}\,,\,\Bc\Pc\Sc_n\otimes H_c^\bullet(B\CC^\times)\big)^*\ar@{^{(}->}[r]\ar[u]^-f
&H_\bullet^\BM(C_{n,\Ipt}/\!/ GL_n).\ar[u]^-h
}$$
The map $f$ is invertible, and
$h$ is the pushforward 
by the closed embedding
$C_{n,\Ipt}\subset C_n$.
We deduce that  the class
$ch_i(\Ec_n)\cap[C_{n,\Ipt}]$ is non-zero
in $H_{2-2i}^{\BM}(C_{n,\Ipt}/\!/ GL_n)$ and that its image by $h$
is non-zero. This image is equal to the class $\theta_{n,i}$.

\smallskip

To prove (b), given an open $U\subset \CC^2$, we define, in the same way as before,  elements
 \[
 \theta_{n,i}(U)\,\in\, R^{n, 2-2i}(U) \,=\, H_{2-2i}^{\BM}(C_{n}(U)/\!/ GL_n(\CC)).
 \]
 For $U'\subset U$ we have
 \[
 \rho_{U, U'}(\theta_{i,n}(U)) \,=\,\theta_{n,i}(U'). 
 \]
 For $U=U_1\sqcup U_2$ being a disjoint union of two opens, a
 length $n$ $0$-dimensional sheaf $\Fc$ on $U$ consists of two sheaves $\Fc_i$ on $U_i$
    of lengths $n_i$, $i=1,2$ such that $n_1+n_2=n$. This can be expressed by saying that
  \begin{align}\label{factC}
    C_n(U_1\sqcup U_2) \,=\bigsqcup_{n_1+n_2=n} 
    \Big(GL_n(\CC)\times_{GL_{n_1, n_2}(\CC)}\big(C_{n_1}(U_1) \times C_{n_2}(U_2)\big)\Big),
   \end{align}
    from which we deduce  the following identification
 \begin{align}\label{factR}
 R^n(U) \,=\, \bigoplus_{n_1+n_2=n} R^{n_1}(U_1) \otimes R^{n_2}(U_2),
 \end{align}
 Let $\Ec_{n,U}$ be the tautological sheaf of $C_n(U)$ and similarly for $U_1$, $U_2$. With
    respect to the  identification \eqref{factC}, we have
     \[
   \Ec_{n,U} \,=\bigsqcup_{n_1+n_2=n} \bigl( \Ec_{n_1,U_1} \boxtimes \Oc \, \oplus \, \Oc\boxtimes 
   \Ec_{n_2,U_2}\bigr). 
    \]
    Thus, the additivity of the Chern character gives
    \begin{align}\label{factCh}
   ch_i(\Ec_{n,U}) \,=\sum_{n_1+n_2=n} \bigl( ch_i(\Ec_{n_1, U_1}) \otimes 1 \, + \, 1\otimes 
   ch_i(\Ec_{n_2, U_2})\bigr),\quad\forall i\geqslant 0. 
    \end{align}
Since, under the identification \eqref{factR}, we have
 \[
 \theta_{n}(U) \,=\, \theta_{n}(U_1)\otimes 1 + 1\otimes \theta_{n}(U_2)
 \]
 we deduce that we have also
 \[
 \theta_{n,i}(U) \,=\, \theta_{n,i}(U_1)\otimes 1 + 1\otimes \theta_{n,i}(U_2), \quad\forall i\geqslant 0.
 \]
 Our statement follows from this and from the definition of $\Delta$ via $\rho$. 
 \end{proof}
 
 \smallskip
 
 \begin{cor}\label{cor:prim}
 The space $R(\CC^2)_\prim$ of primitive elements of $R(\CC^2)$ coincides with the image $\alpha(\Theta)$.
 It is closed under the commutator $[a,b]=ab-ba$. 
 \qed
 \end{cor}
 
 \smallskip
 
Theorem \ref{thm:flat} is proved. 
The symmetrized product map $\sigma: \Sym(\Theta) \to R(\AAA^2)$ 
is considered in details in \eqref{eq:sym-prod1} below.

\medskip

\section{The COHA of a surface $S$ and factorization  homology }\label{sec:global-Hec}

\subsection {Statement of results} 
Let  $S$ be an arbitrary smooth quasi-projective surface and $R(S)= H_\bullet^\BM(\Coh_0(S))$
be the corresponding cohomological Hall algebra. It is $\ZZ^2$-graded
by (length, homological degree). 
We introduce a global analog of the space $\Theta$ generating the flat COHA $R(\AAA^2)$
from \S \ref{subsec:exp-prim}. Let
\be\label{eq:1-pointed}
S\buildrel p_n\over\lla \Coh^{(n)}_\Ipt(S)\buildrel i_n\over\lra \Coh^{(n)}_0(S)
\ee
be the stack of 1-pointed,  length $n$ sheaves on $S$ with its canonical closed 
embedding $i_n$ into $\Coh^{(n)}_0(S)$
and projection $p_n$  to $S$ 
(so $p_n(\Fc)$ is the unique support point of $\Fc$). Proposition \ref{prop:NC-dim} 
implies that $p_n$
is a morphism  with all fibers irreducible of relative dimension $(-1)$ and therefore
$\Coh^{(n)}_\Ipt(S)$
 is irreducible of dimension $+1$.
 Moreover, we have  a natural fundamental class in $H_2^\BM(\Coh^{(n)}_\Ipt(S))$
 constructed as follows. 
 
 We consider the open subscheme
$\FCoh^{(n)}_0(S):=Quot^{(n,0)}(S)$  of the quot-scheme formed by equivalence classes 
of  surjections $\phi:\Oc^{n} \to\Fc$ with $\Fc\in \Coh^{(n)}_0(S)$ such that $\phi$
induces an isomorphism $\CC^{n} \to H^0(S,\Fc)$.
Let $\FCoh^{(n)}_\Ipt(S)\subset \FCoh^{(n)}_0(S)$ be the closed subscheme formed by equivalence classes of 
$\phi$ such that $\Fc$ is a $1$-pointed sheaf. 
Then, the stack $\Coh^{(n)}_0(S)$ is isomorphic to the quotient stack  $\FCoh^{(n)}_0(S)/\!/GL_{n}$
and $\Coh^{(n)}_\Ipt(S)$ is isomorphic to the quotient stack $\FCoh^{(n)}_\Ipt(S)/\!/ GL_n$. 
Further, we have the projection $\FCoh^{(n)}_\Ipt(S)\to S$ with fibers being identified with
the variety $\NC_n$ of pairs of nilpotent commuting matrices, see \S \ref{subsec:exp-prim}.
Since this variety is irreducible of dimension $n^2-1$,  the scheme $\FCoh^{(n)}_\Ipt(S)$
is an irreducible variety of dimension $n^2+1$ and has the fundamental class in
$H_{2(n^2+1)}^\BM (\FCoh^{(n)}_\Ipt(S))$. So the quotient  stack by $GL_n$
has the fundamental class in  $H_2^\BM(\Coh^{(n)}_\Ipt(S))$.

Therefore we have the pullback map
$p_n^\dagger$ given by the composition
\be\label{eq:p-n-dagger}
H^\BM_{\bullet}(S)=H^{4-\bullet}(S)\buildrel p_n^*\over\lra H^{4-\bullet}(\Coh^{(n)}_\Ipt(S))
\lra H_{\bullet-2}^\BM(\Coh^{(n)}_\Ipt(S)),
\ee
where the last arrow is the cap-product with the fundamental class of $\Coh^{(n)}_\Ipt(S)$.

 Define the subspace
\[
\Theta_n(S) \,=\, i_{n*} p_n^\dagger H^\BM_{\bullet}(S) \,\subset \, H_{\bullet-2}^\BM(\Coh^{(n)}_0(S))\,=\,
R^n(S). 
\]
Let $\Ec_n$ denote also the tautological sheaf on $\Coh^{(n)}_0(S)$ and further put, for $i\geq 0$,
\[
\Theta_{n,i}(S)\,=\, \Theta_n(S)\cap ch_i(\Ec_n) \,\subset \, R^n(S). 
\]

\begin{prop}\label{prop:Theta-good}
The canonical map $H_\bullet^\BM(S)\to \Theta_{n,i}(S)$ is an isomorphism.
\end{prop}

\begin{proof}
As before, we use the 
  subscheme
$\FCoh^{(n)}_0(S)$ whose quotient stack by $GL_n$ is $\Coh^{(n)}_0(S)$. 
Let $T\subset GL_n$ be a maximal torus.
Then, the fixed points locus $\FCoh^{(n)}_0(S)^T$ is isomorphic to $\FCoh^{(1)}_0(S)^n=S^n$.
Thus, we have a commutative diagram
$$\xymatrix{H_\bullet^\BM(S)\ar[r]^-{p_n^*}\ar@{^{(}->}[rd]_-{a}&
H_\bullet^{\BM,GL_n}(\FCoh^{(n)}_{\Ipt}(S))_{loc}
\ar[r]^-{i_{n*}}&H_\bullet^{\BM,GL_n}(\FCoh^{(n)}_0(S))_{loc}\\
&H_\bullet^\BM(S)\otimes H^\bullet_{GL_n,loc}\ar@{^{(}->}[r]^-\Delta\ar[u]_-{b}&
(H_\bullet^\BM(S^n)\otimes H^*_{T})^{\frakS_n}_{loc}\ar[u]_-c,
}$$
where $H^\bullet_G=H^\bullet(BG)$ and  the subscript $loc$ means the tensor product by
the fraction field $H^\bullet_{GL_n,loc}$ of $H^\bullet_{GL_n}$ over $H^\bullet_{GL_n}$.
The maps $b$, $c$ are the pushforward by the closed
embeddings $S\subset \FCoh^{(n)}_{\Ipt}(S)$ and $S^n\subset \FCoh^{(n)}_{0}(S)$,
which are invertible
by the localization theorem in equivariant cohomology. The map $\Delta$ is the diagonal embedding.
It is injective.
The map $a$ is equal to $\Id\otimes 1$, up to the cap-product by an invertible element in
$H^\bullet(S)\otimes H^\bullet_{GL_n,loc}$.
It is injective. We deduce that the map 
$$i_{n*} p_n^*: H_\bullet^\BM(S)\to H_\bullet^{\BM,GL_n}(\FCoh^{(n)}_{\Ipt}(S))$$ is injective as well.
\end{proof}

\smallskip

We define
$$\Theta(S) \,=\, \bigoplus_{n,i} \Theta_{n,i}(S) \,\subset\, R(S). $$
Thus, for $S=\AAA^2$ we have that $\Theta({\AAA^2})$ is identified with the graded space $\Theta$
from \eqref{eq:space-V}, embedded into $R$ by the map $\alpha$ as in \eqref{eq:alpha}. 
We recall that $H_\bullet^\BM(\AAA^2)$ is 1-dimensional, concentrated in homological degree $4$.
Thus shifting the grading by putting
\be\label{eq:Theta'}
\Theta'\,=\, \Theta[0,-4]\,=\, qt\cdot \k[q, t],
\ee
we have by Proposition \ref{prop:Theta-good}, an identification of $\ZZ^2$-graded vector spaces
\[
\Theta(S) \,\simeq \, H_\bullet^\BM(S)\otimes \Theta'\,\simeq \, H_\bullet^\BM(S/\!/\CC^\times)\otimes t\k[t]. 
\]
We now consider the symmetrized product map $\sigma = \sigma: \Sym(\Theta(S)) \to R(S)$ defined as
\be\label{eq:sym-prod1}
\sigma = \sum_{n\geq 0} \sigma_{n}, \quad \sigma_{n}: \Sym^n(\Theta(S)) \lra R(S), 
\quad
\sigma_{n}(v_1\bullet\cdots\bullet v_n) \,={1\over n!} \sum_{s\in S_n} v_{s(1)} *\cdots * v_{s(n)}.
\ee
Here $\bullet$ is the product in the symmetric algebra and $*$ is the Hall multiplication. 
The second main result of this paper is a version of the Poincar\'e-Birkhoff-Witt theorem
for $R(S)$ which allows us to commute its graded dimension.  It is proved in the next sections.

\begin{thm}\label{thm:PBW}
 $\sigma: \Sym(\Theta(S)) \to R(S)$  is an isomorphism 
 of $\ZZ^2$-graded
vector spaces. \qed
\end{thm}

\smallskip


\subsection{Reminder on factorization algebras}
We follow the approach of \cite{CG}  and \cite{ginot}. 
Let $(\Cc, \otimes, \1)$ be a symmetric monoidal model category. In particular, it has a class $W$ of weak equivalences.
We will consider three examples:
\begin{itemize}[leftmargin=8mm]
\item[$\mathrm{(a)}$]
$\Cc=\Top$ is the category of topological spaces (homotopy equivalent to a CW-complex),
$\otimes$ is cartesian product, and weak equivalence have the usual topological meaning.
\item[$\mathrm{(b)}$] $\Cc$ is the category of Artin stacks, $\otimes$ is the Cartesian product of stacks and
weak equivalences are equivalences of stacks. 

\item[$\mathrm{(c)}$]
$\Cc= \dgVect$ is the category of cochain complexes, $\otimes$ is the usual tensor product
and weak equivalences are quasi-isomorphisms. 
\end{itemize}
Let $M$ be a $C^\infty$ manifold of dimension $n$.

\begin{Defi}
A {\em prefactorization algebra} on $M$ valued in $\Cc$ is a rule $\Ac$ which associates
      \hfill
\begin{itemize}[leftmargin=8mm]
\item[$\mathrm{(a)}$]
 to any open set $U\subset M$ an object $\Ac(U)\in\Cc$, so that $\Ac(\emptyset) = \1$. 
\item[$\mathrm{(b)}$]
to any system $U_1,\cdots, U_p$ of disjoint open sets contained in an open set $U_0$, a morphism
$\mu_{U_1,\cdots, U_p}^{U_0}: \Ac(U_1) \otimes\cdots \otimes\Ac(U_p) \lra\Ac(U_0),$ such that

\item[$\mathrm{(c)}$]the morphisms $\mu_{U_1,\cdots, U_p}^{U_0}$ satisfy associativity.

\end{itemize}
A {\em morphism of prefactorization algebra} $\sigma: \Ac\to\Ac'$ is   a datum of
morphisms $\sigma_U: \Ac(U) \to\Ac'(U)$ compatible with the structures. It is a
{\em weak equivalence} if each $\sigma_U$ is a weak equivalence.
\end{Defi} 

A prefactorization algebra is, in particular, a \emph{precosheaf}  via the maps $\mu_{U_1}^{U_0}$, i.e., 
it is a covariant functor from the category of open subsets in $M$ to $\Cc$.

\begin{definition}
An open covering of $M$ is called a {\em Weiss covering} if any finite subset of
$M$ is contained in an open set of the covering. 
\end{definition}

\begin{example}\label{ex:fact}
     \hfill
\begin{itemize}[leftmargin=8mm]
\item[$\mathrm{(a)}$]
Let $D\subset \RR^n$ be the standard unit disk $\| x\| < 1$. 
A \emph{disk} in $M$ is an open subset which is homeomorphic to $D$.
The open covering $\Den(M)$ of $M$ generated by the disks of $M$ is a Weiss covering.
By definition, an open subset of $\Den(M)$ consists of a finite disjoint union of disks.

\item[$\mathrm{(b)}$]
A prefactorization algebra is  called \emph{locally constant}, if for any inclusion of disks $U_0\subset U_1$
the map $\mu_{U_1}^{U_0}$ is a weak equivalence.

\end{itemize}
\end{example}

\begin{Defi}
\hfill
\begin{itemize}[leftmargin=8mm]
\item[$\mathrm{(a)}$]
A prefactorization algebra $\Ac$ is called a {\em (homotopy) factorization algebra} if :

\begin{itemize} 
\item[$\mathrm{(a1)}$]
For any Weiss covering $\Uen = \{U_i\}_{i\in I}$
of any open set $U\subset M$ the natural morphism
   \[
   \begin{gathered}
  \hocolim \,  \Nc_\bullet(\Uen, \Ac)   \,\,  \lra  \,\, 
\Ac(U),
   \\
   \xymatrix{
   \Nc_\bullet(\Uen, \Ac) \,\,:=\,\,
  \biggl\{  \cdots   \ar@<.9ex>[r]   \ar@<.3ex>[r]   \ar@<-.3ex>[r]   \ar@<-.9ex>[r] &
\coprod\limits_{i,j,k\in I} \Ac(U_{ijk})  \ar@<.6ex>[r] \ar@<-.6ex>[r] \ar[r] &
\coprod\limits_{i,j\in I} \Ac(U_{ij}) \ar@<.4ex>[r] \ar@<-.4ex>[r] &  
  \coprod\limits_{i\in I} \Ac(U_i) 
\biggr\} 
},
\end{gathered} 
   \]
with $U_{ij}=U_i\cap U_j$, etc., is a weak equivaelnce {\em (co-descent)}. 

\item[$\mathrm{(a2)}$]
$\mu_{U_1,\cdots, U_p}^{U_0}$ is a weak equivalence 
for any system $U_0,\cdots, U_p$ of open sets with \hfill \break
$U_0 = U_1\sqcup\cdots\sqcup U_p$
\emph{(multiplicativity)}.
\end{itemize} 

\item[$\mathrm{(b)}$]
 The {\em factorization homology} of $M$ with coefficients in a factorization algebra $\Ac$ is the object
of global cosections of $\Ac$ which we denote
$$\int_M \Ac \,=\, \Ac(M)\,\in \, \Cc. $$

\end{itemize}
\end{Defi}

\smallskip

\begin{remark}
\hfill
\begin{itemize}[leftmargin=8mm]
\item[$\mathrm{(a)}$]
A multiplicative
prefactorization algebra $\Ac$ 
is a factorization algebra if and only if for the particular  Weiss covering $\Den(U)$ of 
any open subset $U\subset M$, the 
object $\Ac(U)$ is the homotopy colimit of the diagram
 \[
\coprod_{U_1,U_2\in\Den(U)}\Ac(U_1\cap U_2)\rightrightarrows\coprod_{U_1\in\Den(U)}\Ac(U_1).
\]
In particular, we have
$$\int_M \Ac  \, = \, \hocolim_{U\in\frakD(M)} \Ac(U).$$
See \cite[\S A.4.3]{CG} for details.
\item[$\mathrm{(b)}$] 
Any locally constant prefactorization algebra has a unique extension as a locally constant factorization algebra taking 
the same value on any disk, but possibly different values on other open sets, see \cite[rem.~24]{ginot}.

\end{itemize}
\end{remark}

 \smallskip
 
Sometimes it is convenient to use the dual language. 
By a {\em (pre)factorization coalgebra} $\Bc$ in $\Cc$ we will mean a (pre)factorization algebra in $\Cc^\op$. 
Thus, we have maps
\[
\nu_{U_0}^{U_1,\cdots, U_p}: \Bc(U_0) \lra \Bc(U_1) \otimes\cdots \otimes\Bc(U_p)
\]
yieldding a presheaf on $M$. For a factorization coalgebra $\Bc$ we have the
{\em factorization cohomology} which we denote as
\[
\oint_M \Bc \,=\, \Bc(M)  \, = \, \holim_{U\in\frakD(M)} \Bc(U). 
\]

 Let us record the following  two statements for later use.
 
\begin{prop}\label{prop:sym-sh-fa}
If $\Fc$ is a locally constant sheaf  of $\k$-dg-vector spaces, then $\Sym(\Fc): U\mapsto \Sym_\k(\Fc(U))$
 is a locally constant
factorization coalgebra. 
\end{prop}

Note that $\Sym(\Fc)$ as we define it,
 is not  the same as the symmetric algebra of $\Fc$ in the symmetric monoidal category of  sheaves
  of (dg-)vector spaces, in fact
it is not a sheaf in the usual sense.

\begin{proof}
This is an analog of \cite[thm.~5.2.1]{CG} which deals with sheaves corresponding to $C^\infty$ sections of vector bundles,
and their symmetric products in the sense of bornological vector spaces. In our case the proof is similar but easier
due to the absense of analytic difficulties.
That is,  call a covering $\Uen$ an $n$-Weiss covering, if each subset $I\subset M$ of cardinality $\leq n$ is contained
in one of the opens of $\Uen$. Then it suffices to show   that $\Sym^n(\Fc): U\mapsto \Sym^n_\k(\Fc(U))$
 satisfies descent for $n$-Weiss coverings. 
This follows, as in the proof of  \cite[thm.~5.2.1]{CG},  from the fact that $\Fc^{\boxtimes n}$ is a sheaf of $M^n$.
\end{proof}

\begin{prop}\label{prop:mor-fa-local}
Let $\sigma:\Bc\to\Bc'$ be a morphism of factorization coalgebras. Suppose that for any disk $U\subset M$ the
morphism $\sigma_U: \Bc(U)\to \Bc(U')$ is a weak equivalence. Then $\sigma$ is a weak equivalence of factorization coalgebras,
in particular, $\sigma$ induces a weak equivalence $\sigma_M: \oint_M\Bc\to\oint_M\Bc'$. 
\end{prop}

\begin{proof}
For any open $U$ we realize $\sigma_U$ 
  by descent from the Weiss cover $\Den(U)$. 
\end{proof}

 
 \subsection {Analytic stacks}\label{subsec:anal-stack}
 For the analytic version of the theory of algebraic stacks we follow \cite{porta} (where, in fact, the case
 of higher and derived  stacks is also considered). 
 
 \smallskip
 
  An {\em analytic stack} is a stack of groupoids on
  the category of (possibly singular)
 Stein analytic spaces over $\CC$, equipped with the Grothendieck topology consisting of
 open covers in the usual sense.
 Analytic stacks form a 2-category $\Stang$ as well as a model category $\Stan$ where
 weak equivalences are equivalences of stacks. 
 
 \smallskip
 
 For every scheme $Y$ of finite type over $\CC$ we have the analytic
 space $Y^\an$, the analytification of $Y$. Further, 
 or any Artin stack $X$ over $\CC$ we have the analytic stack
 $X^\an$, the  analytification
 $X^\an$ defined as the Kan extension of $X$  from the category of affine
 schemes of finite type to the category of Stein analytic spaces, see
 \cite{porta} \S 4 or \cite{holstein-porta} \S 3.  
 Note that we have a canonical map
 \be
 \eta_X: R\Gamma(X, \omega_X) \lra R\Gamma(X^\an, \omega_{X^\an}).
 \ee
 If $X=Y$ is 
 a scheme of finite type considered as an Artin stack,
  then $X^\an=Y^\an$ is the corresponding analytic space
 considered as an analytic stack. 
 
  \smallskip
 
 We will need only analytic stacks of special form, namely the {\em quotient analytic stacks} $Z/\!/K$,
 where $Z$ is an analytic space and $K$ is a complex Lie group. For such stacks various concepts
 such as Borel-Moore homology, etc., can be defined directly in terms of equivariant homology of
 the topological spaces of $\CC$-points, that is,
 \be
 H^\bullet(Z/\!/K, \omega_{Z/\!/K}) \,=\, H_{-\bullet}^{\BM, K} (Z(\CC), \CC)
 \ee 
 is the equivariant Borel-Moore homology in the topological sense.

 If $Y$ is a scheme of finite type
 over $\CC$
 and $G$ is an algebraic group over $\CC$ acting on $Y$, then $G^\an$
 is a complex Lie group and we have the quotient analytic stack
 $Y^\an/\!/G^\an$. Note that in this case we have
 \be
 (Y/\!/G)^\an \,\simeq \, Y^\an/\!/ G^\an,
 \ee
 and the map $\eta_{Y/\!/G}$ is a quasi-isomorphism, so that
 \be\label{eq:BM=alg=anal}
 H^\bullet (R\Gamma( (Y/\!/G)^\an , \omega_{(Y/\!/G)^\an})) \,\simeq \,  
 H^\bullet (R\Gamma(Y/\!/G , \omega_{Y/\!/G}))
 \,\simeq \, H_{-\bullet}^{\BM, G(\CC)} (Y(\CC), \CC)
 \ee
 is the equivariant Borel-Moore homology in the topological sense, as above. 
  \smallskip
  
\subsection {The stack $\Coh_0$ and factorization algebras}\label{AnCoh}

Let  $\Sigma$ be  complex analytic  surface. 
We view it as a $C^\infty$ manifold of dimension $4$
 and  consider open subsets $U\subset \Sigma$ in
the complex analytic topology. 
For any such  nonempty $U$  we have  the category $Coh_0(U)$
of $0$-dimensional coherent sheaves on $U$ (with finite support). 
We set $Coh_0(\emptyset)=\{\bullet\}$.
We also have 
the analytic  moduli stack $\Coh_0(U)=\bigsqcup_{n\geq 0} \Coh^{(n)}_0(U)$  parametrizing objects of 
  $Coh_0(U)$, with its components given by the length, as in the algebraic case. 
Each component is explicitly realized as a quotient analytic stack
\[
\Coh_0^{(n)}(U) \,=\,\FCoh_0^{(n)}(U)/\!/ GL_n(\CC),
\]
where $\FCoh_0^{(n)}(U)$ is the analytic space parametrizing pairs $(\Fc, \phi)$, where
$\Fc$ is a $0$-dimensional coherent sheaf on $U$ and $\phi: \CC^n\to H^0(U,\Fc)$ is an isomorphism. 
To see that $\FCoh_0^{(n)}(U)$ is well defined as an analytic space, we note that the datum of $\phi$
is equivalent to the datum of the corresponding surjection $\ul\phi: \Oc_U^{\oplus n} \to \Fc$. 
Thus $\FCoh_0^{(n)}(U)$  is a locally closed analytic subspace in $\Quot^{(n)}(\Oc_U^{\oplus n})$,
the analytic analog of the Grothendieck Quot scheme parametrizing all length $n$ quotients of $\Oc_U^{\oplus n}$. This makes clear the following fact.

\begin{prop}\label{prop:coh-alg=anal}
Let $S$ be a smooth quasi-projective algebraic surface over $\CC$. Then we have
an equivalence of analytic stacks
\[
\Coh_0(S^\an) \= \Coh_0(S)^\an. 
\]
In particular, $\Coh_0(\CC^2)$ is identified with the analytification of the Artin stack
$\Coh_0(\AAA^2)$. \qed
\end{prop}

 \smallskip

If $U_1,\dots,U_n$ are disjoint open sets contained in the open subset $U_0\subset \Sigma$,
then we have an open embedding of analytic stacks
\be\label{eq:fact-Coh}
\alpha_{U_1,\cdots, U_n}^{U_0}: \Coh_0(U_1)\times \cdots \times\Coh_0(U_n) \lra \Coh_0(U_0).
\ee

\begin{prop}\label{prop:fact}
$\Coh_0$ is a factorization algebra 
on $\Sigma$ with values in the category $\Stan$.

\end{prop}

\begin{proof}
Let $\Uen = \{U_i\}_{i\in I}$ be a Weiss open cover of $U$. Let us understand more explicitly the 
 analytic stack $\hocolim \, \Nc_\bullet(\Uen, \Coh_0)$, a homotopy colimit in the model category $Stan$, or, equivalently,
 the 2-categorical colimit of $\Nc_\bullet(\Uen, \Coh_0)$ in the 2-category $\Stang$. 
It is parametrized by pairs $(i\in I, \Fc\in Coh_0(U_i))$, 
the leftmost term  in the diagram $\Nc_\bullet(\Uen, \Coh_0)$,  subject to coherent systems of
identifications given by the rest of the diagram. These identifications say that two pairs  $(i\in I, \Fc\in Coh_0(U_i))$ and
$(j\in J, \Fc\in\Coh_0(U_j))$ are identified, whenever in the second pair 
 $\Fc$ is \emph{the same sheaf but living on $U_j$}.
This happens whenever $\Fc$ lives in fact on $U_{ij}=U_i\cap U_j$.
Further terms in the diagram $\Nc_\bullet(\Uen, \Coh_0)$ impose coherence conditions on such identifications.
This means that this homotopy colimit parametrizes $0$-dimensional
coherent sheaves which live {\em on some $U_i$}. But $\Uen$ is a Weiss cover and  every $\Fc\in Coh_0(U)$,
has finite support which, therefore, must lie in some  $U_i$. Thus, our homotopy colimit
is identified with $\Coh_0(U)$. 
\end{proof}


\subsection{Chain-level COHA as a factorization coalgebra} 

For each open set $U\subset \Sigma$ as above we consider the complex
 of Borel-Moore chains of $\Coh_0(U)$

\[
 \Rc(U) \, =  \, C_\bullet^\BM(\Coh_0(U)) \, := \, R\Gamma (\Coh_0(U), \omega_{\Coh_0(U)}).
 \]

\begin{prop}\label{cor:coha-fact-dg}
The assignment $\Rc:U\mapsto\Rc(U)$ is a locally constant factorization coalgebra on $S$ in the category 
$\Co(\Vect_\k)$ (complexes of $\k$-vector spaces). 

\end{prop}

\begin{proof}The fact that $\Rc$ it is a factorization algebra follows from  Proposition \ref{prop:fact}.
The fact that $\Rc$ is locally constant is proved in the same way as Proposition \ref{prop:R(U)-properties}(b).
\end{proof}

Next, we upgrade this statement to take into account the Hall multiplication. The relevant concept here is that of
a {\em homotopy associative} ($E_1$-)algebra which we now recall.
We will use the language of operads, see, e.g.,  \cite{CG} for a brief background and additional references.

\begin{Defi}\label{def:operads}
 Let  $(\Cc,\otimes, \1)$ be a symmetric monoidal category. 
\begin{itemize} 
\item[(a)] An operad $\Pc$  in  $\Cc $ is a  system of:

\begin{itemize}
\item[$\mathrm{(O1)}$] objects $\Pc(r)\in\Cc$ with
actions of $S_r$, given for $r\geq 0$.

\item[$\mathrm{(O2)}$] The unit morphism ${\bf1}\to \Pc(1)$.

\item[$\mathrm{(O3)}$] The 
 {\em operadic compositions} for any $k, r_1,\cdots, r_k$
 \[
\Pc(k)\otimes \Pc(r_1) \otimes\cdots\otimes\Pc(r_k) \lra \Pc(r_1+\cdots + r_k).
\]

These data must satisfy the axioms of equivariance, associativity and unit. 

  \end{itemize}  

\item[(b)] 
An algebra over an operad $\Pc$ is an object $A\in\Cc$ together with  $S_r$-equivariant morphisms
$\Pc(r)\otimes A^{\otimes r} \to A$, $r\geq 0$ satisfying the axioms of unit and associativity.

\end{itemize} 
\end{Defi}

We will use the case when $\Cc=\sset$,  $\Cc=\Top$ or $\Cc=\Co(\Vect_\k)$. We will refer to these cases as
{\em simplicial, topological} and  {\em dg-operads}. Any topological operad $\Pc$ gives a simplicial
operad $\Sing(\Pc)$ by passing to the singular simplicial sets of the $\Pc(r)$'s. It further gives a dg-operad
$C_\bullet(\Pc)$ formed by the singular chain complexes of the $\Pc(r)$ (considered, as usual,
 as cochain complexes with reverse indexation).

\smallskip

A {\em weak equivalence} of simplicial operads is a morphism $\Pc\to\Qc$ of such operads
such that for each $r$ the morphism of simplicial sets $\Pc(r)\to\Qc(r)$ is a weak equivalence,
i.e., it induces a homotopy equivalence on the  realizations. 

 \smallskip

Recall  \eqref{eq:dg-map} that the category $\Co(\Vect_\k)$  is enriched in the category
$\sset$  of simplicial sets. Thus, for any simplicial operad $\Pc$ we can speak about $\Pc$-{\em algebras} in $\dgVect$.
Such an algebra is a cochain complex $A$ together with morphisms of simplicial sets
\[
\Pc(r) \lra \Map(A^{\otimes r}, A)
\]
compatible with the $S_r$-actions and  operadic compositions.
It  sends the image of ${\bf 1}=\pt$ to the identity map. 
Dually, a $\Pc$-{\em coalgebra} in $\dgVect$ is a complex $B$ with
 morphisms of simplicial sets
\[
\Pc(r) \lra \Map(B, B^{\otimes r})
\]
satisfying similar compatibilities. 
If $\Pc$ is a topological operad, its (co)algebras in $\Co(\Vect_\k)$ are understood as (co)algebras over the simplicial
operad $\Sing(\Pc)$.


\smallskip

Let $m\geq 1$. Let $D_m$  be the topological {\em operad of little $m$-disks}. 
The space $D_m(r)$ parametrizes families $(B_1,\cdots, B_r)$ of round $m$-dimensional open balls
disjointly embedded into the standard unit ball $B=\{|x| < 1\}$ of $\RR^m$, see, e.g.,   \cite{CG}
for more details including the definition of 
  the operadic compositions. 
  
  \begin{Defi}
  By a {\em $E_m$-operad} we mean a topological operad weakly equivalent to $D_m$. 
  An \emph{$E_m$-(co)algebra} in $\dgVect$ is a (co)algebra over an $E_m$-operad.
    \end{Defi}
    We can now formulate our upgrade of the chain level COHA. 

\begin{thm}\label{prop:loc-cst}
$\Rc$ is a locally constant factorization coalgebra on $\Sigma$ in the category of
$E_1$-algebras. 
\end{thm}

 An  $E_1$-algebra can be seen  as a weakly (homotopy) associative dg-algebra, see
 discussion below.

 \subsection{Proof of Theorem \ref{prop:loc-cst}.}

 We first note that $D_1(r)$ is the union of $r!$ contractible components which are
 permuted by $S_r$. This means that algebras over $D_1$ (and so over any $E_1$-operad)
 can be described using the concept of a {\em non-symmetric} (non-$\Sigma$) operad
   \cite[Def. 9]{markl}. A non-symmetric operad in  a monoidal (not necessarily symmetric)
    category 
   $\Cc$ is a datum $\Qc$ of objects $\Qc(r)$, $r\geq 0$ (no symmetric
group action is required) together with a unit morphism $\1\to\Qc(1)$ and the compositions
as in (O3) satisfying the axioms of associativity and unit. Similarly, an algebra over
a non-symmetric operad $\Qc$ is an object $A\in\Cc$ together with morphisms
$\Qc(r)\otimes A^{\otimes r}\to A$ satisfying the axioms of unit and associativity. 
 
 Let $ND_1(r)\subset D_1(r)$ be the connected component formed by families
 $(B_1,\cdots, B_r)$ of   disjoint $1$-disks (i.e., open intervals) in $B=(-1,1)$
 such that the centers of the $B_i$ are positioned in the increasing order.
 Then $ND_1=(ND_1(r))$ is a non-symmetric operad in $\Top$ with each $ND_1(r)$ contractible.
 Let us call an $NE_1$-operad any non-symmetric operad $Q$ in $\Top$ with each $Q(r)$ contractible.
 Given an $NE_1$-operad $\Qc$, we can ``symmetrize'' it, forming an $E_1$-operad
 $S\Qc$ with $S\Qc(r)=Q(r)\times S_r$ and the $S_r$ acting via the second factor.
  This establishes an equivalence between the
 categories of $NE_1$-operads and $E_1$-operads, with the categories of algebras
 over the corresponding operads being identified as well. 
 
 Let us now consider dg-versions of the topological operads above and use slightly different notation for these versions. 
  Let us call an {\em $\ne_1$-operad} a non-symmetric dg-operad $\Kc$ such that
 each cochain complex $Q(r)$ is situated in degrees $\leq 0$ and quasi-isomorphic to $\k$. 
 Because of Dold-Kan equivalence between $\Co^{\leq 0}(\Vect_\k)$ and $\Delta^\circ\Vect_\k$,
 see
 Example \ref{ex:compl-stack}(b), equipping a complex with a structure of an algebra
 over a $NE_1$-operad is the same as equipping it with the structure of 
 an algebra over an $\ne_1$-operad. 
 
 An example of a $\ne_1$-operad is given by the {\em non-symmetric associative operad}
 $\Ass$ with $\Ass(r)=\k$ for all $r$ and all the compositions being the identities.
 Dg-algebras over $\Ass$ are the same as associative dg-algebras. 
 
  So for the proof of Theorem  \ref{prop:loc-cst}
 we exhibit an $\ne_1$-operad $\Kc$ and equip each  $\Rc(U)$ with the structure  of
 a $\Kc$-algebra in a way compatible with factorization coalgebra structure. 
 The argument is an upgrade of the proof of Theorem \ref{thm:ALG} (associativity of COHA)
 so parts of the treatment will be brief.



For $r\geq 1$ let  $\FILT^{(r)} = \FILT^{(r)}(U)$ be
  the stack $\FILT^{(r)} $ parametrizing flags of objects of $Coh_0(U)$
\[
E_{1}\subset E_{2}\subset\cdots\subset E_{r}.
\]
For $r=0$ we put $\FILT^{(0)}=\pt$. 
The stack  $\FILT^{(r)}$ comes with the projections
$$
\xymatrix{
\FILT^{(r)}\ar[r]^{\rho_r} 
\ar[d]_{q_r}& \Coh_0(U)
\\
\Coh_0(U)^r&
}
$$
\begin{align*}
\rho_r(E_{1}\subset E_{2}\subset\cdots\subset E_{r})&= E_{r},
\\
q_r(E_{1}\subset E_{2}\subset\cdots\subset E_{r})&= (E_{1}, E_{2}/E_{1}, \cdots, E_{r}/E_{,r-1}).
 \end{align*}
For $r=0$ we have $\Coh(U)^{0}=\pt$ and we define  $q_0: \pt\to\pt$ to be the identity map
and $\rho_0: \pt\to \Coh(U)$ to sent $\pt$ to the zero sheaf.

Let $\Ec_i$, $i=1,\cdots,  r$,  be the $i$th tautological sheaf on $\Coh_0(U)^r\times U$ 
and
$p_r: \Coh_0(U)\times U \to\Coh_0(U)$ be the projection. 

\smallskip

Similarly to \S \ref{sec:FILT}, we form the sheaf of associative dg-algebras (and, passing to the super-commutator, of dg-Lie algebras) on $\Coh_0(U)^r$
\[
\Gc_r = \Gc_r(U)\,=\bigoplus_{1\leq i<j\leq r } Rp_{r*} \, \ul\RHom(\Ec_{j}, \Ec_{i})
\]
 and find that $\FILT^{(r)} = \MC(\Gc_r)$ so that $q_r$ is identified with the projection
 of the Maurer-Cartan stack. 
  Therefore we have the diagram
 \be
 \Coh_0(U)^r \buildrel \pi_r\over\lla \Tot(\Gc_r^{\leq 1}) \buildrel i_r\over\lla \FILT^{(r)} \buildrel\rho_r\over\lra\Coh_0(U),
 \ee
in which the map $i_r$ realizes $\FILT^{(r)}$ as the zero locus of the section $s_r$ of $\pi_r^*\Gc_r^{2}$ given by the curvature
map. This gives a virtual pullback $i_r^!$ on Borel-Moore homology. We get, so far 
{\em at the level of BM-homology}, the map
\be\label{eq:mult-COHA-r-fold}
m_r = \rho_{r*}\circ  i_r^!\circ \pi_r^*: R(U)^{\otimes r}\lra R(U), \quad R(U) = H_\bullet^\BM(\Coh_0(U)).
\ee
As in \S \ref{sec:associativity},  we see 
that $m_r$ is  the $r$-fold product in the
(associative) COHA $R(U)$. 

\smallskip

Next, we notice that the family $(\Gc_r)_{r\geq 0}$ of dg-Lie algebras carries a kind of
operadic structure.  For $r_1,\cdots, r_n\geq 0$ consider the summation map
\[
\begin{gathered}
\sigma_{r_1,\cdots, r_n}: \Coh_0(U)^{r_1+\cdots + r_n} \,\, \lra \Coh(U)^n,
\\
(\Fc_1, \cdots, \Fc_{r_1+\cdots+ r_n}) \,\mapsto \, \biggl( \bigoplus_{i=1}^{r_1} \Fc_i, 
\bigoplus_{i=r_1+1}^{r_1+r_2} \Fc_i, \cdots, \bigoplus_{i=r_1+\cdots + r_{n-1}+ 1}^{r_1+\cdots + r_n} \Fc_i
\biggr). 
\end{gathered}
\]
 
 \begin{prop}\label{prop:semidirect}
 We have a semidirect product decomposition, more precisely, an isomorphism
 \[
 \lambda_{r_1,\cdots, r_n}: \bigl( \sigma_{r_1,\cdots, r_n}^*\Gc_n\bigr) \rtimes
 \bigl( \Gc_{r_1} \boxplus \cdots \boxplus \Gc_{r_n}\bigr) \buildrel \simeq \over\lra
 \Gc_{r_1+\cdots + r_n}
 \]
 of sheaves of dg-Lie algebras on $\Coh(U)^{r_1+\cdots + r_n}$. 
 \end{prop}
 
 The proposition means that we have a split exact sequence 
  \be\label{eq:SES-semidirect}
 0\to \sigma_{r_1,\cdots, r_n}^*\Gc_n \buildrel a\over\lra  \Gc_{r_1+\cdots + r_n}
 \buildrel b\over\lra \Gc_{r_1} \boxplus \cdots \boxplus \Gc_{r_n} \to 0
 \ee
  in which 
 $\sigma_{r_1,\cdots, r_n}^*\Gc_n$ is a dg-Lie ideal with
 quotient $\Gc_{r_1} \boxplus \cdots \boxplus \Gc_{r_n}$.
 
 \vskip .2cm
 
 \noindent { \sl Proof of Proposition  \ref{prop:semidirect}:}  By construction,  $\Gc_{r_1,\cdots, r_n}$
 consists of upper-triangular square matrices of size $r_1+\cdots + r_n$. We can decompose
 such a matrix into blocks of sizes $r_i\times r_j$, $1\leq i\leq j\leq n$. Of these, the diagonal blocks
 (of sizes $r_i\times r_i$) are upper-triangular, since the total matrix must be upper-triangular. These
 correspond to the  $\Gc_{r_i}$. So the block-diagonal part is 
 $\Gc_{r_1} \boxplus \cdots \boxplus \Gc_{r_n}$. Similarly, the over-diagonal blocks
 are seen as representing $\sigma_{r_1,\cdots, r_n}^*\Gc_n$. \qed
 
 \begin{prop}
  The isomorphisms $\lambda_{r_1,\cdots, r_n}$ satisfy operadic assocativity. That is, suppose that
 each $r_i$ is decomposed as $r_i=r_{i,1} +\cdots + r_{i, m_i}$. Then the isomorphisms
 \[
 \lambda_{r_{i,1}, \cdots,  r_{i, m_i}}: 
  \bigl( \sigma_{r_{i,1},\cdots, r_{i, m_i}}^*\Gc_{n_i}\bigr) \rtimes
 \bigl( \Gc_{r_{i,1}} \boxplus \cdots \boxplus \Gc_{r_{i,m_i}}\bigr) \buildrel \simeq \over\lra
 \Gc_{r_{i,1}+\cdots + r_{i,m_i}}=\Gc_{r_i}, \quad i=1,\cdots, n
 \]
 together with $\lambda_{r_i,\cdots, r_n}$, compose to $\lambda_{r_{1,1},\cdots, r_{1,m_1},
 r_{2,1}, \cdots, r_{2, m_2}, \cdots, r_{n,1}, \cdots, r_{n, m_n}}$. 
  \end{prop}
  
  \begin{proof} Straightforward verification  in terms of matrices whose blocks are decomposed
  into further blocks. \end{proof}

 Next, we study the compatibilty of the curvature sections $s_r$ on the $\pi_r^*\Gc_r^2$  with the semidirect product decompositions $\lambda_{r_1, \cdots, r_n}$. 
 Let $r=r_1 + \cdots +r_n$. 
 On $\Tot(\Gc_{r }^{\leq 1})$ the sequence \eqref{eq:SES-semidirect} gives
 a short exact sequence of vector bundles
 \be\label{eq:induced-from-semidirect}
 0\to \pi_r^* \sigma_{r_1,\cdots, r_n}^*\Gc_n^2 \buildrel \alpha\over\lra
 \pi_r^* \Gc_r^2 \buildrel \beta\over\lra \pi_r^*\bigl(\Gc_{r_1}^2 \boxplus\cdots\boxplus
 \Gc_{r_n}^2\bigr) \to 0
 \ee
 pulled back from $\Coh_0(U)^r$. 
 We apply to this situation the analysis of \S \ref{subsec:mult-euler},  taking
 $X=\Tot(\Gc_{r }^{\leq 1})$ and viewing \eqref{eq:induced-from-semidirect} as
 an instance of the sequence \eqref{eq:SES-Euler}. The curvature section $s=s_r$ 
 of the middle bundle
 gives rise to the section $s''=s_r''=\beta(s)$ of $\pi_r^*\bigl(\Gc_{r_1}^2 \boxplus\cdots\boxplus
 \Gc_{r_n}^2\bigr)$ with zero locus $X_{s''} = \Tot(\Gc_{r }^{\leq 1})_{s_r''}$
 and the section $s'=s'_r$ of  $\pi_r^* \sigma_{r_1,\cdots, r_n}^*\Gc_n^2$ over $X_{s''}$. 
 To describe them we consider, for each $i=1,\cdots, n$, the stack $\FILT^{(r_i)}$
 parametrizing flags $E_{i,1}\subset E_{i,2}\subset\cdots\subset E_{i, r_i}=E_i$
 of objects of $\Coh_0(U)$. Let
 \[
 \phi: \prod_{i=1}^n\FILT^{(r_i)} \lra \Coh_0(U)^r = \prod_{i=1}^r \Coh_0(U)
 \]
be the projection which sends a tuple of flags as above to the tuple  $(E_1,\cdots, E_n)$
of their maximal elements. We also denote by
\[
\pi_{r_1, \cdots, r_n}: \Tot(\phi^* \Gc_n^{\leq 1}) \lra \prod_{i=1}^n \FILT^{(r_i)}
\]
the projection. 

\begin{prop}
(a) $s''_r$ equal to (the pullback of) the tuple $(s_{r_1}, \cdots, s_{r_n})$
considered as a section of the external direct sum.

\vskip .2cm

(b) $X_{s''}$ is identified with $\Tot(\phi^*\Gc_n^{\leq 1})$. 

\vskip .2cm

(c) Under the identification of (b), the restriction of $\pi_r^*\sigma_{r_1, \cdots, r_n}^* 
\Gc_n^2$ to $X_{s''}$ is identified with $\pi_{r_1,\cdots, r_n}^*\phi^* \Gc_n^2$.

\vskip .2cm

(d) Under the identification of (c), the section $s_r'$ is identified with the pullback of
$s_n$.
\end{prop}

\begin{proof}
(a) As in the proof of Proposition \ref{prop:semidirect}, let us view sections of $\Gc_r$ as  upper-triangular $r\times r$ matrices subdivided
into blocks of sizes $r_i\times r_j$. The projection $b$ (whose pullback is $\beta$)
associates to such a matrix $x$ its block-diagonal part which we denote $x_\Delta$.
Thus $\beta(s_r)$ associates to $x$ the block-diagonal part of the curvature,
i.e., $(dx+ (1/2)[x,x])_\Delta$. 
Since the block-diagonal subspace is a dg-Lie subalgebra, this equals
$d(x_\Delta) + (1/2)[x_\Delta, x_\Delta]$ which corresponds to the pullback
of $(s_{r_1},\cdots, s_{r_n})$. 

\vskip .2cm

(b) Let us represent a point of $\Coh_0(U)^r$, $r=r_1+\cdots + r_n$
 as a sequence of sheaves
 \[
 \Fc_{1,1},\cdots, \Fc_{1, r_1}, \Fc_{2,1}, \cdots, \Fc_{2, r_2}, \cdots
 \Fc_{n,1}, \cdots, \Fc_{n, r_n}. 
 \]
In terms of matrices $x\in\Gc_r^1$, vanishing of the block-diagonal part of the curvature
of $x$ means that the Ext-data for the $\Fc_{i,j}$ provided by $x$, integrate to
$n$ filtrations
\[
E_{i,1} \subset E_{i,2} \subset\cdots\subset E_{i, r_i} = E_i, 
\quad E_{i,j}/E_{i, j-1}\simeq \Fc_{i,j}, \quad i=1,\cdots, n, \,\, j=1,\cdots, r_i,
\]
i.e., we have a point of $\prod_{i=1}^n \FILT^{(r_i)}$. Further, the summation
map $\sigma_{r_1,\cdots, r_n}$ on the sequence of the $\Fc_{I,j}$ corresponds
to the projection $\phi$. The over-diagonal blocks of $x$ assemble into a section
of $\phi^* \Gc_n^{\leq 1}$, whence the statement. 

\vskip .2cm

Part (c) is clear from the above.
To see (d), notice that the pullback of $s_n$ represents the over-diagonal 
blocks of the curvature of $x$. 
\end{proof}

We now apply the formalism of homotopy canonical Euler classes from Appendix \ref{sec:hom-can-Euler}.
Let $\Kc_r(U)=K_{\pi_r^*\Gc_r^2}$ be the parameter complex for the homotopy canonical orientation
class of the bundle $\pi_r^*\Gc_r^2$ on $\Tot(\Gc_r^{\leq 1})$, see \S \ref {subsec:hom-can-Euler}. 
Here $\Gc_r=\Gc_r(U)$, as above.  By construction, each $\Kc_r(U)$ maps quasi-isomorphically to $\k$. 
The semidirect product decomposition of Proposition 
\ref{prop:semidirect} and the pairings \eqref{eq:pairings-of-K} of the $K$-complexes give morphisms of
complexes
\[
\Kc_n(U)\otimes \Kc_{r_1}(U)\otimes\cdots\otimes\Kc_{r_n}(U) \lra \Kc_{r_1+\cdots r_n}(U). 
\]
The operadic associativity of the isomoprhisms $\lambda_{r_1,\cdots, r_n}$ and the associativity of the
pairings  \eqref{eq:pairings-of-K}  imply that $\Kc(U)=(\Kc_r(U))_{r\geq 0}$ is a (non-symmetric)
dg-operad. Since each $\Kc_r(U)$ is quasi-isomorphic to $\k$, we see that $\Kc(U)$ is
an $\ne_1$-operad. Further, the correspondence $U\mapsto \Kc(U)$ forms a presheaf
(in fact, a sheaf up to homotopy, by the above) of dg-operads on the analytic surface $\Sigma$. 
Let $\Kc=\Kc(\Sigma)$ be the operad of global sections. 

\vskip .2cm
 Finally, let us upgrade the $r$-fold multiplication map \eqref{eq:mult-COHA-r-fold} to the cochain level
 by analyzing the ambiguity. This maps involves the virtual pullback $i_r^!$ which is defined
 in terms of the refined Chern class $c_d(\pi_r^*\Gc_r^2)$, $d=\rk (\Gc_r^2)$. Using the homotopy
 canonical cochain lifting $\wt c_d(\pi_r^*\Gc_r^2)$, $d=\rk (\Gc_r^2)$, see \eqref{eq:wt-c_r},
 we define a cochain level multiplication
 \be\label{eq:mu-r}
\wt m_{r,U}: \Kc_r(U)\otimes  \Rc(U)^{\otimes r} = \Kc_r(U) \otimes C_\bullet^\BM(\Coh_0(U))^{\otimes r}  \lra C_\bullet^\BM(\Coh_0(U))=\Rc(U). 
\ee
The multiplicativity of the $\wt c_d$ in short exact sequences \eqref{eq:c_r-mult-hom-can} implies
that the $\mu_{r,U}$ make $\Rc(U)$ into an algebra over the $\ne_1$-operad $\Kc(U)$ and therefore
over $\Kc=\Kc(\Sigma)$. Further, these $\Kc$-algebra structures are clearly compatible with
the factorization coalgebra structure on the presheaf  $\Rc=(\Rc(U))$. This finishes the proof of Theorem
\ref{prop:loc-cst}. \qed


\subsection{Proof of Theorem \ref{thm:PBW}} As before, let $\Sigma=S^\an$. 
 For any open subset $U\subset \Sigma$   (in the complex analytic topology)
 we have the $\ZZ^2$-graded space $\Theta(U)$ and the symmetrized product map
$\sigma_U: \Sym(\Theta(U))\to R(U)$. Because of  Proposition \ref{prop:coh-alg=anal} and the identification \eqref{eq:BM=alg=anal}, the map $\sigma$ of Theorem \ref{thm:PBW}
is identified with the global map $\sigma_\Sigma$, corresponding to $U=\Sigma$. 
Now, if $U$ is a disk, then $\sigma_U$ is an isomorphism by Theorem \ref{thm:flat}.  
We will deduce the global statement  (for $U=\Sigma$) from these local ones.

\smallskip

For this, we upgrade the correspondence $U\mapsto \Theta(U)$ to a complex of 
sheaves $\Vc$ on $\Sigma$
so that $\Theta(U) = \HH^{-\bullet}(U, \Vc)$ is the hypercohomology of $U$ with coefficients in $\Vc$.
That is,  we define
\[
\Vc = \omega_\Sigma\otimes_\k \Theta', 
\]
the tensor product of the dualizing complex $\omega_\Sigma$ and the graded
vector space $\Theta'=\Theta[0,-4]$, see \eqref{eq:Theta'}.
Recall that $\Theta$ and therefore $\Theta'$ is spanned by the basis vectors $t^n q^{i-1}$, and
such a vector is 
  identified with $\theta_{n,i}= ch_i(\Ec_n)\cap\theta_n \in
R^{n, 2-2i}(\CC^2)$, see \eqref{eq:theta-n-i}.  Here, as we recall,  $\Ec_n$ 
is the tautological rank $n$ bundle on $\Coh^{(n)}_0(\CC^2)$
whose fiber at a point represented by a coherent sheaf $\Fc$ is $H^0(\Fc)$
and $\theta_n$ is the fundamental class of $\Coh_\Ipt^{(n)}(\CC^2)$.

As before, we denote by the same symbol $\Ec_n$ the analogous tautological bundle on $\Coh_0^{(n)}(\Sigma)$
and, if necessary,  its restriction to $\Coh_\Ipt^{(n)}(\Sigma)$. 

Extending the construction of \eqref{eq:p-n-dagger}, we choose a cocycle representing the fundamental class
of  $\Coh_\Ipt^{(n)}(\Sigma)$ in $H_2^\BM$ and define the morphism
\[
p_n^\dagger: p_n^*\omega_\Sigma \lra p_n^!\omega_\Sigma [2] = \omega_{\Coh_\Ipt^{(n)}(\Sigma)}[2]
\]
as the cup-product with this cocycle. Here $p_n: \Coh^{(n)}_{\Ipt}(S)\to S$ is the projection defined in \eqref{eq:1-pointed}. 
Further, for each $i$ we fix a cocycle representative $\wt {ch}_i(\Ec_n)$ of $ch_i(\Ec_n)\in H^{2i}(\Coh_\Ipt^{(n)}(\Sigma), \k)$.

 The sheaf $\Vc$ and the factorization coalgebra $\Rc$ are both presheaves with values
in the category of cochain complexes.
We define a morphism of presheaves $\wt\alpha: \Vc\lra\Rc$ as the composition of the two morphisms:
first, the morphism
\[
 R\Gamma(U, \omega_\Sigma)\otimes  t^n q^{i-1}\lra
  R\Gamma(\Coh_\Ipt^{(n)}(U), \omega_{\Coh_\Ipt^{(n)}(U)}), \quad \gamma\otimes t^n q^{i-1} \mapsto
  \wt{ch}_i(\Ec_n)\cap  p_n^\dagger (p_n^*(\gamma)), 
\]
(here $p_n^*(\gamma)$ is  an element of $R\Gamma(\Coh_\Ipt^{(n)}(U)),  p_n^*\omega_U)$)
and, second, the direct image morphism  
\[
\eps_*:  R\Gamma(\Coh_\Ipt^{(n)}(U), \omega_{\Coh_\Ipt^{(n)}(U)}) \
\lra R\Gamma(\Coh^{(n)}_0(U), \omega_{\Coh^{(n)}_0(U)}) \, = \Rc(U)^{(n)},
\]
where $\eps: \Coh_\Ipt^{(n)}(U) \to \Coh_0^{(n)}(U)$ is the (closed) embedding.

 Since $\Vc$ is a sheaf with values in the  category of cochain complexes, its symmetric algebra
$\Sym(\Vc)$ is a factorization coalgebra with values in this category, by  Proposition \ref{prop:sym-sh-fa}. 
Since $\Rc$ is a factorization algebra in the category of $E_1$-algebras, 
we can define the  symmetrized product 
$
\wt\sigma: \Sym(\Vc) \to \Rc
 $
 by setting $\wt\sigma = \sum_{n\geq 0} \wt\sigma_n,$ where
 \be\label{eq:sym-prod2}
 \wt\sigma_n: \Sym^n(\Vc) \lra \Rc, 
\quad
\wt\sigma_n(v_1\bullet\cdots\bullet v_n) \,={1\over n!} \sum_{s\in S_n} \mu_n (\wt\alpha(v_{s(1)}) \otimes \cdots \otimes 
\wt\alpha(v_{s(n)})),
\ee
lifting the map $\sigma$ from \eqref{eq:sym-prod1}. 
In other words, $\wt\sigma_n$ is the symmetrization of the map
\[
\mu_n\circ(\wt\alpha \otimes\cdots\otimes\wt\alpha): \Vc^{\otimes n} \lra\Rc. 
\]

The map $\wt\sigma$ is a morphism of factorization coalgebras 
{\em in the category of $\ZZ^2$-graded cochain complexes}. Note that we do not claim
(and it is not true) that $\wt\sigma$ is a morphism of factorization coalgebras
in the category of $E_1$-algebras.

By the above,
$\wt\sigma_U$ is a weak equivalence (of $\ZZ^2$-graded cochain complexes)
 for any $U$ which is, topologically, a disk.
  Therefore $\wt\sigma$ is a weak equivalence (of factorization coalgebras
 in the category of $\ZZ^2$-graded cochain complexes)
  by Proposition \ref{prop:mor-fa-local}.
Taking $U=\Sigma$
we obtain Theorem \ref{thm:PBW}.


 \subsection{$E_4$-structure on the flat COHA}

By  \cite{ginot}, \cite{lurie-HA}, locally constant factorization (co)algebras on $\RR^m$ 
with values in  $\Co(\Vect_\k)$ can be identified
with $E_m$-(co)-algebras in $\Co(\Vect_\k)$, the identification associating to a (co)algebra $\Bc$ the object $\Bc(B)$ where $B\subset\RR^m$ is the standard unit $m$-ball. 
Note that $\Bc(B)$ is weak equivalent to $\Bc(\RR^d)$. 

\smallskip

Let us specialize this to the case when $\Bc=\Rc$ and  $m=4$, since  $\CC^2\simeq\RR^4$. 
In this case we form the cochain complex
$\Rc(B) \,\simeq \, \Rc(\CC^2)$ 
whose cohomology is the flat Hecke algebra
$R(B)\simeq R(\CC^2)$ studied in \S \ref{sec:local-Hec}.  The general results above,
applied to the category $\Cc$ of $E_1$-algebras, imply:

\begin{cor}
$\Rc(\CC^2)$ 
is $E_1$-algebra in the category of $E_4$-coalgebras. \qed
\end{cor}

 \begin{rems} 
 \hfill
\begin{itemize}[leftmargin=8mm]
\item[$\mathrm{(a)}$]
The $E_4$-coalgebra structure on $\Rc(\CC^2)$
 is a cochain level refinement of the comultiplication $\Delta$ on $R(\CC^2)$,
 see \S \ref{subsec:R-bialg}. While $\Delta$ is cocommutative,
 because it is independent on the choice of two distinct disks $U_1, U_2\subset \CC^2$,
 at the cochain level  we do not seem to have cocommutativity 
 since the space of choices of such pairs of disks is not
 contractible (it is precisely the space of binary operations in the operad $D_4$).

\item[$\mathrm{(b)}$] By forming the Koszul dual to   the $E_1$-algebra   structure on $\Rc(\CC^2)$, 
we obtain an $E_1$-coalgebra in the category of
 $E_4$-coalgebras, i.e., an
 $E_5$-coalgebra. Alternatively,
 forming the Koszul dual to the $E_4$-algebra structure, we obtain an $E_5$-algebra. This suggest that some
 $5$-dimensional field theory may be relevant to this picture. 
 
 \end{itemize}
\end{rems}


\medskip

\appendix


\section{Basics on $\infty$-categories  and dg-categories}\label{sec:appendix}

 \subsection{$\oo$-categories}
 
 Let $\k$ be a field  of characteristic $0$. 
 By $\Vect = \Vect_\k$ we denote the category of  $\k$-vector spaces
 and by $\Co(\Vect) = \Co(\Vect_\k)$ the category of complexes of $\k$-vector
 spaces bounded below, with morphisms being morphisms of complexes.
By $\sset$ we denote the category of simplicial sets. 
For a simplicial set $Y$ we denote by $|Y|$ the geometric realization of $Y$.
We say that $Y$ is {\em contractible}, if $|Y|$ is a contractible topological space. 
For a topological space $T$ we denote by $\Sing(T)$ the singular simplicilal set of $T$.

\smallskip
 
 An $\infty$-{\em category}  $\frakC$ is a simplicial set
$(\frakC_n)_{n\geq 0}$ satisfying the partial Kan condition, with elements of $\frakC_0$ called objects and elements
of $\frakC_1$ called morphisms.

Every $\oo$-category $\frakC$ gives rise to an ordinary category $\h\frakC$
known as the {\em homotopy category of} $\frakC$. It has the same objects
as $\frakC$ and its morphisms are certain equivalence classes of morphisms in
$\frakC$. 
Further, $\frakC$ contains the maximal Kan simplicial subset
$\frakC^\Kan$ with $\frakC_0^\Kan=\frakC_0$, having the meaning of the subgroupoid of (weakly) invertible morphisms. 
We refer to \cite{lurie-HTT} for more details. 

\smallskip

A {\em simplicial category} is a category $\Cc$ enriched in $\sset$,
so that for any two objects $\Fc, \Gc\in\Cc$ we are given a simplicial set $\Map_\Cc(\Fc, \Gc)$ with standard properties.
A simplicial category $\Cc$ gives an $\oo$-category $\frakN(\Cc)$ with  the same objects,
 as explained in \cite{lurie-HTT}. 
 
 \smallskip

A {\em dg-category} is a category  $\Cc$  enriched in $\Co(\Vect)$, so that  for any two objects 
$\Fc, \Gc\in\Cc$ we are given 
a cochain complex $\Hom^\bullet_\Cc(\Fc, \Gc)$ with standard properties.
Any dg-category $\Cc$ gives rise to a $\k$-linear category $H^0\Cc$ with the
same objects as $\Cc$ and
\be
\Hom_{H^0\Cc} (\Fc, \Gc) \,=\, H^0 \Hom^\bullet_\Cc(\Fc, \Gc). 
\ee
Further, $\Cc$ can be made into a simplicial category (with the same objects) by
\begin{align}\label{eq:dg-map}
\Map(\Fc, \Gc) \,=\, \DK\bigl( \tau_{\leq 0} \Hom^\bullet_\Cc(\Fc,\Gc)\bigr)
\end{align}
where $\DK$ is the Dold-Kan simplicial set associated to a $\ZZ_{\leq 0}$-graded complex, 
 see  \cite[\S 8.4.1]{weibel} and a discussion in
Example \ref{ex:compl-stack}.  
So it gives rise to an $\oo$-category denoted $N^\dg(\Cc)$, see \cite{lurie-HA}.

\smallskip

\subsection{Enhanced derived categories}\label{subsec:enh-der}

Let $\Ac$ be a $\k$-linear abelian category. We denote by $\Co(\Ac)$ the  category of  complexes over $\Ac$   
with morphisms being morphisms of complexes.   By $\Co(\Ac)_\dg$ we denote the dg-category with the same objects
as $\Co(\Ac)$. For any two objects $\Fc, \Gc$ of  $\Co(\Ac)_\dg$,  the complex $\Hom_{\Co(\Ac)_\dg}(\Fc,\Gc)$ 
is the graded $\k$-vector space $\Hom_\Ac(\Fc,\Gc)$
with the differential given by the commutation with $d_\Fc$ and $d_\Gc$. Thus
$\Co(\Ac)=H^0\Co(\Ac)_\dg$. 
 By $\D(\Ac)= \Co(\Ac)[\Qis^{-1}]$ we denote the  derived category of $\Ac$, i.e., 
 the  localization of $\Co(\Ac)$ by the class $\Qis$ of
 quasi-isomorphisms. 
   There are three closely related \emph{enhancements} of $\D(\Ac)$ with the same objects:

 \begin{itemize}[leftmargin=8mm]
 \item[(a)] The {\em derived dg-category} $\D(\Ac)_\dg$ with the property that
 $\D(\Ac)=H^0 \D(\Ac)_\dg$. 
 If $\Ac$ has canonical injective resolutions $A\mapsto I(A)$, then we define,
  see \cite{bondal-kapranov},
 \[
 \Hom_{\D(\Ac)_\dg}(\Fc, \Gc) = \Hom^\bullet_{\Co(\Ac)_\dg}(I(\Fc), I(\Gc)).
 \]
 The complex in the RHS is also denoted $ \RHom^\bullet(\Fc, \Gc)$.

 \item[(b)] The {\em simplicial derived category}  $\D(\Ac)_\Delta$
  with the property that 
  $\Hom_{\D(\Ac)}(\Fc, \Gc) = \pi_0 \Map_{\D(\Ac)_\Delta}(\Fc, \Gc)$.
 There are two homotopy equivalent ways of constructing 
 $ \Map_{\D(\Ac)_\Delta}(\Fc, \Gc)$: 
 \begin{itemize}[leftmargin=8mm]
 \item[(b1)]  Given the
 data in (a), we can define, as in \eqref{eq:dg-map},
 \[
 \Map_{\D(\Ac)_\Delta}(\Fc, \Gc) \,=\, \DK\bigl(\tau_{\leq 0} \RHom^\bullet(\Fc, \Gc)\bigr).
 \]
 \item[(b2)] 
 The Dwyer-Kan simplicial localization procedure \cite{DK1},\cite{ DK2}  produces
 simplicial sets $\Map(\Fc, \Gc)$,  starting 
 from the category $\Co(\Ac)$ and the class  of  morphisms $\Qis$.  
 We can take  $ \Map_{\D(\Ac)_\Delta}(\Fc, \Gc)$ to be the simplicial sets $\Map(\Fc, \Gc)$. 
 Further, we can use them to 
  get  an intrinsic definition of the
 $\RHom^\bullet(\Fc,\Gc)$ in (a) by taking the normalized chain complexes and stabilizing with respect to the
 shift.  This allows one to define $\D(\Ac)_\dg$  even without the  use of canonical
 injective resolutions. 
 \end{itemize} 
 
 \item[(c)] The {\em derived $\oo$-category}  $\D(\Ac)_\infty$
 with the property that $\h \D(\Ac)_\infty=\D(\Ac)$. 
  As in (b2), it can be defined intrinsically, as the $\infty$-categorical localization
 of $\Co(\Ac)$ by $\Qis$, see \cite{lurie-HA}. 
 \end{itemize}


\medskip

\section{Homotopy canonical Euler classes}\label{sec:hom-can-Euler}

The concept of \emph{coherent homotopy uniqueness} of objects, morphisms, cohomology classes, etc.,
is implicit in the formalism of $\oo$-categories, as well
as in homotopical algebra in general. In this appendix we spell out some instances of this concept
in the dg-context.  

\subsection{Cocycles defined up to a contractible choice}

Let $V$ be a cochain complex over $\k$, and $a\in H^d(V)$. Viewing $a$ as a 
morphism $a: \k\to V[d]$ in $\D(\Vect_\k)$, we can represent $a$ (non-uniquely)  
by a diagram of morphisms of complexes
$\k \buildrel q\over\leftarrow K \buildrel \alpha\over\to V[d]$,
where $q$ is a quasi-isomorphism. Such a diagram is just a right fraction representing the morphism
$a$ in  $\D(\Vect_\k)=\Co(\Vect_\k)[\Qis^{-1}]$ as $a=\alpha q^{-1}$.
We will refer to any such diagram as a {\em $d$-cocycle defined up to a contractible choice} and say that it {\em represents $a$ up to a contractible choice}.

\begin{exas}\label{exas:h<0=0}
(a) Suppose that $a\neq 0$ and $H^j(V)=0$ for $j<d$. Let 
$Z^d(V)\subset V^d$ be the space of $d$-cocycles and $\gamma: Z^d(V)\to H^d(V)$ be the projection. Let
\[
V^{\leq d}_a \,=\,\bigl\{ \cdots \to V^{d-2}\to V^{d-1}\to \gamma^{-1}(\k a)\bigr\}\,
\subset \, V^{\leq d}.
\]
Then the projection to $\k a \simeq \k$ gives a
quasi-isomorphism $V^{\leq d}_a [d] \buildrel \sim\over\to \k$
and the diagram
\[
\k \buildrel\sim\over\longleftarrow V^{\leq d}_a[d] \hra V[d]
\]
represents $a$ up to a contractible choice. 

\vskip .2cm

(b) In particular, let $\Cc$ be a dg-category and $x,y\in\Ob(\Cc)$ be such that
$\Hom^\bullet_\Cc (x,y)$ has $H^j=0$ for $j<0$. Then any nonzero
morphism $f: x\to y$ in $H^0\Cc$ is represented, up to a contractible choice, by the
diagram
\[
\k \buildrel \sim\over\longleftarrow \Hom^{\leq 0}_\Cc (x,y)_f \hra \Hom^\bullet_\Cc
(x,y). 
\]

\end{exas}

\subsection{Homotopy canonical orientation classes}\label{subsec:hom-can-or}
 Let $X\buildrel i\over\hra Y$
be a regular embedding of stacks of codimension $d$. We have then the
canonical {\em orientation isomorphism}
$\eta_{X/Y}: \ul\k_X\buildrel \sim\over\to i^! \ul\k_Y[2d]$
in the derived category $\D(X)$.  If $X\buildrel i\over\hra Y \buildrel j\over\hra Z$
are two composable regular embeddings, with $i$ of codimension $d$ and
$j$ of codimension $e$, then $ji$ is a regular embedding of codimension $d+e$
and $\eta_{X/Z}: \ul \k_X \to (ji)^! \ul \k_Z[2(d+e)]$ is equal to the composition
\be\label{eq:composition-eta}
\ul \k_X \buildrel \eta_{X/Y}\over\lra i^! \ul \k_Y[2d] \buildrel i^! \eta_{Y/Z}[2d]
\over\lra i^! j^! \ul \k_Z[2d+2e].
\ee
Passing to the dg-enhancements, we notice that $\eta_{X/Y}$ connects two
objects which are quasi-isomorphic to single sheaves in degree $0$
and so negative Ext's between these objects vanish. We are therefore
in the situation of Example \ref{exas:h<0=0}(b) and so the diagram
\be
\k \buildrel \sim\over\leftarrow K_{X/Y}:= \Hom^{\leq 0}_{\D(X)_\dg}
(\ul \k_X, i^! \ul\k_Y[2d])_{\eta_{X/Y}} \,\hra \,
\Hom^{\bullet}_{\D(X)_\dg}
(\ul \k_X, i^! \ul\k_Y[2d])
\ee
represents $\eta_{X/Y}$ up to a contractible choice. We can write it
as a canonical closed morphism in $\D(X)_\dg$ of degree $0$
\be
\wt\eta_{X/Y}: K_{X/Y} \otimes \ul\k_X \lra i^! \ul \k_Y[2d]. 
\ee
If $X\buildrel i\over\hra Y \buildrel j\over\hra Z$
are two composable regular embeddings as before, then the composition
of Hom-complexes in the dg-category $\D(X)_\dg$ induces a composition
\be\label{eq:pairing-m-XYZ}
m_{X,Y,Z}: K_{Y/Z}\otimes K_{X/Y} \lra K_{X/Z}
\ee
and such compositions are associative for any composable triple of
regular embeddings. The composition $m_{X,Y,Z}$ fits into the
commutative square
\be\label{eq:composition-of-K}
\xymatrix{
K_{Y/Z}\otimes K_{X/Y}\otimes \ul\k_X 
\ar[d]_{m_{X,Y,Z}\otimes \ul\k_X} 
\ar[rr]^{K_{Y/Z}\otimes \wt\eta_{X/Y} }&&
K_{Y/Z}\otimes i^! \ul\k_Y[2d]
\ar[d]^{i^! \wt\eta_{Y/Z}}
\\
K_{X/Z}\otimes \ul\k_X \ar[rr]_{\wt\eta_{X/Z}} && i^! j^! \ul\k_Z[2d+2e]
}
\ee
which underlies the identification of $\eta_{X/Z}$ with the composition
\ref{eq:composition-eta}.

 \subsection{Homotopy canonical Euler classes  } \label{subsec:hom-can-Euler}
 Let $\Ec$ be a rank $d$ vector
 bundle over a stack $X$. Let $s\in H^0(X,\Ec)$ be a section.
 We consider it as a morphism $s: X\to\Tot(\Ec)$. 
  Let $i_s: X_s\to X$
 be the embedding of the zero locus of $s$. We have then a Cartesian square
 of closed embeddings
 \be\label{eq:square:X-s}
 \xymatrix{
 X_s \ar[d]_{i_s}
 \ar[r]^{i_s} & X\ar[d]^0
 \\
 X \ar[r]_s & \Tot(\Ec). 
 }
 \ee
 The zero section embedding $0: X\hra\Tot(\Ec)$ is regular of codimension $d$, so
 we have the orientation isomorphism in $\D(X)$
 \be
 \eta_\Ec:= \eta_{X/\Tot(\Ec)}: \ul\k_X \lra 0^! \ul\k_{\Tot(\Ec)}[2d].
 \ee
 Applying $i_s^{-1}$ to $\eta_\Ec$, we get a morphism in $\D(X_s)$
 \be\label{eq:euler-orient}
 \ul\k_{X_s} = i_s^{-1}\ul\k_X \buildrel i_s^{-1} \eta_\Ec\over\lra
 i_s^{-1} 0^! \ul\k_{\Tot(\Ec)}[2d] \buildrel\text{B.C.}\over\lra
  i_s^! s^{-1}\ul\k_{\Tot(\Ec)}[2d] = i_s^! \ul\k_X[2d], 
 \ee
 where ``B.C.'' means the base change morphism for the square
 \eqref{eq:square:X-s}, see \cite{kashiwara-schapira} Prop. III.1.9(iii).
  The morphism \eqref{eq:euler-orient} can be
 viewed as an element $c_d(E,s)\in H^{2d}_{X_s}(X, \k)$ which is known
 as the {\em refined Euler (top Chern)  class of} $(\Ec, s)$. Its image
 in $H^{2d}(X,\k)$ is the usual Euler (top Chern) class $c_d(\Ec)$. 
 
 Passing to dg-enhancements, we denote $K_\Ec:= K^\bullet_{X/\Tot(\Ec)}$.
 We can think of objects of the dg-categories  $\D(Y)_\dg$ associated to various stacks $Y$
 as (systems of, see \eqref{eq: D-holim})  complexes consisting of flabby sheaves.
 Now, for a flabby sheaf the $!$-inverse image under a closed embedding
 is given by the sheaf of sections with support. With this understanding, 
 the base change morphism in a Cartesian square of closed embeddings of topological
 spaces is a canonical morphism of sheaves. 
 Therefore our conventions imply that the base change morphism in 
 \eqref{eq:euler-orient} is defined canonically  (no choice needed).  So lifting $\eta_\Ec$ 
 to $\wt\eta_\Ec:=\wt\eta_{X/\Tot(\Ec)}$ as defined in \S\ref{subsec:hom-can-or},
  we upgrade the 
  composite
 morphism \eqref{eq:euler-orient}  to a  closed degree $0$ morphism in $\D(X)_\dg$
 \be\label{eq:wt-c_r}
 \wt c_d(\Ec,s): K_\Ec \otimes \ul\k_{X_s} \lra i_s^! \ul\k_{\Tot(\Ec)}[2d],
 \ee
 representing $c_d(\Ec,s)$ up to a contractible choice.
 
 \subsection{Multiplicativity of homotopy canonical Euler classes}\label{subsec:mult-euler}
 Let
 \be\label{eq:SES-Euler}
 0\to\Ec'\buildrel a\over\lra \Ec\buildrel b\over\lra \Ec''\to 0
 \ee
 be a short exact sequence of vector bundles on a stack $X$, of ranks $d', d, d''$ respectively, so
 $d=d'+d''$. We explain how to upgrade the multiplicativity relation $c_d(\Ec)=c_{d'}(\Ec') c_{d''}(\Ec'')$
 in $H^\bullet(\X,\k)$ to the level of homotopy canonical refined classes. 
 
 Let $s\in H^0(X,\Ec)$ be a section. Then $s'':=b(s)$ is a section of $\Ec''$. Its zero locus 
 $ i_{s''}: X_{s''} \buildrel\over  \hra X$ can be described, informally, as the locus of points
 $x\to X$ such that $s(x)\in\Ec'$. In particular, the bundle $i_{s''}^*\Ec'$ on $X_{s''}$ carries a section
 $s'$ given by the restriction of $s$. The zero locus $(X_{s''})_{s'}$ of this latter section is nothing but $X_s$,
 so we have a commutative triangle of closed embeddings
 \be\label{eq:i-s-s'-s''}
 \xymatrix{
 X_s=(X_{s''})_{s'}
 \ar[rd]_{i_s}  \ar[r]^{\hskip 0.7cm i_{s'}} & X_{s''}\ar[d]^{i_{s''}}
 \\
 & X. 
 }
 \ee
 The multiplicativity of refined Euler classes at the cohomology level can be expressed as the commutativity of the
 triangle in $\D(X_s)$
 \be\label{eq:c_r-mult-homology}
 \xymatrix{
 \ul\k_{X_s} \ar[rr]^{\hskip -0.5cm c_{r'}(i_{s''}^*\Ec', s')} 
 \ar[drr]_{c_d (\Ec, s)}
 &&i_{s'}^! \ul\k_{X_{s''}} [2 d']
 \ar[d]^{i_{s'}^! c_{r''}(\Ec'', s'')[2d']}
 \\
 && i_{s'}^! i_{s''}^! \ul\k_X[2d'+2d'']. 
 } 
 \ee
 To prove this commutativity and to lift it to the homotopy canonical level, we denote by
 \be
 \xymatrix{
 \Tot(\Ec') \ar[dr]_{\pi'} \ar[r]^a & \Tot(\Ec)\ar[d]^\pi  \ar[r]^b&\Tot(\Ec'')\ar[dl]^{\pi''}
 \\
 & X & 
 }
 \ee
 the diagram of the total spaces induced by \eqref{eq:SES-Euler}. We note that
 \be\label{eq:Tot-E'-E-E''-Cart}
 \xymatrix{
 \Tot(\Ec') \ar[r]^a \ar[d]_{\pi'}
  & \Tot(\Ec) \ar[d]^b
 \\
 X \ar[r]_{0_\Ec''} & \Tot(\Ec'') 
 }
 \ee
 is a Cartesian square. Therefore the same base change argument as used in \eqref{eq:euler-orient}
 gives a morphism of complexes
 \[
 \Hom_{\D(X)_\dg}(\ul\k_X, 0_{\Ec''}^! \ul\k_{Tot(\Ec'')}[2d'']) \lra \Hom^\bullet_{\D(\Tot(\Ec')_\dg}
 (\ul\k_{\Tot(\Ec')}, a^! \ul\k_{\Tot(\Ec)}[2d'']. 
 \]
 This morphism induces a morphism
 \be\label{eq:K_E''-->K-rel}
 K_{\Ec''}=K_{X/\Tot(\Ec'')} \lra K_{\Tot(\Ec')/\Tot(\Ec)}. 
 \ee
 Also,
 \be\label{eq:X-Tot(E')-Tot(E)}
 X\buildrel 0_{\Ec'}\over\lra \Tot(\Ec') \buildrel a\over\lra \Tot(\Ec)
 \ee
 is a composable pair of regular embeddings with composition $0_\Ec$. Therefore composing the pairing
 \eqref{eq:pairing-m-XYZ} of this composable pair with the morphism \eqref{eq:K_E''-->K-rel}, we get a
 pairing
 \be\label{eq:pairings-of-K}
 m_{\Ec', \Ec, \Ec''} : K_{\Ec''} \otimes K_{\Ec'} \lra K_\Ec. 
 \ee
 These pairings are associative for any admissible (locally split)  filtration $\Ec_1\subset\Ec_2\subset\Ec$ of vector bundles. 
 
 Further,  \eqref{eq:i-s-s'-s''} and \eqref{eq:X-Tot(E')-Tot(E)} combine into a diagram 
  \be
 \xymatrix{
 X\ar[r]^{0_{\Ec'}} & \Tot(\Ec') \ar[r]^a & \Tot(\Ec)
 \\
 X_s \ar[u]^{i_s} \ar[r]_{i_{s'}}& X_{s''}\ar[u]_{s'}  \ar[r]_{i_{s''}} & X\ar[u]_{s}
 }
 \ee
consisting of two
 Cartesian squares, whose concatenation (i.e., the outer perimeter diagram with horizontal edges composed)
 is  the Cartesian square \eqref{eq:square:X-s}.  We now notice that:
 \begin{itemize}
 \item The right square recovers $\wt c_{d''}(\Ec'', s'')$ by pullback, as in 
 \eqref{eq:euler-orient},  from $\wt \eta_{\Tot(\Ec')/\Tot(\Ec)}$. This follows from the square  
  \eqref{eq:Tot-E'-E-E''-Cart} which shows that $\wt \eta_{\Tot(\Ec')/\Tot(\Ec)}$ 
  is the image of $\wt\eta_\Ec$ under \eqref{eq:K_E''-->K-rel}.
  
  \item The left square recovers $\wt c_{d'} (i_{s''}^*\Ec', s')$ by pullback from $\wt\eta_{\Ec'}=\wt\eta_{X/Tot(\Ec')}$.
  This is because we can subdivide the square into two Cartesian squares
  \[
  \xymatrix{
  X\ar[rr]^{0_{\Ec'}} && \Tot(\Ec')
  \\
  X_{s''} \ar[u]^{i_{s''}} \ar[rr]^{0_{i_{s''}^*\Ec'}} && \Tot(i_{s''}^*\Ec')\ar[u]
  \\
  X_s \ar[u]^{i_{s'}} \ar[rr]_{i_{s'}} && X_{s''}\ar[u]_{s'}
  }
  \]
 \noindent  which show that $\wt\eta_{i_{s''}\Ec'}$ is the pullback of $\wt \eta_{\Ec'}$
  
 \item The composite (outer)  square   \eqref{eq:square:X-s}
 recovers $\wt c_d(\Ec, s)$ by pullback from $\wt\eta_\Ec$ by definition. 
   \end{itemize}
   
   So applying \eqref{eq:composition-of-K}, we obtain a commutative square
   \be\label{eq:c_r-mult-hom-can}
 \xymatrix{
 K_{\Ec''}\otimes K_{\Ec'} \otimes \ul\k_{X_s}
 \ar[d]_{m_{\Ec', \Ec, \Ec''}\otimes \ul\k_{X_s}} \ar[rr]^{ K_{\Ec''} \otimes \wt c_{d'}(i_{s''}^*\Ec', s')} 
 &&i_{s'}^! \ul\k_{X_{s''}} [2d']
 \ar[d]^{i_{s'}^! \wt c_{d''}(\Ec'', s'')[2d']}
 \\
 K_\Ec\otimes\ul\k_{X_s} \ar[rr]_{\wt c_d(\Ec, s)}&& i_{s'}^! i_{s''}^! \ul\k_X[2d'+2d'']. 
 } 
 \ee
  which is a lift of  \eqref{eq:c_r-mult-homology} to the homotopy canonical level. 
 

\vskip 1.5cm

\small{

M.K.: Kavli IPMU, 5-1-5 Kashiwanoha, Kashiwa, Chiba, 277-8583 Japan.
Email:  \hfil\break
{\tt mikhail.kapranov@protonmail.com}

\smallskip

E.V.: Institut de Math\'ematiques de Jussieu - Universit\'e Paris Diderot, B\^ atiment Sophie Germain, 8 place Aur\` elie Nemours, 75013 Paris, France.  Email: {\tt     eric.vasserot@imj-prg.fr   }

}

\end{document}